\documentclass[a4paper]{amsart}
\usepackage[colorlinks=true,linkcolor=blue]{hyperref}
\usepackage{amssymb}
\usepackage{amsfonts}
\usepackage{amsthm}
\usepackage{amsmath}
\usepackage{mathtext}
\usepackage{latexsym}
\usepackage{caption} 
\usepackage[dvipdf]{graphicx}
\numberwithin{equation}{section}
\newcommand{\point}{\par\noindent\vspace{1mm}$\bullet$ \ \ }
\newcommand{\dsize}{\textstyle}
\newcommand{\R}{{\mathbb R}}
\newcommand{\Z}{{\mathbb Z}}
\newcommand{\N}{{\mathbb N}}
\newcommand{\C}{{\mathbb C}}

\newcommand{\re}{{\rm Re}\,}
\newcommand{\im}{{\rm Im}\,}
\newcommand{\res}{{\rm res}\, }

\newcommand{\ctg}{{\mathrm cot}\,}

\newcommand{\tr}{{\rm tr}\,}

\theoremstyle{plain}
\newtheorem{theorem}{Theorem}[section]
\newtheorem{lemma}{Lemma}[section]
\newtheorem{Pro}{Proposition}[section]
\newtheorem{Cor}{Corollary}[section]
\newtheorem{hypothesis}{Hypothesis}
\theoremstyle{definition}
\newtheorem{Rem}{Remark}[section]

\title{A series of spectral gaps for the almost Mathieu operator 
with a small coupling constant}
\author{Alexander Fedotov}
\address[Alexander Fedotov]{St. Petersburg State University, 
7/9 Universitetskaya nab., St.Petersburg, 199034, Russia}
\email{a.fedotov@spbu.ru}
\thanks{The work was supported by the Russian
Science Foundation under the grant No. 17-11-01069}
\begin{document}
\setcounter{section}{0}
\begin{abstract} 
For the almost Mathieu operator with a small coupling constant $\lambda$, 
for a series of spectral gaps, we describe the asymptotic locations of the 
gaps and get lower bounds for their lengths. The number of the gaps we 
consider can be of the order of $\ln 1/\lambda$, and  the 
length of the $k$-th gap is roughly of the order of $\lambda^k$.
\end{abstract}
\maketitle
%
%
\section{Introduction} 
We consider the almost Mathieu operator acting in $ l ^ 2 (\Z) $ 
by the formula
\begin {equation}
  \label{op:AM}
  (H_\theta f) _k = f_ {k+1} + f_ {k-1} +2 \lambda \cos (2 \pi (\theta + 
hk)) f_k, 
\quad k \in \Z,
\end {equation}
where $ \lambda> 0 $, \ $ 0 \le \theta <1 $, and $ 0 <h <1 $ are parameters.
The parameter $ \lambda $ is called  a {\it coupling constant}.
The operator~\eqref{op:AM} arises when studying  an electron in a crystal 
submitted to a constant magnetic field when the field is weak, 
when it is strong, in semiclassical regime etc, see, e.g.,~\cite{GHT:89} 
and references therein. This operator attracts attention of  mathematicians 
as well as  physicists thanks to its rich and unusual properties. One of 
the most difficult and interesting problems is the problem of  describing 
the geometry of the spectrum of $ H_ \theta $. During three decades, 
efforts of many mathematicians have been aimed at proving that for 
irrational $ h $ the spectrum is a Cantor set. Among them are A. Avila, J. 
Bellissard, B. Helffer, S.Zhitomirskaya, R. Krikorian, Y. Last, J. Puig, 
B. Simon, J. Sj\"ostrand and many others, see~\cite{A-J:09}, where the 
proof was completed. We mention also~\cite{K:17, LYZZ:17, LS:19} that are 
ones of the latest papers on the geometry of the spectrum.

Among the papers of physicists explaining the cantorian structure of the 
spectrum, we single out the paper~\cite{W:86} containing heuristic 
analysis clear for mathematicians. In the semiclassical approximation, the 
author has successively described sequences of shorter and shorter 
spectral gaps, i.e., obtained a constructive description of 
the spectrum as a Cantor set. According to~\cite{W:86}, the spectrum 
located on certain intervals of the real line, being ``put under a
microscope'', looks like the spectrum of the almost Mathieu 
operator with new parameters. That is why the approach described  by 
Wilkinson is called a renormalization method. Using methods of the 
pseudodifferential operator theory, B. Helffer and J. 
Sj\"ostrand have developed a rigorous asymptotic renormalization method, 
and  turned the heuristic  results into mathematical theorems. 
\\
Let us note that the asymptotic renormalization methods are used when 
the parameter $h$ can be represented as a continued fraction with 
sufficiently large elements.
\\
Later,  V. Buslaev and A. Fedotov suggested  the monodromization 
method, one more renormalization approach that arose when trying to use 
the Bloch-Floquet theory ideas to study the geometry of the spectrum of 
difference operators  in $L^2(\mathbb R)$. The method was further 
developed by A.Fedotov and F.Klopp when studying adiabatic quasiperiodic 
operators. More details can be found in review~\cite {F:13}. The 
monodromization method can be used for one-dimensional two-frequency 
difference and differential quasiperiodic operators independently  of any 
assumptions on  the continued fraction. If such an equation contains an 
asymptotic parameter, one can effectively describe the asymptotic geometry 
of the spectrum.
In~\cite {F:13} we described how to apply the monodromization method to
get a constructive asymptotic description of the spectrum  as a Cantor set
in the case studied by B. Helffer and J. Sj\"ostrand. The present paper is 
the first step to a similar constructive asymptotic description in the 
case of a small coupling constant. Here, we make only the first 
renormalization. This allows to describe asymptotically a series of the 
longest spectral gaps. We get asymptotic formulas for the gap centers and 
lower bounds for the gap lengths. Note that,  the number of the 
gaps we describe can be of the order of $\ln (1/\lambda)$, and the length 
of the $k$-th gap is roughly  of the order of $\lambda^k$. So, it can be
difficult to get such results by means of standard perturbation methods.
\\
The results we prove in this paper were partially announced in the short 
note~\cite{F:18}.

Below $ C $ denotes positive constants independent of any  parameters, 
variables and indices, $H$ denotes expressions of the form 
$Ce^{C/h}$. When writing $ a = O (b) $, we mean that 
$ | a | \le C | b | $, and when writing $a=O_H(b)$, we mean that $|a|\le 
e^{C/h} |b|$.
\\
Furthermore, for $z\in\mathbb C$, we often use the notations 
$x={\rm Re}\, z$ and $y={\rm Im}\, z$.
\\ 
For $v_1,\,v_2\in\mathbb C^2$, we denote by  $V=(v_1\;v_2)$  the 
matrix with the columns $v_{12}$ and $v_2$.
\section{Main results} 
It is well known, see, for example,~\cite {H-S:88},
that for the irrational $h$, as a set, the spectrum of the almost Mathieu 
operator  is independent of the parameter $ \theta $ and coincides with 
the spectrum of the Harper operator acting in $ L ^ 2 (\R) $ by the formula
$ H \psi(x) = \psi(x + h) + \psi(x-h) +2\lambda \cos(2 \pi x) \psi(x)$.
It is a difference Schr\"odinger operator  with a $1$-periodic potential. 
Below we discuss only it.
\\
As for the one-dimensional periodic differential operators, for the 
operator $H$ one can define a monodromy matrix. Here, we describe 
asymptotics of a monodromy matrix and spectral results obtained by means 
of these asymptotics.
\subsection{Monodromy matrix}
\subsubsection{}
Let us consider the Harper equation
\begin {equation}
\label{eq:harper}
  \psi (x + h) + \psi (x-h) +2 \lambda \cos (2 \pi x) \psi (x) = E \psi (x), \quad x \in \R,
\end {equation}
where $ E \in \C $ is a spectral parameter. Its solution space is 
invariant with respect to  translation by one. Let us fix 
a basis in the solution space. The corresponding monodromy matrix 
represents the restriction of the translation operator to the solution 
space. Equation~\eqref{eq:harper} is a second order difference equation 
on $\mathbb R$, and its solution space is a two-dimensional modul over 
the ring of $h$-periodic functions, see, e.g.~\cite{Bu-Fe:01}. Thus, the
monodromy matrix is a $2\times2$ matrix $h$-periodic in $x$. When 
defining a monodromy matrix, it is convenient to make  a linear change 
of variables, for it to be  $1$-periodic. 
The  formal definitions can be found in sections~\ref{def:MM:matrix-eq} 
and~\ref{sss:MM:diffSch}.
\\
The following two theorems describe the functional structure of one of the 
monodromy matrices. 
\begin {theorem}\label{th:1} 
In the solution space of~\eqref{eq:harper}, there exists a basis  such 
that the corresponding monodromy matrix  is of the form
\begin {equation}\label{firstM}
{\mathcal M} (x) =
\left (\begin {array} {cc}
a -2 \lambda_1 \cos (2 \pi x) & s + t \, e ^ {\textstyle-2 \pi ix} \\{} \\
-s-t \, e ^ {\textstyle 2 \pi ix} & st / \lambda_1
\end {array} \right), \ \
a = \lambda_1 \frac {1-s ^ 2-t ^ 2} {st}, \ \  
\lambda_1 = \lambda ^ {\frac1h},
\end {equation}
where the coefficients $ s $ and $ t $ are independent of $ x $
and meromorphic in $E$. 
\end {theorem}
\noindent This theorem is a part of  Theorem 7.3 from ~\cite{Bu-Fe:01}. 
The basis solutions   are minimal entire solutions,
i.e., solutions that are entire functions of $z=x+iy$  growing the 
most slowly  as  $ |y| \to \infty $. These minimal solutions are 
meromorphic in $E$.
\begin{Rem} It follows from the proof of this theorem  that
if, for given $ h=h_0\in(0,1)$,  $ \lambda=\lambda_0 \in (0,\infty)$, and 
$E=E_0\in\mathbb C$, the coefficients $s$ and $t$ are finite, then 
they are continuous in $(h,\lambda, E)$ in a neighborhood of 
$(h_0,\lambda_0, E_0)$.
\end{Rem}
\noindent
In Section~\ref{s:mm:str}, we check
\begin {theorem} \label{th:2} If for an $E\in \R$ the coefficient  $t$ is 
finite, then, for this $E$,
  \begin{equation}
    \label{eq:s/a}
t \in i\R, \quad 
|s| = \lambda_1 \sqrt {\frac{1+ |t|^ 2}{\lambda_1^2 + |t|^2}}, 
  \end{equation}
and the zeroth Fourier 
coefficient of the trace of  the monodromy matrix equals
\begin {equation} \label{trace}
L = \frac{2i}{t} \sqrt {(1+ | t | ^ 2) ( \lambda_1^2 +| t | ^ 2)} \cos ({\rm arg} \, (is)).
\end {equation}
\end {theorem}
\noindent
Relations~ \eqref {eq:s/a} and~\eqref{trace}  reflect the self-adjointness 
of the Harper operator. 
\subsubsection{} Pick $a\in(0,\pi)$.
The asymptotics of the coefficients $ s $ and $ t $ as $ \lambda \to 0 $ are described in terms
of  a meromorphic function $ \sigma_a $ satisfying the  equation
\begin {equation}
  \label{eq:sigma}
  \sigma_a (z + a) = (1 + e ^ {- iz}) \sigma_a (z-a), \quad z \in \C.
\end {equation}
Let $ S = \{z \in \C \,:\, | {\rm Re} \, z | <\pi + a \} $.
The function $\sigma_a$  is uniquely characterized by the following properties. 
In the strip $S$, it  is analytic, does not vanish, tends to one as 
$y\to-\infty$ and has the minimal possible growth as $y\to+\infty$.
This function and functions related to it arose in different areas of  
mathematical physics, see, e.g.,~\cite{Bu-Fe:01, BLG, Bo-Fi:88, F-K-V:01, 
Fe:2016, Ma:58, R:00}. We discuss $\sigma_a$ in section~\ref{sec:sigma}. Let 
\begin{equation}\label{def:F0}
 F_0(p)=\sigma_{\pi h} (4\pi p - \pi +\pi h)\,
\overline{\sigma_{\pi h} (4\pi \overline{p} - \pi +\pi h)}.
\end{equation}
To describe the asymptotics of $t$ and $s$, we also use the parameter $p$ 
related to $E$ by the equation 
\begin{equation}\label{def:p}
 E = 2 \cos (2\pi p),
\end{equation}
and the notation $\xi =\frac1{2\pi}\ln\lambda$. 
\begin {theorem} \label{th:0:ne0} Fix $\beta\in(0,1/2) $. There exists 
$c>0$ such that if $\lambda< e^{-c/h}$, then, for $p$ satisfying the 
inequalities $ |\im p|\le h $ and $ h / 4 <{\rm Re} \, p <1/2-h / 4 $,
one has
\begin {equation}\label{as:st}
t = \frac {ie ^ {\frac {4\pi (1/2-p) \xi} {h}} F_0 (p)} {2 \sin (2\pi p)} 
\,
(1 + O_H(\lambda^\beta)),
 \quad
s = \frac {-2i \, e ^ {\frac {4\pi p \xi} { h} + \frac {2\pi ip} {h}} \, 
\sin (2\pi p)}
{F_0 (p)} \, (1+ O_H(\lambda^\beta)).
\end {equation}
\end {theorem}
\noindent
Note that, in the case of this theorem, one has
$st/\lambda_1=e^{\frac{2\pi ip}{h}} (1+ O_H(\lambda^\beta))$.
\\
The asymptotics of $s$ and $t$ near the points $p=0$ and $p=1/2$ are
more complicated. They are described by Theorem~\ref{th:0:0}.

To get the asymptotics of $s$ and $t$, we obtain asymptotics of 
minimal entire solutions to equation~\eqref{eq:harper} as $\lambda\to0$. 
For this, first, we construct entire  solutions to the 
auxiliary equations
$\psi (x + h) + \psi (x-h) + \lambda e ^ {\mp 2 \pi i x} \psi (x) = E 
\psi(x) $.
Next,  in the half-planes $ \C_ \pm $ respectively, for sufficiently small 
$\lambda$, we construct analytic solutions to the Harper equation 
that are close to the solutions to the auxiliary equations. Finally, with 
the help of a Riemann-Hilbert problem, we make of these analytic 
solutions 
the minimal entire  solutions.

Our asymptotic method works if $ \lambda < e ^ {- c / h} $. If  $ h $ is 
so small  that $ \lambda >> e ^ {- c / h} $, then the asymptotics of the 
solutions to the Harper equation can be obtained by semiclassical methods, 
see, e.g.,~\cite{Bu-Fe:95}. 
\subsection{Spectral gaps} 
The first renormalization of the 
monodromization method consists in replacing equation ~ \eqref {eq:harper}
with {\it the first monodromy equation}
\begin {equation} \label{eq:M1}
\Psi_1(x + h_1) = M_1(x) \Psi_1(x), \quad x \in \C, \qquad h_1 = \{1/h\},
\end {equation}
where $ M_1 $ is a monodromy matrix, and $ \{\cdot \} $ is the fractional 
part. 
Equations~\eqref {eq:harper} and~\eqref {eq:M1} 
simultaneously have pairs of linearly independent solutions such that one 
solution  of a pair decays exponentially as $ x\to+ \infty $, and the 
other decays as $x\to - \infty $, see Corollary~\ref{cor:exp-dec-sol}. 
This allows to find gaps in the spectrum of the Harper equation by 
studying solutions to the monodromy equation. In 
section~\ref{sec:proof:th:5} we prove
\begin {theorem} \label{th:5} Let $I\subset \R$ be an open interval. There 
exists such a constant $ C $ independent of $h$ and $E$ that if 
\begin {equation}\label{cond:gap}
  (L/2)^2 \ge (1 + C \lambda_1) (1+ | t | ^ 2), \quad
  \lambda_1\le | t | \le 1, \qquad \forall E \in I,
\end {equation}
then $ I $ is in a gap of the Harper operator.
\end {theorem}
It is useful to compare this theorem with a well-known theorem from
the theory of the one-dimensional periodic differential  Schr\"odinger 
operators. The latter says  that the spectrum of a periodic operator is 
located  on the intervals where the absolute value 
of the trace of a monodromy matrix is less than or equal to two.

Using Theorem~\ref{th:5}, formula~\eqref {trace} and the asymptotics 
$ s $ and $ t $ described in Theorem~\ref {th:0:ne0},
one can describe a sequence of the longest gaps in the spectrum of the 
Harper operator. As its spectrum is symmetric with respect to zero, and 
as the spectra for the frequencies $h$ and $1-h$ coincide,
we consider only the spectrum located on  $\R_+$ in 
the case where $0<h<1/2$.
Let $ [\cdot] $ be the integer part. One has
\begin {theorem} \label{th:6} Let $h\in(0,1/2)$ and  $\beta\in (0,1/2)$. 
Let  $\lambda \le e ^ {-c / h} $ with a sufficiently large 
$c$. There exist points $E_k>0$, \ $1\le k\le K$, \  $K= 
[1/2h] $, such that 
\begin {equation*}
E_k = 2 \cos (\pi h k + O_H(\lambda^p)),
\end {equation*}
and, if $1\le k\le K-1$ or if $k=K$ and $[1/h]$ is odd, then the point 
$E_k$ is located inside a gap $g_k$. The length of $g_k$ satisfies the 
estimate
\begin {equation} \label{eq:gk}
| g_k | \ge 4\left(\frac{\lambda}4\right)^{k} \; \frac 
{(1+ O_H(\lambda^\beta))}
{\sin ^ 2 (\pi h) \sin ^ 2 (2 \pi h) \dots \sin ^ 2 (\pi h (k-1))}
\end {equation}
where for $ k = 1 $ the product of the sines 
has to be replaced with one.
\\
If $[1/h]$ be even, and  $\lambda^{h_1} \le e ^ {-c / h} $, then the 
point $E_K$ also is located in a gap $g_K$, and the length of $g_K$ 
satisfies \eqref{eq:gk} with  $\beta$ replaced with $p=\min\{h_1,\beta\}$. 

\end {theorem}
\noindent
In the next paper, we will prove that the expression in the right 
hand side of~\eqref {eq:gk} is the leading term of the asymptotics of the 
length of the $k$th gap.
\\
Theorem~\ref{th:6} agree well with the results of computer 
calculations described in~\cite{GHT:89}.
\\
The $K$th gap in the case where where $[1/h]$ is even is the most 
difficult to describe. For small $h h_1$, it is located near zero that is 
a very special point. For $h\not\in\mathbb Q$, the complexity of the 
spectrum near zero  is well-known. One can find a series of new 
results  in~\cite{K:17}.

The right hand side in~\eqref{eq:gk} equals $\lambda^ke^{O(1/h)}$,  
see Corollary~\ref{cor:gk:roughly}. Therefore, if $\lambda$ 
is small, and $h$ is of the order of $\frac1{|\ln\lambda|}$, then  
$|g_K|$, the length of the gap closest to zero,  is of the order of 
$\lambda^{|\ln\lambda| +O(1)}$. So, it can be rather difficult to 
compute $|g_K|$ using standard perturbation methods. 

According to~\cite{LYZZ:17}, for the almost Mathieu 
equation with $\lambda<1$ and Diophantine frequencies,  as the gap number, 
say, $k$ tends to infinity, the gap lengths are bounded from above and 
below by expressions of the form $C(\lambda,h,\epsilon) \lambda^{(1\pm 
\epsilon)k}$, where $\epsilon$ is a fixed positive number.\footnote{I am 
grateful to Qi Zhou for attracting my attention to paper~\cite{LYZZ:17} 
and discussing its results.} And as we mentioned, in our case, one has  
$|g_k|\ge\lambda^ke^{O(1/h)}$.

It is interesting  to compare our results with the results obtained in the 
case of small $ h $  and $ \lambda = 1 $, see~\cite {F:13}. In this case, 
there is a statement similar to Theorem~\ref {th:5}. However, the 
asymptotics 
of the coefficients $ t $ and $ s $ turn out to be quite different, and, on
the most of the  interval $[-4,4]$ containing the spectrum,  $L(E)$  
oscillates with an amplitude that is exponentially large with respect to 
$h$ , whereas in our case, for most $ E \in (0,2) $, one has $L\approx 2 
\cos (p / h) $. Thus, in the case of  small $ h $,  the spectrum is 
located on a series of exponentially small intervals, and in the case of 
small $ \lambda $, we observe small gaps in the spectrum. 
\subsection{Other gaps} Since the matrix $ M_1 $ is $ 1 $-periodic,
for equation~\eqref {eq:M1}, we can also define a monodromy matrix $M_2$ 
and consider the second monodromy equation that can be obtained from 
the first one by replacing $M_1$ with $M_2$ and $h_1$ with $h_2=\{1/h_1\}$.
Continuing, we can construct an infinite sequence  of difference equations.
In the next paper, studying consequently the equations of this sequence, 
we will  describe consequently  series of shorter and shorter gaps. 
We  also obtain upper bounds for the gaps lengths. 

Note that, to prove Theorem~\ref{th:6}, we use only the asymptotics of 
the coefficients $s$ and $t$ from Theorem~\ref{th:0:ne0}, i.e., their 
asymptotics for $E$ being bounded away from $2$. 
However,  the asymptotics of $s$ and $t$ for $E$ close to $2$ are 
crucial to study the geometry of the spectrum near its edge. All we need 
to get these asymptotics is prepared when proving Theorem~\ref{th:0:ne0}, 
and they are obtained almost like the ones described in this theorem. So, 
omitting elementary details, we get them 
in section~\ref{ss:s-and-t-for-p-close-to-0}.
\subsection{The plan of the paper}
In section~\ref{sec:MM}, we give the definition of a monodromy matrix, and 
prove Theorems~\ref{th:5} and~\ref{th:6}. For this we use 
Theorems~\ref{th:1}--\ref{th:0:ne0}. The most of the remaining part of 
the paper is devoted to the proof of Theorem~\ref{th:0:ne0}.
In section~\ref{model:eq}, we construct and analyze analytic solutions to 
the model equation~\eqref{mu:eq}.  
In section~\ref{sec:int-eq}, we show that, in the upper half-plane, there 
are analytic solutions to the Harper equation that are close 
to the solutions to the model equation.
Recall that  the monodromy matrix described in Theorem~\ref{th:1} 
corresponds to a  basis of two minimal entire solutions to the Harper 
equation. In section~\ref{s:mm:str}, we recall the definition of minimal 
entire solutions and prove Theorem~\ref{th:2}.
In section~\ref{as:mm}, for sufficiently small $\lambda$, using the 
analytic solutions to the Harper equation constructed in  
section~\ref{sec:int-eq}, we construct and study  the minimal entire 
solutions, and  prove Theorem~\ref{th:0:ne0}. 
Section~\ref{ss:s-and-t-for-p-close-to-0} is devoted to the asymptotics 
of $s$ and $t$ for $p$ close to zero (i.e., for $E$ close to 2).
In Section~\ref{sec:sigma}, we describe the properties of 
$\sigma_a$ that are used in this paper.
 \section{Monodromy matrices, monodromy
   equation and spectral results}\label{sec:MM}
 Here we remind the definition of a monodromy matrix, describe relations
 between solutions to a difference equation with periodic coefficients and
 solutions to a corresponding monodromy equation, and  prove
 Theorems~\ref{th:5} and~\ref{th:6}.
 \subsection{Monodromy matrices and monodromy equation}
 \subsubsection{Definition and elementary properties of a monodromy 
matrix}\label{def:MM:matrix-eq}
 Here, following~\cite{F:13} we discuss the difference equations of the 
form
 \begin{equation}\label{eq:matrix}
  \Psi(x + h) = M (x) \Psi (x),
 \end{equation}
where $x$ is a real variable, $M:\R\to SL(2,\C)$ is a given 1-periodic 
function,
and $h \in (0,1)$ is a fixed number.
\\
Obviously, for any solution $\Psi$ to~\eqref{eq:matrix},
we have  $\det \Psi(x + h) = \det \Psi(x)$,\ $x\in\R$.
\\
We call $SL(2,\mathbb C)$-valued solutions to~\eqref{eq:matrix} 
fundamental matrix solutions.
\\
Note that, to construct a fundamental solution, it suffices to define
it arbitrarily on  the interval $0< x < h$, and then, to define its
values outside of this interval directly with the help of 
equation~\eqref{eq:matrix}.
\\
It can be shown that $\tilde \Psi \mathbb R\mapsto M_2(\C)$
is a matrix-valued solution to~\eqref{eq:matrix} if and only if it can be 
represented in the form
\begin{equation*}
  \tilde \Psi(x)= \Psi (x) \;p (x), \quad x \in\R,
\end{equation*}
where $p:\R\mapsto M_2(\C)$ is an $h$-periodic function, and $\Psi$ is a
fundamental solution.
\\
Note that this representation implies that the space of 
matrix-valued solutions to~\eqref{eq:matrix} is a module over the ring of 
$h$-periodic functions.
\\
Let $\psi_1,\ \psi_2 : \R\to \C^2$ be two vector-valued solutions 
to~\eqref{eq:matrix}. We say that they are linearly independent if 
$\det(\psi_1,\; \psi_2)$ does not vanish.
In this case,  a function $\psi: \R\to \C^2$ is a vector-valued solution 
to~\eqref{eq:matrix} if and only
if it is a linear combination of $\psi_1$ and $\psi_2 $ with $h$-periodic 
coefficients.
\\
Let $\Psi$  be a fundamental solution. As $M$ is $1$-periodic, the function
$x\mapsto \Psi(x + 1)$ is also a solution to~\eqref{eq:matrix}, and we can 
write
\begin{equation*}
\Psi(x + 1) = \Psi (x)\;p (x),\quad p(x + h) = p(x),\qquad x \in\R.
\end{equation*}
The matrix $M_1 (x) = p^t(hx)$, where ${}^t$ denotes transposition,
is called the monodromy matrix corresponding to the fundamental solution 
$\Psi$.
\\
Note that, by construction, a monodromy matrix is 1-periodic and
unimodular.
\\
In early papers, the $h$-periodic matrix $p$ was called the monodromy
matrix. It is more convenient to consider $1$-periodic monodromy matrices.
\subsubsection{Monodromy equation}
Let $M_1$ be the monodromy matrix corresponding to a fundamental solution
$\Psi$ to~\eqref{eq:matrix}. Let us consider  {\it the first monodromy 
equation}~\eqref{eq:M1}.
It appears that the behavior of solutions
to~\eqref{eq:matrix} at infinity ``copies'' the behavior of solutions 
to~\eqref{eq:M1}.
Let us formulate the precise statement.
\\
Let $M$ be a $SL(2,\C)$-valued function of real variable, and $h>0$.
Let  $k\in \Z$ and $x\in \R$. We put
\begin{equation*}
P_k(M,x,h)=
\begin{cases}
 M(x+h(k-1))\dots M(x+h)M(x),& k\ge 0,\\
 M^{-1}(x+ h k)\dots  M^{-1}(x-2h)M^{-1}(x-h),& k<0.
\end{cases}
\end{equation*}
Clearly, if $\psi\,:\,\R\to \C^2$ satisfies~\eqref{eq:matrix}, then
\begin{equation*}
  \psi(x+hk)=P_k(M,x,h)\psi(x)
\end{equation*}
One has
\begin{theorem} \cite{F-S:15}
\label{th_renormalization_formula}
Let $\Psi$ be a fundamental  solution to~\eqref{eq:matrix}, 
and let $M_1$ be the corresponding monodromy matrix.
Then, for all $N\in \Z$,
\begin{gather}
\label{eq_main_renormalization_formulae}
P_N(M,h,x)=\Psi(\{x+Nh\})\sigma_2
P_{N_1}(M_1,h_1,x_1)\sigma_2\Psi^{-1}(x),\\
\nonumber
N_1=-[\theta+N h],\qquad h_1=\left\{1/h\right\},\qquad 
x_1=\left\{x/h\right\},
\end{gather}
where  $\sigma_2$ is the Pauli matrix $\begin{pmatrix} 0&-i\\ i& 
0\end{pmatrix}$.
\end{theorem}
\noindent In this paper we use 
\begin{Cor}\label{cor:exp-dec-sol} 
  In the case of  Theorem~\ref{th_renormalization_formula}, we assume that
  $\Psi\in L_{loc}^\infty(\R,SL(2,\C))$, and that $\psi_\pm^{(1)}$ are two 
  vector-valued solutions to the monodromy equation such that
  \begin{equation}\label{est:psi1}
   \|\psi_\pm^{(1)}(\pm x)\|_{\C^2}\le C_0 e^{\mp \varkappa\,x},\quad x\ge 
0,
 \end{equation}
  with some positive constants $C_0$ and $\varkappa$.
  Then there are two vector-valued solutions $\psi_\pm^{(0)}$  to 
equation~\eqref{eq:matrix}
  such that
  \begin{equation}\label{det:mon}
    \det \left(\psi_+^{(0)}(x), \psi_-^{(0)}(x)\right)=
    \det \left(\psi_+^{(1)}(\{x/h\}), \psi_-^{(1)}(\{x/h\})\right),\quad 
\forall x\in\R, 
  \end{equation}
  \begin{equation}\label{est:psi0}
    \|\psi_\pm^{(0)}(\pm x)\|_{\C^2}\le C_1 e^{\mp \varkappa\,h_1 x},\quad 
\forall x\ge 0,
  \end{equation}
  with a positive constant  $C_1$.
\end{Cor}
\begin{proof} As $P_N(M, h, x) \Psi(x)=\Psi(x+Nh)$,
   formula~\eqref{eq_main_renormalization_formulae} implies that
  \begin{equation}\label{eq:cor:renorm}
\Psi(x+Nh)\sigma_2=\Psi(\{x+Nh\})\sigma_2
P_{N_1}(M_1,h_1,x_1).
  \end{equation}
  Let us define the  solutions $\psi_\pm^{(0)}\,:\R\to\C^2$ 
to~\eqref{eq:matrix}
  by the formulas
  \begin{equation}\label{def:psi0}
    \psi_\pm^{(0)}(x)=\Psi(x)\sigma_2\psi_\mp^{(1)}\left(\{x/h\}\right).
  \end{equation}
  As $\det\Psi\equiv 1$, relation~\eqref{det:mon} is obvious.
  Furthermore, as $x_1=\{x/h\}$, and $N_1=-[x+Nh]$,
  formulas~\eqref{def:psi0} and~\eqref{eq:cor:renorm} lead to the relation
  \begin{equation*}    
\psi_\pm^{(0)}(x+Nh)=\Psi(\{x+Nh\})\sigma_2\psi_\mp^{(1)}(x_1-[x+Nh]h_1),
    \quad N\in \Z, 
  \end{equation*}
  that can be rewritten in the form
  \begin{equation*}
    \psi_\pm^{(0)}(x)=\Psi(\{x\})\sigma_2\psi_\mp^{(1)}(x_1-[x]h_1), \quad
    x\in \R. 
  \end{equation*}
  This formula and estimates~\eqref{est:psi1} imply~\eqref{est:psi0}.
  The proof is complete.
 \end{proof} 
 \subsubsection{Monodromy matrices for difference Schr\"odinger
   equations}\label{sss:MM:diffSch}
Let $h>0$ and $v:\R\to \C$. The difference Schr\"odinger equation
 \begin{equation}\label{diff-Sch}
   \psi(x+h)+\psi(x-h)+v(x)\psi(x)=E\psi(x),\quad x\in \R,
 \end{equation}
 is equivalent to~\eqref{eq:matrix} with
\begin{equation}\label{eq:M:Shcr}
 M(z)=\begin{pmatrix} E-v(x) & -1 \\ 1 & 0\end{pmatrix}.
\end{equation}
More precisely, a function $\Psi:\R\mapsto C^2$
satisfies~\eqref{eq:matrix} with this matrix if and only if
$\Psi(x)=\begin{pmatrix}\psi(x) \\ \psi(x-h)\end{pmatrix}$, and
$\psi$ is a solution to~\eqref{diff-Sch}.
This allows to turn the observations made for~\eqref{eq:matrix}
into observations for~\eqref{diff-Sch} . 
\\
Let $\psi_1$ and $\psi_2$ be two solutions to~\eqref{eq:matrix}.
The expression
\begin{equation*}
  \{\psi_1(x),\psi_2(x)\}=\psi_1(x+h)\psi_2(x)-\psi_1(x)\psi_2(x+h),
\end{equation*}
their Wronskian,  is $h$-periodic in $x$.
\\
Assume that  the Wronskian  is constant and nonzero. Then $\psi_{1,2}$ 
form 
a basis in the 
space of solutions, and a function $\psi$ satisfies~\eqref{diff-Sch}
if and only if 
\begin{equation}\label{eq:three-solutions}
  \psi(x)=a(x)\psi_1(x)+b(x)\psi_2(x),
\end{equation}
where $a$ and $b$ are $h$-periodic coefficients.
One easily proves that
\begin{equation}
  \label{eq:periodic-coef}
  a(x)=\frac{\{\psi(x),\,\psi_2(x)\}}{\{\psi_1(x),\,\psi_2(x)\}},\quad 
  b(x)=\frac{\{\psi_1(x),\,\psi\,(x)\}}{\{\psi_1(x),\,\psi_2(x)\}}.
\end{equation} 
If $v$ is $1$-periodic,  the functions $x\to \psi_1(x+1)$ and $x\to 
\psi_2(x+1)$
are also solutions to~\eqref{diff-Sch}, and one can write
\begin {equation}
  \label{eq:MM-def} \Psi (x + 1) = M_1 (x / h) \Psi (x), \quad \Psi (x) =
  \begin {pmatrix} \psi_1 (x) \\ \psi_2 (x) \end {pmatrix}, \quad x \in \R,
\end {equation}
where $ M_1 $ is a $ 1 $-periodic $2\times2$ matrix. It is  the matrix
monodromy corresponding to $\psi_1$ and $\psi_2$.
It  coincides with a monodromy matrix for~\eqref{eq:matrix} with the
matrix~\eqref{eq:M:Shcr}. 
\subsection{Gaps in the spectrum of the Harper equation: proof 
of Theorem~\ref{th:5}}\label{sec:proof:th:5}
Let us consider equation~\eqref{eq:M1} with the matrix $M_1$ 
described in Theorem~\ref{th:1}.
Theorem~\ref{th:5}, a sufficient condition for $E$ to 
be in gap, follows from 
\begin{Pro}\label{pro:M1:gaps} In the case of Theorem~\ref{th:5},  for all 
$E\in I$,  there
  exist two vector solutions $\psi_\pm^{(1)}$ to~\eqref{eq:M1} such that
  $\det(\psi_+^{(1)}(x),\psi_-^{(1)}(x))$ is a nonzero constant,  and, for 
$x\ge 0$, \   
   $ \|\psi_\pm^{(1)}(x)\|_{\C^2}\le Ce^{\mp x\,\ln\frac{|L|}2}$.
\end{Pro}
\noindent This proposition implies 
\begin{lemma}\label{le:polynomially}
In the case of Theorem~\ref{th:5}, for any  $E\in I$, the only 
polynomially bounded solution to the 
Harper equation~\eqref{eq:harper} is zero.
\end{lemma}
\noindent
First, using these proposition and lemma, we prove Theorem~\ref{th:5}.
\\
Assume that the Harper operator has some spectrum on $I$. As it is a 
direct integral of the almost Mathieu operators with respect to $\theta$, 
then,  for some $\theta$,  the almost Mathieu operator, has a nontrivial 
spectrum on $I$, and on $I$ almost everywhere  with respect to the 
spectral 
measure,
the almost Mathieu equation
\begin {equation*}
   f_ {k+1} + f_ {k-1} +2 \lambda \cos (2 \pi (\theta + hk))f_k=Ef_k, 
\quad k \in \Z.
\end {equation*}
has a polynomially bounded solution $f$  (see section 2.4 in~\cite{CFKS}).
One defines a solution $\psi$ to the Harper equation so that 
$\psi(x)=f_k(\theta)$ if $x=\theta+kh$ with $k\in\mathbb Z$,  and 
$\psi(x)=0$ otherwise. The $\psi$ is a non-trivial polynomially 
bounded solution to the Harper equation. This contradicts   
Lemma~\ref{le:polynomially}. 
\\
Now, let us prove
Lemma~\ref{le:polynomially} and Proposition~\ref{pro:M1:gaps}.
\\
{\it Proof of Lemma~\ref{le:polynomially}}. \  Let us assume that, for an 
$E\in I$, there is a nontrivial polynomially bounded solution $\psi$ 
to~\eqref{eq:harper}.
For this $E$ we construct the solutions  to the monodromy equation 
described in 
Proposition~\ref{pro:M1:gaps}. Then, in terms of these solutions, we 
construct 
the solutions $\psi_\pm^{(0)}$ to equation~\eqref{eq:matrix} with 
matrix~\eqref{eq:M:Shcr} as described in Corollary~\ref{cor:exp-dec-sol}
(this is possible as the fundamental solution used to define the monodromy 
matrix from Theorem~\ref{th:1} is entire in $x$).
\\
Let $\psi_1$ be the first entry of $\psi_+^{(0)}$, and $\psi_2$ be the one 
of 
$\psi_-^{(0)}$.
One has 
$$\{\psi_1(x),\,\psi_2(x)\}=\det(\psi_+^{(0)}(x),\,\psi_-^{(0)}(x))=\det 
(\psi_+^{(1)}(x),\,\psi_-^{(1)}(x)),$$ 
where we used Corollary~\ref{cor:exp-dec-sol}. Thus, by 
Proposition~\ref{pro:M1:gaps}, $\{\psi_1(x),\,\psi_2(x)\}$ is a nonzero 
constant. Therefore, one has
\eqref{eq:three-solutions}--\eqref{eq:periodic-coef}. Now, it suffices to 
show that the
Wronskians  $\{\psi(x),\,\psi_{j}(x)\}$, \ $j=1,2$,  equal zero.
But this is obvious, as these Wronskians are periodic and, on the other 
hand,
$\{\psi(x),\,\psi_\pm(x)\}\to 0$ as $x\to\pm \infty$ since $\psi$ is 
polynomially bounded, and
$\psi_\pm$ are exponentially decreasing as $x\to\pm \infty$. The proof of 
Lemma~\ref{le:polynomially}  is complete. \qed

\medskip

\noindent Now, let us prove Proposition~\ref{pro:M1:gaps}.
\begin{proof} Below we assume that $E\in I$, and that
conditions~\eqref{cond:gap} are satisfied. 
\\
In view of Theorem~\ref{th:1}, we can represent  the monodromy matrix in 
the form
\begin{equation}
  \label{eq:MMM}
  M_1(x)=M_1^0+\tilde M(x), \quad \tilde M(x)=O(t), \qquad
  M_1^0=\begin{pmatrix}
    \frac{\lambda_1}{st}(1-s^2-t^2) & s \\ -s & \frac{st}{\lambda_1}
    \end{pmatrix}.
\end{equation}
The plan of the proof is the following.
First, we transform the monodromy equation with a matrix $M_1$ of 
the form~\eqref{eq:MMM} to the equation
\begin{equation}
  \label{eq:phi}
  \phi(x+h)=p \left(D+\Delta(x)\right)\,\phi(x), \quad
  D=\begin{pmatrix} 1/U & 0 \\ 0 &
    U\end{pmatrix},\qquad x\in\R,
  \end{equation}
where $p$ and $U$ are parameters, and $\Delta$ 
is a ``sufficiently   small'' matrix.  Then, we construct two   solutions 
to~\eqref{eq:phi}   by means of 
\begin{Pro}
  \label{le:resolvent-set:1} Let us consider equation~\eqref{eq:phi}
  with parameters  $h>0$, $p\ge 1$ and $U\in \mathbb R$, and a function
  $p\left(D+\Delta\right) \in L^\infty(\R,SL(2,\C))$. 
  Let 
  \begin{equation}
    \label{f-d-e-on-R1}
   |U+U^{-1}|> 2, \quad\text{and}\quad |U-U^{-1}|> 4m, \qquad 
   m=\max_{1\leq i,j\leq 2}\sup_{x\in\R}|\Delta_{ij}(x)|.
  \end{equation}
  There exist $\phi_\pm \in L^{\infty}_{\rm loc} (\R,\C^2)$,  
vector-valued solutions to~\eqref{eq:phi}, such that
  \begin{equation}
   \label{est:phi}
    \|\phi_\pm(\pm x)\|_{\C^2}\le Ce^{ \mp \frac{x}h\ln\frac{p|U+U^{-1}|}2}
    \ \  \forall x\ge 0,\qquad \inf_{x\in 
\R}|\det(\phi_+(x),\phi_-(x))|>0.
  \end{equation}
\end{Pro} 
\noindent
Mutatis mutandis, the proof of Proposition~\ref{le:resolvent-set:1} 
repeats the proof  of Proposition 4.1 from~\cite{F-K:05b} where we have  
considered the case of $p=1$. 
\\ 
The $\det(\phi_+,\phi_-)$ being $h$-periodic, the function
$x\mapsto \phi_+(x)/\det(\phi_+(x),\phi_-(x))$ satisfies~\eqref{eq:phi}.
We keep for this new function the old notation $\phi_+$. It
belongs to $L^{\infty}_{\rm loc} (\R,\C^2)$, satisfies 
estimate~\eqref{est:phi}, and we have  $\det(\phi_+(x),\phi_-(x))=1$. 
\\
To complete the proof of Proposition~\ref{pro:M1:gaps}, we return 
from~\eqref{eq:phi} to the monodromy equation constructing 
$\psi_\pm^{(1)}$ in terms of $\phi_\pm$.

Let us transform the monodromy equation to the form~\eqref{eq:phi}.
Therefore, we compute the eigenvalues and eigenvectors of $M_0$.
In view of Theorem~\ref{th:2}, one has 
\begin{equation*}
  t=i\tau,\quad \tau\in\R,\qquad
  s=-i\lambda_1\sqrt{\frac{1+\tau^2}{\tau^2+\lambda_1^2} }\;e^{i\alpha},
  \quad \alpha\in \R,
\end{equation*}
  Let
  \begin{equation*}
    p=\sqrt{1+\tau^2},  \quad q=\frac1\tau\sqrt{\tau^2+\lambda_1^2},\quad
    Q=\sqrt{q^2\cos^2\alpha-1}.
  \end{equation*}
Then
  \begin{equation*}
    \tr M_1^0=L=2pq  \cos\alpha,\qquad \det M_1^0=p^2.
  \end{equation*}
  The eigenvalues $\nu_\pm$ and the corresponding
  eigenvectors  $v_\pm$ are given by the formulae
  \begin{gather*}
      \nu_\pm= p\left(q\cos\alpha \pm Q\right),\quad 
      v_\pm=\begin{pmatrix}
        1 \\ -\frac{p}s \left( q\cos\alpha-\frac1qe^{i\alpha}\mp Q\right)
    \end{pmatrix}.
  \end{gather*}
  Let $V=(v_+\;v_-)$. We represent
  a vector-valued solution to~\eqref{eq:M1} in the form
  \begin{equation}\label{eq:psi-phi}
    \psi(x)= V\phi(x).
  \end{equation}
  Then $\phi$ satisfies equation~\eqref{eq:phi}  with
  \begin{equation}\label{eq:theta-Delta}
    U=q\cos\alpha +Q \quad \text{and}
    \quad \Delta(x)=\frac1p\,V^{-1} \tilde M(x)V.
  \end{equation}
  Let us determine the conditions under which 
  $U$ and $\Delta$ from \eqref{eq:theta-Delta} satisfy the assumptions of
  Proposition~\ref{le:resolvent-set:1}.
  Let
  \begin{equation}
    \label{eq:cond1}
    q^2\cos^2\alpha>1.
  \end{equation}
  We can and do assume that  $Q>0$ and consider the case where 
  $q\cos\alpha>1$. The complementary case is analyzed similarly.
  \\
  We have $U>1$, and the first condition in~\eqref{f-d-e-on-R1} is 
  satisfied. 
  \\
  Now, let us estimate the entries of $\Delta$. One has 
  \begin{equation*}
      |\sin\alpha|=\sqrt{1-\cos^2\alpha}\le \sqrt{1-\frac1{q^2}}\le
  \left|\frac{\lambda_1}\tau\right|,\quad
    Q=\sqrt{\left(1+\frac{\lambda_1^2}{\tau^2}\right)\cos^2\alpha-1}
    \le\left|\frac{\lambda_1}\tau\right|.
\end{equation*}
Therefore and as $\lambda_1\le |\tau|$,
the second entries of $v_\pm$ are uniformly bounded :
\begin{equation*}
 \left|\frac{p}s \left( q\cos\alpha-\frac1qe^{i\alpha}\mp Q\right)\right|
 =\left|\frac{p}{qs}\right|\left|\frac{\lambda_1^2}{\tau^2}\cos\alpha-
   i\sin\alpha\mp qQ\right|\le C\left|\frac{p\lambda_1}{qs\tau}\right|=C.
\end{equation*}
As $\det V=-2pQ/s$, this estimate implies that, for all $i,j\in\{1,2\}$,
\begin{equation*}
  \max_{x\in\R} |\Delta_{ij}(x)|\le C\left|\frac{s\tau}{p^2Q}\right|=
  C\left|\frac{\lambda_1}{pqQ}\right|< C\frac{\lambda_1}{Q}.
\end{equation*}
Therefore, the second condition from~\eqref{f-d-e-on-R1} is satisfied if
\begin{equation}\label{eq:phi:1}
  \frac{U-U^{-1}}2=Q\ge C\frac{\lambda_1}{Q}
  \quad \Longleftrightarrow\quad
  q^2\cos^2\alpha\ge 1+C\lambda_1.
\end{equation}
Clearly, this condition implies~\eqref{eq:cond1}
and, as $L=2pq\cos\alpha$, it  is equivalent to~\eqref{cond:gap}.
\\
Let assume that~\eqref{eq:phi:1} is satisfied.
Then, using Proposition~\ref{le:resolvent-set:1} and
formula~\eqref{eq:psi-phi}, we construct two vector-valued solutions
$\psi_\pm=V\phi_\pm$   to the monodromy equation.
As $p(U^{-1}+U)/2=pq\cos\alpha=L/2$, estimates~\eqref{est:phi} 
imply  the estimates for $\psi_\pm$ from Proposition~\ref{pro:M1:gaps}.
Furthermore, one has
$$\det(\psi_+(x),\psi_-(x))=
\det V\,\det(\phi_+(x),\phi_-(x))=\det V\ne 0.$$
The proof is complete.
\end{proof}
\subsection{Gaps in the spectrum of the Harper equation: proof 
of Theorem~\ref{th:6}}
Theorem~\ref{th:6} describing a sequence of gaps in the spectrum of the 
Harper operator follows from Theorem~\ref{th:5}, a sufficient condition 
for 
$E$ to be in a gap written in terms of the coefficients $s$ and $t$ of a 
monodromy 
matrix,  
and Theorem~\ref{th:0:ne0} describing the asymptotics of the monodromy 
matrix 
as $\lambda\to0$.
\\
Below we assume that 
\begin{equation*}
 p\in I_p=[h/4,\; 1/4].
\end{equation*}
where  $p$ is the parameter  related to $E$ by~\eqref{def:p}. 
In view of~\eqref{trace}, 
condition~\eqref{cond:gap} can be written in the form
\begin{equation}\label{cond:gap:1}
 (1+X^2)\cos^2\alpha\ge 1+C_0\lambda_1, \quad  
X=\lambda_1/\tau=i\lambda_1/t,
\end{equation}
where  $\alpha=\arg{is}$, \ $s$ and $t$ being the coefficients from 
Theorem~\ref{th:1},  $\lambda_1=\lambda^{1/h}$, and $C_0$ is a certain 
positive constant. 
\\
Below, when analyzing~\eqref{cond:gap:1}, we assume that the following 
hypothesis is true.
\begin{hypothesis}\label{H}
One has $\lambda\le e^{-c_0/h}$, where $c_0>0$ is 
a sufficiently large  constant. 
\end{hypothesis} 
\noindent
This hypothesis is needed in particular, to use Theorem~\ref{th:0:ne0} on 
the asymptotics of the monodromy matrix coefficients $s$ and $t$.
\\
Also, we fix $\beta\in (0,1/2)$.
\subsubsection{Locations of gaps}
Inequality~\eqref{cond:gap:1} can be rewritten as
$\frac{X^2-C\lambda_1}{1+X^2}\ge \sin^2\alpha$, and assuming that 
$\lambda_1/X^2$ and $X^2$ are sufficiently small, we transform it to the 
form
\begin{equation} \label{cond:gap:2}
 |X|\,(1+O(\lambda_1/X^2)+O(X^2))\ge |\sin \alpha|,
\end{equation} 
and see that there are gaps containing the points $E_k$ defined by 
the relations
\begin{equation}\label{def:Ek} 
 E_k=2\cos(2\pi p_k),\quad \alpha(p_k)= \pi k,\quad k\in\mathbb Z
\end{equation} 
To continue, we need the following two lemmas:
\begin{lemma}\label{le:as:alpha} For  $p\in I_p$, one has
\begin{equation}\label{as:alpha}
\alpha=\frac{2\pi p}h+ O_H\big(\lambda^\beta \big).
\end{equation}
If $p>Ch>h/4$,  the error term is analytic in $p$ and satisfies the 
estimate
\begin{equation}\label{est:error:alpha}
 \Big(O_H\big(\lambda^\beta\big)\Big)'_p=O_H\big(\lambda^\beta\big).
\end{equation}
\end{lemma}
\begin{proof} Pick $c$ sufficiently small.  Let $V_{ch}$ be the 
$(ch)$-neighborhood of $I_p$. 
\\
Recall that $F_0$ is given by~\eqref{def:F0}. The description of the zeros 
and poles of the function $\sigma_{a}$ from 
section~\ref{sigma:poles,zeros} 
implies that in $V_{ch}$ the function $F_0$ is real analytic and does not 
vanish.
\\
Formula~\eqref{as:alpha} follows from the second formula in~\eqref{as:st}.
\\
As the function $p\mapsto\sin(2\pi p) / F_0 (p) $  
from~\eqref{as:st} is analytic and does not vanish 
in $V_{ch}$, and as $s$ is meromorphic in $E$, 
the second formula in~\eqref{as:st} implies   that
$s$ is bounded and thus analytic in $p\in V_{ch}$. And the analyticity 
of $s$ and the same formula imply that the error term in~\eqref{as:st} is 
also analytic in $V_{ch}$.
\\
The estimate~\eqref{est:error:alpha} follows from the analyticity of the 
error term and the Cauchy representation for the derivatives of analytic 
functions.
\end{proof}
\begin{lemma}\label{le:as:x} For $p\in I_p$, one has
\begin{gather}\label{as:x}
 X=X_0(p)\left(1+O_H\left(\lambda^\beta\right)\right),\qquad  
X_0(p)=e^{2p\ln\lambda /h} \frac{2\sin(2\pi 
p)}{F_0(p)},
\\
\label{est:lnX0}
\ln X_0(p)=\frac{2p\ln\lambda}h +O\Big(\frac1h\Big),\qquad 
\left(O\Big(\frac1h\Big)\right)'_p=O\Big(\frac1h\Big).
\end{gather}
\end{lemma}
\begin{proof} 
Formulas~\eqref{as:x} follow from Theorem~\ref{th:0:ne0}.
\\
For $p\in I_p$, we get
\begin{equation}\label{eq:lnX0}
\ln X_0=\frac{2p\ln\lambda}h-\ln F_0(p) +\ln p+O(1),\quad \frac{d 
O(1)}{dp}=O(1),
\end{equation}
where $\ln F_0(p)$ and $\ln p$ are real.  It suffices to show that
\begin{equation*}
 \ln F_0(p)=O(1/h),\quad  (\ln F_0)(p)=O(1/h),\quad p\in I_p.
\end{equation*} 
Let us fix $\delta\in (0,\pi)$.
By Corollary~\ref{cor:sigma:2}
we get
\begin{equation}\label{est:f}
\ln F(p)=O\left(1/h\right)\quad 
 \text{if}\quad  |4\pi p+\pi h|\ge \delta, \ \ -h/4\le \re p \le 1 \text{ 
\ 
and 
\ }|\im p|\le 1.
\end{equation}
This estimate and the Cauchy representations for the derivatives 
of analytic functions imply that, in any fixed subinterval of the interval 
$[\delta/4\pi-h/4,  1]$, one has $(\ln F)'(p)=O(1/h)$.
However, as  $\delta$ is chosen quite arbitrarily, we can write
\begin{equation}\label{est:df}
 (\ln F)'(p)=O(1/h), \quad \delta/4\pi-h/4\le p\le 1/4.
\end{equation} 
If $\max \{\delta/4\pi-h/4,\;h/4\}\le p\le 1/4$, 
estimates~\eqref{est:f},~\eqref{est:df}, \eqref{eq:lnX0} 
imply~\eqref{est:lnX0}.
\\
Let us assume that $|4\pi p +\pi h|\le\delta$ and $p\in I_p$. 
For $p\in I_p$, one has $2p/h>1/2$.
\\
Theorem~\ref{theta:uni-rep:2} implies that
\begin{equation*}
 \ln F(p)=\frac{4p\ln h}{h}+2\ln\Gamma\,\Big(\frac{2p}h+1\Big)+
O\Big(\frac1{h}\Big),\quad  
\left(O\Big(\frac1h\Big)\right)'_p=O\Big(\frac1h\Big),\quad p\in I_p,
\end{equation*}
where the values of the logarithms are real. 
When deriving this formula, we have estimated the derivative of the error 
term from~\eqref{theta:delta} arguing as when proving~\eqref{est:df}.
\\
As $2p/h>1/2$, we estimate $\ln\Gamma(2p/h+1)$ using the Stirling formula 
$(\ln\Gamma)(x+1)=x(\ln x-1)+O(\ln  x)$, where the error term satisfies 
the estimate $(O(\ln  x))_x'=O(1/x)$. We get 
\begin{equation*}
 \ln F(p)=\frac{4 p\ln 
p}{h}+O\Big(\frac1{h}\Big)=O\Big(\frac1{h}\Big),\quad  
\left(O\Big(\frac1h\Big)\right)'_p=O\Big(\frac1h\Big),\qquad 
\frac{h}4\le p\le \frac{\delta}{4\pi} -\frac{h}4,
\end{equation*}
and thus, in view of~\eqref{eq:lnX0}, estimate \eqref{est:lnX0} . The 
proof 
is complete.
\end{proof}
Lemma~\ref{le:as:alpha}  immediately implies 
\begin{Cor} For  $p_k\in I_p$, one has 
\begin{equation}\label{as:pk}
 p_k=\frac{hk}2+O_H\big(\lambda^\beta \big)
\end{equation} 
where  the error is uniform uniform in $k$
\end{Cor}
\noindent One can easily see that, for sufficiently small $\lambda$  
$hk/2\in I_p$ if and only if 
\begin{equation}
 \label{eq:k:pk}
 1\le k \le K,\quad K=\left[\frac1{2h}\right],
\end{equation} 
where $[x]$ denotes the integer part of $x\in\mathbb R$.
We note also that
 \begin{equation*}
  \textstyle\frac14-\frac{hK}2=\frac{h}4\cdot
  \begin{cases}
   h_1& \text{\rm \ if \ } [1/h] {\rm \ is \ even},\\ 
   1+h_1 & \text{\rm \ if \ } [1/h] {\rm \ is \ odd}.\\ 
  \end{cases}
 \end{equation*}
For the points $p_k$ one has
\begin{Cor} Let us fix a positive constant $C_1$ so that the error terms 
in~\eqref{as:alpha} and~\eqref{est:error:alpha} be bounded by 
$\delta_0=\lambda^\beta e^{C_1/h}$. 
If   $\delta_0/\pi \le\min\{1/3,\,h_1/2\}$, and $k$ 
satisfies~\eqref{eq:k:pk}, then there is  $p_k\in I_p$.
\end{Cor}
\begin{proof}
By Lemma~\ref{le:as:alpha}, the function  $\alpha$ is monotonous on the 
interval $[h/3,\,1/4]$, and  $[2\pi/3+\delta_0,\,\pi/(2h)-\delta_0]\subset 
\alpha\big([h/3,\,1/4]\big)$. Thus, for any $k$ satisfying 
$$\frac23+\frac{\delta_0}\pi\le k\le \frac1{2h}-\frac{\delta_0}\pi,$$
the equation 
$\alpha(p_k)=\pi k$ has a unique solution in $[h/3,\,1/4]$.
\\
As $\delta_0/\pi\le 1/3$, the the minimal possible value of $k$ equals 1. 
Recall that $1/h=[1/h]+h_1$, \ $h_1\in(0,1)$. Thus, 
one has 
\begin{equation*}
\frac1{2h}=
\begin{cases} 
\left[\frac1{2h}\right]+\frac{h_1}2, & \text{\rm \ if \ } [1/h] {\rm \ is 
\ 
even},\\ 
\left[\frac1{2h}\right]+\frac{1+h_1}2, & \text{\rm \ if \ } [1/h] {\rm \ 
is 
\ odd}, \end{cases}
\end{equation*}
and as $\delta_0/\pi\le h_1/2$, the maximal value of $k$ equals 
$\left[\frac1{2h}\right]$.
The proof is complete.
\end{proof}
\begin{Rem}
As seen from the proof, either if $\left[\frac1h\right]$ is odd, or if $k< 
\left[\frac1{2h}\right]$, then the condition on $\delta_0$ can be weakened 
and 
replaced with $|\delta_0/\pi|\le 1/2$.
\end{Rem}
\noindent One has
\begin{lemma}\label{le:PkInGap}
For any $k=1,2,\dots K-1$, the point  $E_k=2\cos(2\pi p_k)$ is located 
in a gap. The point  $E_K$ is in a gap if either
$\left[\frac1{h}\right]$ is odd, or  $\left[\frac1{h}\right]$ is even, 
and $\lambda^{h_1}<e^{-\frac{c_1}h}$ where $c_1>0$ is a certain constant.
\end{lemma}
\begin{proof} 
In view of~\eqref{cond:gap:1} and the definition 
of $p_k$, see~\eqref{def:Ek}, it suffices to prove that
$X(p_k)\ge C_0 \lambda_1$, where $\lambda_1=\lambda^{1/h}$, and $C_0>0$ is 
certain constant.
If the constant  $c_0$ in Hypothesis~\ref{H} is sufficiently large, then 
using Lemma~\ref{le:as:x} and~\eqref{as:pk}, we get 
\begin{equation}
\label{eq:Xpk}
X(p_k)=e^{\frac{2p_k\ln\lambda}h+O\left(\frac1h\right)}=
e^{k\ln\lambda+O_H(k\lambda^\beta\ln\lambda)+  
O\left(\frac1h\right)}=\lambda^k e^{O\left(\frac1h\right)}.\\
\end{equation}
So, $\frac{\lambda_1}{X^2(p_k)}\le\lambda^{\frac1h-2k+\frac Ch}$, 
and we get 
\begin{equation}\label{ineq:lambdaX2}
\frac{\lambda_1}{X^2(p_k)}\le e^{\frac Ch}\cdot
 \begin{cases} 
 \lambda^2 & \text{\rm \ if \ }  1\le k\le K-1,
 \\
  \lambda^{1+h_1} & \text{\rm \  if \ } k=K, \text{\rm \ and \ }
  \left[\frac1{h}\right] {\rm  \ is \ odd},
\\
\lambda^{h_1} & \text{\rm \  if \ } k=K, \text{\rm \ and \ }
 \left[\frac1{h}\right] {\rm  \ is \ even}.
 \end{cases}
\end{equation} 
This implies the needed.
\end{proof}
\noindent Below, we denote by $g_k$ the gap containing $E_k$.
We have proved the statement of Theorem~\ref{th:6} on the location 
of $g_k$,  \ $k=1,2,\dots K$.
\subsubsection{The lengths of gaps}
Here we prove~\eqref{eq:gk}.
Let us fix $1\le k\le K$. To get a lower bound for  $|g_k|$, the 
length of $g_k$, we first assume that
\begin{equation*}
 |p-p_k|\le ChX(p_k)
\end{equation*} 
and prove
\begin{lemma}\label{as:AlphaXpk} 
Under hypothesis~\ref{H}, one has
\begin{equation}\label{eq:AlphaXnearPk}
X(p)=X_0(p_k) (1+O_H(\lambda^\beta)),\quad 
\sin \alpha(p)=\frac{2\pi(p-p_k)}h\;(1+O_H(\lambda^\beta )).
\end{equation} 
\end{lemma}
\begin{proof} 
As when proving Lemma~\ref{le:PkInGap}, we see that
$X(p_k)=\lambda^ke^{O(1/h)}$. Therefore, 
\begin{equation}\label{eq:ppk:1}
 |p-p_k|\le h\lambda^k e^{C/h}.
\end{equation} 
Lemma~\ref{le:as:x} and~\eqref{eq:ppk:1} imply that one has
\begin{equation*}
 X(p)=X_0(p)(1+O_H(\lambda^\beta ))=
 X_0(p_k) e^{\ln \lambda\; \lambda^k e^{C/h}}(1+O_H(\lambda^\beta )).
\end{equation*}
This implies the first formula in~\eqref{eq:AlphaXnearPk}. 
Lemma~\ref{le:as:alpha}  implies that
$\alpha(p)=\frac{2\pi (p-p_k)}h\;(1+O_H(\lambda^\beta))$.
Thus, in view of~\eqref{eq:ppk:1}, we get
\begin{equation*}
 \sin\alpha=\frac{2\pi (p-p_k)}h
 \textstyle{\left(1+O\left( 
\frac{(p-p_k)^2}{h^2}\right)\right)}(1+O_H(\lambda^\beta))=
\frac{2\pi (p-p_k)}h(1+O_H(\lambda^\beta)).
\end{equation*} 
We proved the second formula in~\eqref{eq:AlphaXnearPk}.
\end{proof}
Now, we can  readily  prove
\begin{Pro} 
Let  $1\le k\le 
K$, let  $\lambda^{h_1}e^{c_0/h}$ be sufficiently small if 
$k=K$ and  $\left[1/h\right]$ is even. Then \eqref{cond:gap:2} can be 
transformed to the form
\begin{equation}\label{cond:gap:3}
 |p-p_k|\le \frac{h}{2\pi}X_0(p_k)(1+O_H(\lambda^p))
\end{equation} 
where $p=\beta$ if  $1\le k\le K-1$ or  $\left[1/h\right]$ is odd, 
and $p=\min\{h_1,\beta\}$ if $k=K$ and  $\left[1/h\right]$ is even.
\end{Pro}
\begin{proof} The statement follows from Lemma~\ref{as:AlphaXpk} and 
formulas~\eqref{eq:Xpk} and~\eqref{ineq:lambdaX2}.
\end{proof}
Now, to complete the proof of~\eqref{eq:gk}, we need to compute 
$X_0(p_k)$ and to write~\eqref{cond:gap:3} in terms of $E$. We begin with
\begin{lemma}
For $1\le k\le K$, one has
\begin{equation}\label{formula:X0pk}
 X_0(p_k)=\left(\frac{\lambda}4\right)^k
 \frac{\sin(\pi h k)}{h\,\sin^2(\pi h k)\dots \sin^2(\pi h)} 
(1+O_H(\lambda^\beta)).
\end{equation}
\end{lemma}
\begin{proof}
By the second estimate in~\eqref{est:lnX0},
\begin{equation*}
 \textstyle X_0(p_k)=X_0\left(\frac{hk}2\right)\,(1+O_H(\lambda^\beta))
\end{equation*} 
as $\beta$ in Lemma~\ref{le:as:x} can be chosen greater 
then in this lemma.
Let us recall that $X_0$ is defined in~\eqref{as:x} with $F_0$ given 
by~\eqref{def:F0}. Using~\eqref{eq:sigma} and~\eqref{sigma:res}, 
we get
\begin{equation*}\textstyle
F_0\left(\frac{hk}2\right)=
\prod_{\kappa=1}^{k} |1+e^{-i(2\pi h\kappa-\pi)}|^2
|\sigma_{\pi h}(\pi h-\pi)|^2=
4^k\,h\,\prod_{\kappa=1}^{k} |\sin(\pi h\kappa)|^2.
\end{equation*} 
This formula and formulas~\eqref{as:x} and~\eqref{def:F0}
imply~\eqref{formula:X0pk}.
\end{proof}
\noindent Finally, we rewrite~\eqref{cond:gap:3} in terms of $E$.
We denote the right hand side in~\eqref{cond:gap:3} by $\Delta_k$.
As $E=2\cos (2\pi p)$, we get
\begin{equation*}
|g_k|\ge  
2(\cos\big(2\pi(p_k-\Delta_k)\big)-\cos\big(2\pi(p_k+\Delta_k)\big)=
          4\sin(2\pi p_k)\sin (2\pi \Delta_k).
\end{equation*}
As $X_0(p_k)=O_H(\lambda^k )$,  we get $|\Delta_k|=O_H(\lambda^k),$ 
and, by means of~\eqref{as:pk}, we obtain
$$|g_k|\ge 8\pi\Delta_k \,\sin(2\pi p_k)(1+O_H(\lambda^{2k}))=8\pi\Delta_k 
\,\sin(\pi hk)(1+O_H(\lambda^{\beta})).$$
Substituting into this estimate instead of $\Delta_k$ the right hand side 
from~\eqref{cond:gap:3} and using formula~\eqref{formula:X0pk}, we come 
to~\eqref{eq:gk}.
This completes the proof Theorem~\ref{th:6}.\qed

\smallskip
We will use the following immediate corollary from the above 
proof:
\begin{Cor}\label{cor:gk:roughly}
 The right hand side in~\eqref{eq:gk} satisfies equals to 
$\lambda^ke^{O(1/h)}$.
\end{Cor}
\begin{proof} As $X_0(p_k)=O_H(\lambda^k )$, it suffices to 
compare~\eqref{eq:gk} with~\eqref{formula:X0pk}.\end{proof}
\section{Model equation}\label{model:eq}
Entire solutions to the equation 
\begin{equation} \label{mu:eq}
\mu\,(z+h)+\mu\,(z-h)+e^{-2\pi i z}\mu(z)=2\cos(2\pi p)\,\mu(z),\quad
z\in \C.
\end{equation}
were constructed in~\cite{F:18a}.
Here, we briefly recall  construction of these solutions, and get  
for them  estimates uniform in $h$. 
\subsection{Construction of solutions}
\subsubsection{Integral representation}
We construct solutions   in the form:
\begin{equation}\label{mu:form}
\mu\,(z)=\frac1{\sqrt{h}}\int_{\gamma}e^{\frac{2\pi izk}h}\, v\,(k)\,dk,
\end{equation}
where $\gamma$ is a curve in the complex plane that we
describe later, and $v$ is a function analytic in a sufficiently 
large neighborhood of $\gamma$. The function  $\mu$ 
satisfies~\eqref{mu:eq} if
\begin{eqnarray*}
v(k+h)=2(\cos (2\pi p)-\cos (2\pi k))\,v(k).
\end{eqnarray*}
We note that
\begin{equation*}
2(\cos(2\pi p)-\cos (2\pi k))=-e^{2\pi ik}\,(1-e^{-2\pi 
i(k-p)})\,(1-e^{-2\pi i(k+p)}).
\end{equation*}
Therefore, one can choose
\begin{equation}\label{v:form}
v(k)=e^{\frac{i\pi k^2}{h}-\frac{\pi i k}{h}-\pi ik}\, 
\sigma_{\pi h}(2\pi(k-p-{\textstyle \frac h2-\frac12}))\,
\sigma_{\pi h}(2\pi(k+p-{\textstyle \frac h2 -\frac12})),
\end{equation}
where $\sigma_a$ is a solution to~\eqref{eq:sigma}. 
In this paper, we use the meromorphic solution described in 
Section~\ref{sec:sigma}.  
\subsubsection{Properties of the function $v$}
As $\sigma_a$ the function $v$ is meromorphic.
In Section~\ref{sigma:poles,zeros} we list all the poles of $\sigma_a$.
This description implies  that the poles of $v$  are located at the points
\begin{equation}\label{v:poles}
k=\pm p-lh-m,\quad l,m=0,1,2,3\dots.
\end{equation}
Let $l_\pm=\pm p-(-\infty,0]$. The rays $l_\pm $ contain 
all the poles of $v$.
\\
Below, for  $k\in\C$, we write $r=\re k$ and $q=\im k$.
\\
Let $|q|> C\,|r|$. Corollary~\ref{cor:sigma:2} implies the uniform
asymptotic representations
\begin{gather}
\label{v:down}
v(k)= v_-\,e^{\frac{i\pi\left(k-{\textstyle \frac12-\frac h2} 
\right)^2}{h}+o(1)},\quad
v_-=-ie^{-\frac{i\pi}{4 h}-\frac{\pi ih}4},\qquad q\to-\infty,
\\
\label{v:up}
v(k)=v_+\,e^{-\frac{i\pi\left(k-{\textstyle \frac 12-\frac h2} 
\right)^2}{h}+o(1)}, \quad 
v_+=-ie^{-\frac{2\pi ip^2}{h}-\frac{i\pi}{12h}-\frac{\pi ih}{12}},\qquad
q\to+\infty.
\end{gather}  
\subsubsection{Integration path}
We  choose the integration path  in~\eqref{mu:form}
so that it does not intersect $ l _\pm $, comes from infinity
from bottom to top along the line $ e ^ {i \pi / 4} \, \R $ and 
goes to infinity upward along the line $ e ^ {3i \pi / 4} \, \R $.
This completes the construction of $\mu$.
\subsubsection{Notations}
Fix $X>0$. We study   $\mu$ in the strip
$|x|\le X$ assuming that $p\in P$,
\begin{equation*} 
P=\{p\in\C\,:\,0\le \re p\le 1/2, \,|\im p|\le h\}.
\end{equation*}
\subsection{Estimates in the upper half-plane}
Let
\begin{equation*}
\xi_\pm(z)=e^{\pm\frac{\pi iz^2}{h}+ \frac{i\pi z}h+\pi iz}.
\end{equation*}
One has
\begin{Pro}\label{pro:mu:up} Let us pick $X, Y>0$. Assume that $p\in P$.
One has
\begin{equation}\label{mu:up}
\mu=\nu_+\,\xi_++\nu_-\,\xi_-,\qquad
\nu_\pm=\mp e^{\mp\frac{i\pi}4}\,v_\pm\,F_\pm;
\end{equation}
and  for $|x|\le X$ the functions $F_\pm$ satisfy the uniform in $x$ 
estimates:
\begin{equation}\label{mu:up:Delta}
F_\pm=O(H)\ \text{ if } \ y\ge -Y, \qquad \text{and}\qquad 
F_\pm=1+o(1) \ \text{ as } \  y\to+\infty.
\end{equation}
\end{Pro}
\begin{proof}
In view of~\eqref{v:down}--~\eqref{v:up},  as $q\to\pm \infty$ 
the behavior of the integrand in~\eqref{mu:form} is described by the 
exponentials
\begin{equation}\label{exps}
e^{\frac{2\pi ik z}h}\,e^{\mp\frac{i\pi (k-{\textstyle 
\frac12}-{\textstyle 
\frac h2} )^2}{h}}=
\xi_\pm(z)\,e^{\mp\frac{\pi i(k-k_\pm(z))^2}{h}},
\quad\text{where}\quad
k_\pm(z)=\pm z+{\textstyle \frac h2+\frac12}.
\end{equation}
Let us consider the straight lines
\begin{equation*}
L_\pm(z)=k_\pm(z)+e^{\mp\pi i/4}\R.
\end{equation*}
The lines $L_\pm(z)$ are lines of steepest descent for the functions  
$k\mapsto|e^{\mp\frac{\pi i(k-k_\pm(z))^2}{h}}|$.
They intersect one another at $k_*=y+ix+h/2+1/2$. 

Let us pick $d_0>0$. There is an $Y>0$ such that if $y>Y$, then for all 
$(h,p,x)\in (0,1)\times P\times [-X,X]$, the distance from $L_\pm$ to  
$l_\pm$  is greater than $d_0$. 

First, we assume that $y>Y$. In this case, we choose the integration path 
$\gamma$ in~\eqref{mu:form} so that it go upwards first   along $L_-(z)$ 
from infinity to $k_*$ and then along $L_+(z)$ from $k_*$ to infinity.  
We denote by $\gamma_-$ ($\gamma_+$), the part of $\gamma$ below (resp., 
above) 
$k_*(z)$. We define two functions $\mu_\pm$  by same the formula as $\mu$,
i.e. by~\eqref{mu:form}, but with the integration path $\gamma$ replaced 
with
$\gamma_\pm$.  It suffices to show that
\begin{equation}\label{mu:pm}
\mu_\pm=\mp e^{\mp \frac{i\pi}4}\,v_\pm\,\xi_\pm\,F_\pm,
\end{equation}
with $F_\pm$ satisfying the estimates~\eqref{mu:up:Delta}. We
prove~\eqref{mu:pm} only for $\mu_-$\;; \ $\mu_+$ is estimated similarly.
Substituting~\eqref{v:form} into~\eqref{mu:form}
 and using~\eqref{exps}, we get
\begin{equation}\label{mu:up:-:1}
\mu_-=\xi_-(z)\,v_-\,\nu,\quad\quad \nu=\frac1{\sqrt{h}}
\int_{\gamma_-}\,e^{\frac{\pi i(k-k_-(z))^2}{h}}\,F(k,p,h)\,dk,
\end{equation}
where $F(k,p,h)= \sigma_{\pi h}(2\pi(k-p- {\textstyle 
\frac{h}2-\frac12}))\, 
\sigma_{\pi h}(2\pi(k+p-{\textstyle\frac{h}2-\frac12}))$. 
\\
We remind that $k=r+iq$, \ $r,q\in\R$. Let $0<\kappa<1$. By 
Corollary~\ref{cor:sigma:2}, for $k\in\gamma_-$,
\begin{equation}\label{F:1}
  F(k,p,h)=e^{O\left(h^{-1}e^{-2\pi \kappa\,|q|}\,(1+|r|)\right)}.
\end{equation}
As along $\gamma_-$ one has 
\begin{equation*}
k-k_-(z)=\sqrt{2}e^{i\pi/4}\im (k-k_-(z)),\quad \im (k-k_-(z))=q+y,\quad 
|r|\le C+|q+y|,
\end{equation*}
formulas~\eqref{mu:up:-:1} and~\eqref{F:1} imply that
\begin{equation}\label{eq:nu:up:1}
\begin{split}
\nu=e^{\frac{i\pi}4}\sqrt{\frac{2}{h}}
\int_{-\infty}^{x}&
e^{-\frac{2\pi(q+y)^2}{h}+O\left(\frac{e^{-2\pi 
\kappa|q|}\,(1+|q+y|)}h\right)}\,dq\\ &=
e^{\frac{i\pi}4}\sqrt{\frac{2}{h}}
\int_{-\infty}^{x+y}
e^{-\frac {2\pi 
t^2}{h}+O\left(\frac{e^{-2\pi\kappa|t-y|}\,(1+|t|)}h\right)}\,dt.
\end{split}
\end{equation}
Therefore,
\begin{equation*}
  |\nu|\le \frac{C}{\sqrt{h}}\int_{-\infty}^{\infty}
e^{-\frac {2\pi t^2}{h}+\frac{C(1+|t|)}h}\,dt\le H.
\end{equation*}
This proves the first estimate in~\eqref{mu:up:Delta} for $y>Y$. 
To prove the second one, we note that, for sufficiently large $y$, one has
\begin{equation*} 
e^{-2\pi \kappa|t-y|}\,(1+|t|)\le Cy 
\begin{cases}
  e^{-\pi \kappa y} &
\text{if \ }  t\le \frac{y}2,\\
1 & \text{if \ }  \frac{y}2\le t\le y+X.
\end{cases}
\end{equation*}
So,  representation~\eqref{eq:nu:up:1} implies that as $y\to\infty$
\begin{equation*}
\nu=e^{\frac{i\pi}4}\sqrt{\frac{2}{h}}\left(
\int_{-\infty}^{\frac{y}2}
e^{-\frac {2\pi t^2}{h}+O\left(e^{-\pi \kappa y}y/h\right)}\,dt+
\int_{\frac{y}2}^{x+y}
e^{-\frac {2\pi t^2}{h}+O(y/h)}\,dt\right)=e^{i\pi/4}+o(1).
\end{equation*}
This proves the second estimate in~\eqref{mu:up:Delta}.

To complete the proof, it suffices to check that
if $ |x|\le X$ and $|y|\le Y$, then $\mu\le H$. 
In this case, we pick  $\delta,r>0$ and choose the integration 
path $\gamma$  in~\eqref{mu:form} that goes along $L_-(z)$ upwards   
from infinity to the circle $c_r$ with radius $r$ and center at $k_*$,
then along $c_r$ in the anticlockwise direction to the upper point of 
intersection
of   $L_+(z)$ and $c_r$, and finally along $L_+(z)$ upwards from this 
point 
to
infinity. We assume that $r$ is sufficiently large so that the distance 
between
$\gamma$ and the rays $l_\pm$ be greater than $\delta$.
In view of Corollary~\ref{cor:sigma:2}, on $\gamma\cap c_r$ the integrand 
in~\eqref{mu:form} is bounded by $H$. On $\gamma\setminus c_r$ it is 
estimated 
 as when proving the first estimate in~\eqref{mu:up:Delta}.
\end{proof}
\subsection{Estimates in the lower half-plane}
Set
\begin{equation}\label{a(p,h)}
a(p)=e^{-\frac{i\pi}{12h}+\frac{i\pi}4-\frac{\pi ih}{12}} \ 
e^{\frac{\pi ip^2}{h}-\frac{ip\pi}h-\pi ip}\,
\sigma_{\pi h}(4\pi p-\pi h-\pi)).
\end{equation}
We note that $a$ is meromorphic in $p$, and its poles in $(-\infty,0]$ .
One has
\begin{Pro}\label{pro:mu:down} Pick positive $X$. 
Let $p\in P$. 
There is an $Y>0$ independent of $p$ and $h$ and such that, for $y\le -Y$ 
and $|x|\le X$, one has the following results
\\ 
$\bullet$ \ Fix $\alpha$, $\beta$  so that $0<\beta<\alpha<1$. Then,
\begin{equation}\label{mu:down:1}
\mu(z)= e^{\frac{2\pi ip z}h}\left( a(p)+
O_H\big(\,e^{-2\pi \beta |y|}\,\big)\right),\quad
{\rm if}\quad \re p\ge \alpha h/2,\quad
\end{equation}
\\
$\bullet$ \ Fix $\alpha$ and $\beta'$ so that $0<\alpha<1$ and 
$0<\beta'<1$. Then
\begin{equation}\label{mu:down:2}
\begin{split}
\mu(z)= a(p)\,e^{\frac{2\pi ip z}h}+a(-p)\,e^{-\frac{2\pi ip z}h}&+
O_H\left(\,e^{\frac{2\pi ip z}h-2\pi \beta'|y|}\,\right),\\
&{\rm if}\quad 0\le \re p\le \alpha h/2.
\end{split}
\end{equation}
\end{Pro}
\begin{proof}  Let us begin with justifying~\eqref{mu:down:1}.
Remind that $\mu$ has the integral representation~\eqref{mu:form}. For
$y<0$, the behavior of $\mu$ appears to be  determined by the rightmost 
poles of $v$. 

The poles of $v$ are at the points listed in~\eqref{v:poles}.
As $p\in P$, they are inside the strip $|\im k|\le h$. As $\re p\ge
\alpha h/2$, \ $0<\alpha<1$, we see that, to the right of the line $\re 
k=\re
p-\alpha h$, the function $v$ has only one simple pole; it is situated at
$k=p$.

We deform $\gamma$, the integration contour in~\eqref{mu:form},
to a curve  that    has the same asymptotes as $\gamma$, but, instead of 
staying to the right of all the poles of $v$, it goes
between the pole at $k=p$ and the other ones (they stay to the left of this
curve). We keep for the new integration curve the old notation $\gamma$. 
One has
\begin{equation}\label{mu:res:1}
\begin{split}
&\mu= A+B, \\
A= \frac{2\pi i}{\sqrt{ h}}\, \res_{k=p}\,I(k),\quad 
&B= \frac1{\sqrt{h}}\,\int_{\gamma}I(k)\,dk,\qquad 
I=e^{\frac{2\pi ikz}h}\, v\,(k).
\end{split}
\end{equation}
Using the representation~\eqref{v:form}, the information on the poles of 
the
function $\sigma_{\pi h}(z)$ from section~\ref{sigma:poles,zeros},
and formula~\eqref{sigma:res}, we get
\begin{equation*}
A=a(p)\,e^{\frac{2\pi ip z}h}
\end{equation*}
with $a(p)$ given by~\eqref{a(p,h)}.

Now, to complete the proof of the proposition, we need  only to
estimate the term $B$  in~\eqref{mu:res:1}.
Let  $\gamma_+$ ($\gamma_-$) be the part of $ \gamma$ located above 
(resp., 
below) the line $\im k=\im p$. First, we  choose $\gamma_\pm$, and then, 
we  prove that
\begin{equation}\label{gamma1gamma2}
\left|\int_{\gamma_\pm}I(k)\,dk\right|\le 
H\,\left|e^{\frac{2\pi ip z}h}\right|\, e^{-2\pi \beta  |y|}.
\end{equation}

We begin with estimating the integral along $\gamma_-$.
We remind that the  exponential $e^{\frac{\pi i(k-k_-(z))^2}{h}}$
governs the behavior of $I(k)$, the integrand in~\eqref{mu:form}, 
as  $\im k\to-\infty$, see the beginning of the proof of 
Proposition~\ref{pro:mu:up}. We assume that  $y<-Y$ with 
a positive $Y$. Therefore, $\im k_-(z)>Y$.

For $\xi\in\C$, we denote by  $H_i(\xi)$ the smooth curve
described by an  equation of the form $\im (k-k_-(z))^2=c$, \ $c\in\R$,  
and containing $\xi$.
If $c=0$ this curve is one of the straight lines $k_-(z)+\R$ and 
$k_-(z)+i\R$,
otherwise it is a hyperbola located in one of the sectors bounded by these
lines. Its asymptotes are two half lines of these straight lines.
Let $H_r(\xi)$ be the smooth curve described  by an  equation of the form
$\re (k-k_-(z))^2=c$, \ $c\in\R$, and containing  $\xi$. If $c=0$ this 
curve is
one of the straight lines $k_-(z)+e^{\pm \frac{i\pi}4}\R$, 
otherwise it is a hyperbola located in one of the sectors bounded by 
these lines, and they are its 
asymptotes.

Set $k_1=p-\beta h$. As $0<\beta<\alpha<1$, the point $k_1$
is to the right of the line $\re k=p-\alpha h$ and to the left of $p$. As
$p\in P$, one has $|\im k_1|\le h<1$.

Let $Y$ be sufficiently large, and $y<-Y$. Then  the hyperbola $H_i(k_1)$
stays in the half plane $\im k< \im k_-(z)$ and intersects the line $\im 
k=-2$
at a point $k_2$.We denote by $\gamma_1$ its segment of $H_i(k_1)$ between
$k_1$ and $k_2$.
\\
Furthermore, if $Y$ is sufficiently large, $H_r(k_2)$ is a hyperbola
located below $k_-(z)$. We denote by  $\gamma_2$   its segment between
$k_2$ and $\infty$ along which $\re k\to -\infty$ as $k\to\infty$.
The curve $\gamma_-$ is the union of  $\gamma_1$  and $\gamma_2$,
see Fig.~\ref{fig:1}.

\begin{figure}[t!]\centering
  \includegraphics[width=0.5\textwidth]{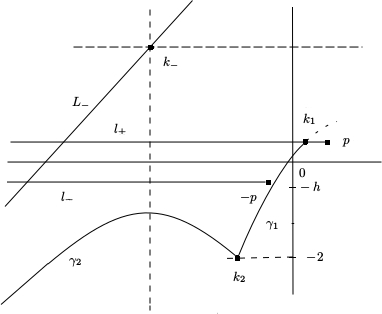}
  \caption{Curves $\gamma_1$ and $\gamma_2$}
\label{fig:1}
\end{figure} 

If $Y$ is sufficiently big, then (1) the curve
$\gamma_-$ does stays between $p$ and all the other poles of $v$; (2) its
segment $\gamma_2$   is located below the poles of the integrand at a 
distance
greater than $1$.
\\
Let us estimate $\int_{\gamma_1}I(p)\,dp$.
We note that, by the definition of $H_i$ and by~\eqref{exps},
the expression $\im\left(2k z+k^2-k-kh\right)$ is constant on
$H_i(k_1)$. By~\eqref{v:form} we get
%
\begin{multline}\label{int-term:1}
\left|\frac1{\sqrt{h}}
\int_{\gamma_1}e^{\frac{2\pi izk}h}v(k)dk\right|
\le 
\frac{C\left|e^{-\frac{C(H_i)}{h}}\right|}{\sqrt{h}}\times\\
 \times\int_{\gamma_1} |\sigma_{\pi h}(2\pi(k+p-{\textstyle \frac 
h2-\frac12})) 
\sigma_{\pi h}(2\pi(k-p-{\textstyle \frac h2-\frac12}))dk|
\end{multline}
%
where
\begin{equation*}
C(H_i)=\pi \left. \im \left(2kz+k^2-k-kh\right)
\right|_{H_i(k_1)},\quad
\end{equation*}
Computing $C(H_i)$ at the point $k_1$, we get
\begin{equation}\label{C(H)}
\left|e^{-\frac{C(H_i)}{h}}\right|\le C\,\left|e^{\frac{2\pi ip 
z}h}\right|\,
e^{-2\pi \beta  |y|}.
\end{equation}
Let us estimate the integrand in the right hand side
of~\eqref{int-term:1}.  Using~\eqref{eq:sigma}, we get
\begin{equation*} 
\sigma_{\pi h}(2\pi(k-p-{\textstyle\frac h2-\frac12}))=
\frac{\sigma_{\pi h}(2\pi(k-p+{\textstyle \frac h2-\frac12}))}{1-e^{-2\pi 
i(k-p)}}. 
\end{equation*}
For $k\in\gamma_1$, one has
\begin{equation*}
\begin{split}
\re  (k+p-h/2 -1/2)&\ge -1/2-h/2+2\re p  - \beta h - C(Y) |\im (k-p)|\\
&\ge -1/2-h/2+(\alpha - \beta) h - C(Y) |\im (k-p)|,
\\
\re (k-p+h/2 -1/2)&\ge  -1/2-h/2+(1-\beta) h-C(Y) |\im (k-p)|.
\end{split}
\end{equation*}
where $C(Y)>0$ tends to zero as $Y\to\infty$.
These observations and Corollaries~\ref{cor:sigma:2}--~\ref{cor:sigma:3}
imply that, for sufficiently large $Y$ and $k\in\gamma_1$, 
\begin{equation*}
|\sigma_{\pi h}(2\pi(k+p-{\textstyle \frac h2-\frac12}))\, 
\sigma_{\pi h}(2\pi(k-p-{\textstyle\frac h2-\frac12}))|\le H.
\end{equation*}
This estimate and~\eqref{C(H)} imply estimate~\eqref{gamma1gamma2}
with $\gamma_1$ instead of $\gamma_\pm$.

Consider the integral $\int_{\gamma_2}I(k)\,dk$.
As $\gamma_2$ stays below the poles of $I$, at a distance greater than 
$1$, 
by means of Corollary~\ref{cor:sigma:2}, one immediately obtains
\begin{equation*}
|\sigma_{\pi h}(2\pi(k+p-{\textstyle \frac h2-\frac12} ))\, 
\sigma_{\pi h}(2\pi (k-p-{\textstyle\frac h2-\frac12} ))|\le  H,\quad 
k\in\gamma_2,
\end{equation*}
and
\begin{equation*}
\left|\int_{\gamma_2} I\,dk\right|\le H\,
\left|e^{\frac{\pi i}h\,\left.(2kz+k^2- k-kh)\right|_{k=k_2}}
\right|\,\int_{\gamma_2}\left|e^{\frac{\pi i}h 
((k-k_-(z))^2-(k_2-k_-(z))^2)} \,dk \right|.
\end{equation*}
Clearly,
\begin{equation*}
\left|e^{\frac{\pi i}h 
\left.(2kz+k^2-k-kh)\right|_{k=k_2}}\right|=e^{-\frac{C(H_i)}h}\le
C\,\left|e^{\frac{2\pi ip z}h}\right| e^{-2\pi\beta  |y|}.
\end{equation*}
We remind that curve $\gamma_2\subset H_r(k_2)$ goes to infinity 
approaching the 
asymptote $e^{i\pi/4}(-\infty,0]$. Integrating by parts, we get
\begin{equation*}
\int_{\gamma_2} \left|e^{\frac{\pi i}h ((k-k_-(z))^2-(k_2-k_-(z))^2)} 
\,dk\right|\le 
Ch/|k_2-k_-|.
\end{equation*}
These estimates imply that $\int_{\gamma_2} I\,dk$ satisfies an estimate 
of 
the 
form~\eqref{gamma1gamma2}. This implies~\eqref{gamma1gamma2} with
$\gamma_-$.

The estimates of the integral along $\gamma_+$, the part of
$\gamma$  above the line $\im k=\im p$ are similar. We omit the
details and mention only that now the role of  $e^{i\frac{(k-k_-(z))^2}h}$
is played by the exponential $e^{-i\frac{(k-k_+(z))^2}h}$.
The obtained estimates and the formula for $A$ imply~\eqref{mu:down:1}.  
This completes the proof of~\eqref{mu:down:1}.
Representation~\eqref{mu:down:2} is obtained similarly.
\end{proof}
%
\subsection{Rough estimates}
%
We shall need
\begin{lemma}\label{mu:estimates} Pick $X>0$. Let $p\in P$
and $|x|\le X$. One has
\begin{equation}\label{mu-est}
|\mu(z)|,\; |\mu'(z)|\le H w(x,y),\quad
w(x,y)=(1+|y|)\begin{cases} e^{\frac{2\pi(|x|-1/2)y}h-\pi y}, & y\ge
0,\\ e^{\frac{2\pi \re p |y|}h}, &  y\le 0.
\end{cases}
\end{equation}
\end{lemma}
When proving this lemma we use
\begin{lemma}
Pick $\alpha\in (0,1)$. For  $ p\in P$ one has 
\begin{equation}
  \label{est:a0}
  |a(p)|\le H\quad\text{if}\quad |p|\ge \alpha h/2,
\end{equation}
\begin{equation}\label{est:a_0}
|a(p)+a(-p)|\le H,\quad   |p a(p)|\le H\quad\text{if}\quad
|p|\le \alpha h/2.
\end{equation}
\end{lemma}
\begin{proof} Let $|p|\ge \alpha h/2$ and $p\in P$. 
Corollaries~\ref{cor:sigma:2} and~\ref{cor:sigma:3} imply that 
$|\sigma_{\pi h}(4\pi p-\pi h-\pi))|\le H$.
This and~\eqref{a(p,h)} lead to~\eqref{est:a0}.

Assume that $|p|\le\alpha h/2$ ($p$ is not necessarily in $P$).
In view of~\eqref{a(p,h)}, it suffices to check that
\begin{equation*}
|\sigma_{\pi h}(4\pi p-\pi h-\pi)+
\sigma_{\pi h}(-4\pi p-\pi h-\pi)|\le \frac{C}h,\quad 
|p \sigma_{\pi h}(4\pi p-\pi h-\pi)|\le C.
\end{equation*}
Both the functions 
$p\mapsto \sigma_{\pi h}(4\pi p-\pi h-\pi)+
\sigma_{\pi h}(-4\pi p-\pi h-\pi)$ 
and $p\mapsto p\sigma_{\pi h}(4\pi p-\pi h-\pi)$
are analytic in $p$ in the $\frac{\alpha h}2$-neighborhood of zero (see 
section~\ref{sigma:poles,zeros}). By Theorem~\ref{theta:uni-rep:2}
$\sigma_{\pi h}(2\pi(2p-{\textstyle\frac{h}2-\frac12}))$ is bounded by 
$C/h$ 
at the boundary of this neighborhood. This and the maximum principle 
imply the needed estimates.
\end{proof}
Now we can check Lemma~\ref{mu:estimates}.
\begin{proof}  Let $p\in P$. Pick $Y$ sufficiently large.
For $y\ge - Y$,  the estimate for $\mu$ follows directly from
Proposition~\ref{pro:mu:up}.
For $|y|\le Y$, the estimate for $\mu'$ follows from the estimate for $\mu$
and the Cauchy estimates for the derivatives of analytic functions
(as $X$ and $Y$ were chosen rather arbitrarily).
Let $y\ge Y$ and $|x|\le X$. 
The first estimate  in~\eqref{mu-est} implies that, in the 
$(h/y)$-neighborhood 
of $z$, $\mu$ is bounded by $H w(x,y)$ (we again use the fact
that $X$ and $Y$ were chosen rather arbitrarily).  
This and  the Cauchy estimates for the derivatives of analytic functions 
lead then
to the estimate $|\mu'(z)|\le yH w(x,y)$.  This completes the proof 
of~\eqref{mu-est}
for $y\ge -Y$.

Let us prove~\eqref{mu-est} for $y\le -Y$. Pick $0<\alpha<1$. Let
$|x|\le X$,  \  $y\le -Y$ and
$\re p\ge \alpha h/2$. Estimate~\eqref{est:a0} and 
Proposition~\ref{pro:mu:down} imply that  
\begin{equation}\label{est:mu:rough:1}
|\mu(z) e^{-\frac{2\pi ipz}h}|\le H.
\end{equation}
By means of the Cauchy estimates for the derivatives of analytic functions 
we get the estimate
\begin{equation}\label{est:mu:rough:2}
\left|\frac{d}{dz}(\mu(z) e^{-\frac{2\pi ipz}h})\right|\le H.
\end{equation}
Estimates \eqref{est:mu:rough:1}--\eqref{est:mu:rough:2} lead 
to~\eqref{mu-est}
for $|x|\le X$, $y\le -Y$ and $\re p\ge \alpha h/2$.
If $0\le \re p\le \alpha h/2$, the estimates for $\mu$ and its derivative 
are 
deduced from~\eqref{mu:down:2} and~\eqref{est:a_0} similarly. 
We omit further details.
\end{proof}
%
\subsection{One more solution to the model equation}
%
Let
\begin{equation}\label{tildeMu}
\tilde\mu\,(z,p)=e^{-i\pi z/h}\,\mu\,(z+1/2,1/2-p),
\end{equation}
were we indicate the dependence of $\mu$ on $p$ explicitly.
Together with $\mu$, \ $\tilde \mu$ is  a solution to \eqref{mu:eq}.
It is entire in $z$ and $p$. We use it to construct entire solutions to 
the 
Harper equation.
Here, we compute the Wronskian $\{\mu,\,\tilde \mu\}=
\mu(z+h) \tilde \mu(z)-\mu(z)\tilde \mu(z+h)$. 
\begin{lemma}
For all $p\in\C$, 
\begin{equation}\label{wronskian:formula}
\{\mu,\,\tilde \mu\}=
ie^{-\frac{2\pi i p^2}h-\frac{i\pi}{12h}-\frac{\pi ih}3}.
\end{equation}
\end{lemma}
\begin{proof} Pick $X>0$ and $0<\alpha<\frac12$. 
Assume that $\alpha h/2\le p\le1/2-\alpha h/2$. 
Then, by means of~\eqref{tildeMu},~\eqref{mu:down:1} and~\eqref{est:a0}, 
we check that, uniformly in $x\in [-X,X]$, as $y\to-\infty$ the Wronskian
$\{\mu,\,\tilde \mu\}$ tends to
\begin{equation}\label{eq:2:w}
2i e^{\frac{i\pi}{2h}-\frac{i\pi p}h}\,a(1/2-p)a(p)\sin(2\pi p) .
\end{equation}
By means of~\eqref{mu-est}, 
we check that, as $y\to+\infty$, uniformly in $x\in[-h,0]$, 
\begin{equation*}
|\{\mu(z),\,\tilde \mu(z)\}|\le H |y|
\left( e^{\frac{2\pi y}h (x-\frac12+|x+1/2|)}+1\right)\le H |y|.
\end{equation*}
The Wronskian  being entire (as  $\mu$ and $\tilde \mu$) and  
$h$-periodic in $z$ (as the Wronskian of any two solutions to a 
one-dimensional difference Schr\"odinger equation, see 
section~\ref{sss:MM:diffSch}),
we conclude that it is independent of  $z$ and  equals  the expression 
in~\eqref{eq:2:w}. As the Wronskian is entire in $p$, this 
statement is valid for all $p$. 
Finally, using the definition of $a$, equation~\eqref{eq:sigma} and 
formula~\eqref{sigma:sym}, we check that
\begin{equation}\label{wr-relation}
2i \,a(1/2-p)a(p)\sin(2\pi p)=
ie^{-\frac{2\pi i p^2}h+\frac{i\pi p}h-\frac{7i\pi}{12h}-\frac{\pi ih}3}.
\end{equation}
This leads to the statement of  the lemma.
\end{proof}
%
%
\section{Analytic solution to Harper equation}
\label{sec:int-eq}
%
%
\subsection{Preliminaries}
%
%
For $Y\in\R$, we  set $\C_+(Y)=\{y> Y\}$. 
Here, we pick $Y>0$ and for sufficiently small $\lambda>0$ 
construct a solution to~\eqref{eq:harper} analytic in $\C_+(-Y)$. 

Below, we represent the spectral parameter $E$ in the form $E=2\cos (2\pi 
p)$
and consider solutions to~\eqref{eq:harper} as functions of the parameter 
$p$.

As $\lambda$ is small, then, when constructing solutions
to~\eqref{eq:harper} in $\C_+(-Y)$, it is natural to rewrite 
this equation in the form
\begin{equation} \label{pert-model}
\psi\,(z+h)+\psi\,(z-h)+\lambda\,e^{-2\pi iz}\psi\,(z)-2\cos(2\pi 
p)\,\psi(z)=
-\lambda\,e^{2\pi iz}\,\psi(z),
\end{equation}
so that the term in the right hand side could be considered as a 
perturbation.
Let $\xi=\frac1{2\pi}\ln\lambda$. Then $\mu(z+i\xi)$
is a solution to the unperturbed equation
\begin{equation}\label{un-pert-eq}
\psi\,(z+h)+\psi\,(z-h)+\lambda\,e^{-2\pi iz}\psi\,(z)-2\cos(2\pi 
p)\,\psi(z)=0.
\end{equation}
We construct $\psi$, an analytic solution
to equation~\eqref{pert-model}  close to
$\mu(z+i\xi)$. We prove
\begin{theorem}\label{psi0} Pick positive $Y$. 
There exists $C$ such that if \ 
$\lambda\le e^{-\frac Ch}$, \ then:
\point There is $\psi_0$, a solution to~\eqref{eq:harper}  analytic
in $(z,p)\in \C_+(-Y)\times P$;
\point Pick positive $X$. As $y\to+\infty$, uniformly in $x\in[-X,X]$
\begin{equation}\label{psi-0:up}
\psi_0(z)=e^{\frac{i\pi}4}\,\xi_-(z+i\xi)\,v_-\,(1+\varkappa_0 +o(1))
-e^{-\frac{i\pi}4}\,\xi_+(z+i\xi)\,v_+\,(1+o(1)),
\end{equation}
where $\varkappa_0$ is a constant satisfying the estimate
\begin{equation*}\dsize
|\varkappa_0|\le H\lambda^{1+\nu(p)}(1+|\xi|)^3,\quad
\nu(p)=\min\left\{1,\frac{1-2\re p}h\right\}.
\end{equation*}
\point Fix $X>0$. For $|x|\le X$ and $y\le Y/2$, one has
\begin{equation}\label{psi-0:down}
|\psi_0(z)-\mu(z+i\xi)|\le H\lambda (1+|\xi|)^3\,e^{2\pi \re p|\xi|/h}.
\end{equation}
\end{theorem}
The rest of the section is devoted to the proof of Theorem~\ref{psi0}.

Let us explain the idea of the proof of Theorem~\ref{psi0}.
Let $\gamma=\{z\in i\R\,:\, y\ge -Y\}$.
We study  the integral equation 
\begin{equation}\label{int-eq}
\psi\,(z)=\mu\,(z+i\xi)-(K\psi)(z),\quad K\psi(z)=\int_{\gamma}
\kappa\,(z,\,z')\psi\,(z')\,dz',\quad z\in \gamma.
\end{equation}
The kernel $\kappa$  is constructed in terms of  $\mu(\cdot+i\xi)$ and 
$\tilde \mu(\cdot +i\xi)$, two
linearly independent solutions to  the unperturbed 
equation~\eqref{un-pert-eq},
\begin{gather*}
\kappa\,(z,\,z')= \frac {\lambda}{2ih}\,\Theta\,(z,\,z')\, \frac{
        [\,\mu(z+i\xi)\,\tilde \mu\,(z'+i\xi)-
         \mu\,(z'+i\xi)\,\tilde \mu\,(z+i\xi)\,]}
        {\{\mu,\,\tilde \mu\}}\,e^{2\pi iz'},
\\
\Theta\,(z,\,z')=\ctg\frac{\pi(z'-z)}h\,-i.
\end{gather*}
Similar integral operators have appeared in~\cite{Bu-Fe:96b}.
The kernel $\kappa$ can be considered as a difference analog 
of the resolvent kernel  arising in the theory of differential equations.
First, we construct a solution to the integral equation \eqref{int-eq},
and then, we check that it is analytic in $z$ and satisfies the difference
equation~\eqref{eq:harper}. Finally, we obtain the asymptotics of this
solution for $y\to+\infty$ and for $y\sim 0$.
%
\subsection{Integral equation} 
%
Here, we  prove the existence of a solution (continuous 
in $z$ and analytic in $p$) to the integral equation.
Below,
\begin{equation*}
q(y)=(1+|y|)\begin{cases}
e^{-\frac{\pi y}h-\pi y}, &y\ge 0,\\
e^{\frac{2\pi \re p |y|}h},&y\le 0,
\end{cases}
\quad
\tilde q(y)=(1+|y|)\begin{cases}
e^{+\frac{\pi y}h-\pi y}, & y\ge 0,\\
e^{-\frac{2\pi \re p |y|}h},& y\le 0.
\end{cases}
\end{equation*}
\subsubsection{Estimates of $\mu$ and $\tilde \mu$}
To estimate the norm of the integral operator, we use
\begin{Cor}\label{Mu,tildeMu:estimates}
On the curve $\gamma$, the functions $\mu$ and $\tilde \mu$ satisfy the 
estimates
\begin{equation}\label{mu-est:rough}
|\mu(z)|,\,|\mu'(z)| \le H q(y),\qquad 
|\tilde \mu(z)|,\,|\tilde \mu'(z)|\le H \tilde q(y),
\end{equation}
\end{Cor}
This Corollary follows directly from Lemma~\ref{mu:estimates}.
\subsubsection{A solution to the integral equation}
One has
\begin{Pro}\label{pro:psi-a:1}
Fix positive $\alpha<1$.
There is a positive constant $C$
such that if
$\lambda\le e^{-\frac{C}h}$, then
the integral equation~\eqref{int-eq} has a solution
$\psi_0$ continuous in $z\in\gamma$, analytic in $p\in P$
and satisfying the estimate
\begin{equation}\label{est:psi-a:1}
\left|\psi_0(z)-\mu(z+i\xi)\right|\le
 \lambda^\alpha H q(y+\xi),\quad z\in\gamma.
\end{equation}
\end{Pro}
\begin{proof}
Let $C(\gamma,q)$ be the space of functions defined and continuous
on $\gamma$  and having the finite norm
$\|f\|=\sup_{z\in \gamma} |q^{-1}(y+\xi)\,f(z)|$.
The proof  is based on 
\begin{lemma}\label{le:kernel-estimate}
For $z,\,z'\in \gamma$, one has
  \begin{equation}\label{kernel-estimate}
|q^{-1}(y+\xi)\,\kappa\,(z,\,z')\,q(y'+\xi)|\le \lambda\,H\,
e^{-2\pi y'}(1+\xi^2).
\end{equation}
\end{lemma}
First, we prove the proposition, and then, we check
estimate~\eqref{kernel-estimate}. This estimate implies that the norm of 
$K$ as an operator acting in $C(\gamma, q)$ is bounded by
$\lambda\,H\,(1+\xi^{2})$. By Corollary~\ref{Mu,tildeMu:estimates}, 
$\mu(.+\xi)\in C(\gamma,q)$. So, there is a positive constant 
$C$ such that, if $\lambda(1+\xi^{2})\, <e^{-C/h}$, then there is 
$\psi_0$, 
a solution to~\eqref{int-eq} from $C(\gamma,q)$. The estimate of 
the norm of the integral operator  implies that
$$
\left|\psi_0(z)-\mu(z)\right|\le
\lambda(1+\xi^{2})H q(y+\xi),\quad z\in\gamma.
$$
This  implies~\eqref{est:psi-a:1} for any fixed positive $\alpha<1$.
The analyticity of $\psi_0$ in $p$ follows from one of $\mu$
and the uniformity of the estimates. The proof is  complete.\qed

Let us prove Lemma~\ref{le:kernel-estimate}. 
Below, $z,z'\in \gamma$. We note that by~\eqref{wronskian:formula}, for
$p\in P$ one has $C^{-1}\le \{\mu(z),\tilde\mu(z)\}\le C$. 

First, we consider 
the case where $|y-y'|\ge h$. In view of 
Corollary~\ref{Mu,tildeMu:estimates}, we get 
\begin{gather*}
|q^{-1}(y+\xi)\,\kappa\,(z,\,z')\,q(y'+\xi)|\le \lambda H e^{-2\pi 
y'}(E_1+E_2),\\
E_1=|\Theta(z,z')|\,\tilde q(y'+\xi)\,q(y'+\xi),\qquad
E_2=|\Theta(z,z')|\,\frac{\tilde q(y+\xi)}{q(y+\xi)}\,q^2(y'+\xi).
\end{gather*}
To justify~\eqref{kernel-estimate}, it suffices to check that 
$E_{1,2}\le C\,(1+|\xi|)^2$. Note that
\begin{equation}\label{est:theta}
  |\Theta(z,z')|\le C\,\begin{cases} e^{-\frac{2\pi(y-y')}h}&\text{ if 
}y-y'\ge 
h,\\
 1 & \text{ if } y'-y\ge h.\end{cases}
\end{equation}
Clearly, $E_1\le C\,\tilde q(y'+\xi) q(y'+\xi)$. For $y'\ge -Y$, we have
\begin{equation}\label{kernel-estimate:2a}
\begin{split}
\tilde q(y'+\xi) q(y'+\xi)&\le (1+|y'+\xi|)^2e^{-2\pi(y'+\xi)}\le 
C,\quad \text{if}\quad y'+\xi\ge0,\\
\tilde q(y'+\xi) q(y'+\xi)&\le (1+|y'+\xi|)^2\le 
C\,(1+|\xi|)^2 \quad \text{otherwise}.
\end{split}
\end{equation}
This implies that  $E_{1}\le C\,(1+|\xi|)^2$.
To estimate $E_2$, we have to consider four cases. If $y+\xi,\,y'+\xi\ge 
0$, 
we have
\begin{equation*}
  E_2\le |\Theta(z,z')|\,(1+|y'+\xi|)^2\,e^{\frac{2\pi(y-y')}h-2\pi 
(y'+\xi)}
\le C\, |\Theta(z,z')|\,e^{\frac{2\pi(y-y')}h}\le C.
\end{equation*}
If $y+\xi\ge  0\ge y'+\xi$, then
\begin{equation*}
\begin{split}
  E_2\le(1+|y'+\xi|)^2\, &|\Theta(z,z')|\,
e^{\frac{2\pi(y+\xi)}h-\frac{4\pi\re p(y'+\xi)}h}\\
&\le C\,(1+\xi^2)\,e^{\frac{2\pi (1-2\re p)(y'+\xi)}h}\le C\,(1+\xi^2).
\end{split}
\end{equation*}
If $y'+\xi\ge  0\ge y+\xi$, then
\begin{equation*}
\begin{split}
  E_2\le (1+|&y'+\xi|)^2\,|\Theta(z,z')|\,
e^{\frac{4\pi\re p(y+\xi)}h-\frac{2\pi(y'+\xi)}h-2\pi (y'+\xi)}\\
&\le C\,|\Theta(z,z')|\,e^{\frac{4\pi\re p(y+\xi)}h-\frac{2\pi(y'+\xi)}h}
\le C\,|\Theta(z,z')|\le C.
\end{split}
\end{equation*}
Finally, if $y+\xi,\,y'+\xi\le 0$, we get
\begin{equation*}
  E_2\le (1+|y'+\xi|)^2\,|\Theta(z,z')|\,e^{\frac{4\pi \re p(y-y')}h}\le 
C\,(1+|\xi|)^2.
\end{equation*}
These estimates imply that  $E_{2}\le C\,(1+|\xi|)^2$. This completes 
the proof in the case where $|y-y'|\ge h$.
\\
Let us consider the case where $|y-y'|\le h$. Let $\eta=\im \zeta$.
Using~\eqref{mu-est} we get  
\begin{equation*}
\begin{split}
 |\Theta(z,z') \,(\mu(z+i&\xi)\,\tilde \mu\,(z'+i\xi)-
         \mu\,(z'+i\xi)\,\tilde \mu\,(z+i\xi)\,)|\\
&\le Ch\,\max_{\zeta\in\gamma,\, |y-\eta|\le h}\left|
\mu(z+i\xi)\,\tilde \mu'(\zeta+i\xi)-
         \mu'(\zeta+i\xi)\,\tilde \mu\,(z+i\xi)\right|\\
&\le H \left(q(y+\xi)\max_{|y-\eta|\le h}\tilde q(\eta+\xi)+
\tilde q(y+\xi)\max_{|y-\eta|\le h}q(\eta+\xi)\right)\\
&\le Hq(y+\xi)\tilde q(y+\xi),
\end{split}
\end{equation*}
and, using~\eqref{kernel-estimate:2a}, we again come 
to~\eqref{kernel-estimate}. This completes the proof. 
\end{proof}
Note that Corollary~\ref{Mu,tildeMu:estimates}
and Proposition~\ref{pro:psi-a:1} imply
\begin{Cor}
In the case of Proposition~\ref{pro:psi-a:1}, there is  $C$
such that, for $\lambda\le e^{-\frac{C}h}$,
\begin{equation}\label{est:psi-a:2}
|\psi_0(z)|\le H q(y+\xi),\quad z\in\gamma.
\end{equation}
\end{Cor}
\subsection{ Analytic continuation of the solution to the
integral equation}
Here, we prove the first point of Theorem~\ref{psi0}. One has
\begin{lemma}
The solution $\psi_0$ can be analytically
continued in $\C_+(-Y)$.
\end{lemma}
\noindent Similar statements were checked in~\cite{Bu-Fe:96b}
and~\cite{Bu-Fe:01},  we outline the proof only for the convenience
of the reader.
\begin{proof} For $z'\in \gamma$, the kernel
$\kappa\,(z,\,z')$ is analytic in $z\in S_h=\{|\re z|<h,\,y>-Y\}$, 
and the function $K\psi_0$ can be analytically
continued in $S_h$; \  $\psi_0$, being a solution to
\eqref{int-eq}, can be also analytically continued in $S_h$.

Having proved that $\psi_0$ is analytic in $S_h$,
one can deform the integration contour in the formula
for $K\psi_0$ inside $S_h$, and check
that, in fact, $\psi_0$ can be analytically continued in
$S_{2h}=\{|\re z|<2h,y>-Y\}$. Continuing in this way,
one comes to the statement of the Lemma.
\end{proof}

Below, we denote by $\psi_0$ also the analytic continuation 
of the old $\psi_0$. 
\subsubsection{Function $\psi_0$ and the difference equation}
%
To check that $\psi_0$ satisfies~\eqref{pert-model}, we again borrow
an argument from~\cite{Bu-Fe:96b} and~\cite{Bu-Fe:01}. 

We call a curve in $\C$ vertical if along it $x$ is a smooth function of 
$y$.
For $z\in \C_+(-Y)$ we denote by $\gamma(z)$ a vertical curve that begins 
at $-iY$,
goes upward to $z$, then comes back to the imaginary axis and goes 
along it to $+i\infty$. One can compute $K\psi_0(z)$ by the formula 
in~\eqref{int-eq} with the integration path $\gamma$ replaced with 
$\gamma(z)$.   

Set $({\mathcal H}_0\,f)\,(z)=f\,(z+h)+f\,(z-h)+\lambda\,e^{-2\pi i z}f(z)-
2\cos(2\pi p)\,f(z)$. Then, ${\mathcal H}_0\,\psi_0=-{\mathcal 
H}_0\,K\psi_0$.
Using~\eqref{int-eq} with $\gamma$ replaced with $\gamma(z)$, \ 
$z\in\C_+(-Y)$,
we get
$$({\mathcal H}_0\,K\, \psi_0)\,(z)= 2\pi
i\;\res_{z'=z}\,\kappa\,(z+h,\,z')\,\psi_0(z')=
\lambda\,e^{2\pi iz}\,\psi_0(z).$$
Thus, we come to
\begin{lemma} 
The solution $\psi_0$
satisfies equation \eqref{eq:harper} in $\C_+(-Y)$.
\end{lemma}
The last two lemmas imply the first point of the
Theorem~\ref{psi0}.
\subsection{Asymptotics in the upper half-plane}\label{sss:asymp:up}
We get asymptotics~\eqref{psi-0:up}  using the integral equation for 
$\psi_0$.
First, we pick $\delta\in(0,1)$ and sufficiently large $Y>0$,    assume
that $|x|\le \delta h$ and  $y>Y$, 
and represent $(K\psi_0)(z)$  in the form
\begin{equation}\label{psi-0:up:Kpsi}
(K\psi_0)(z)=\dsize \frac {\lambda}{2ih \{\mu,\,\tilde \mu\}}\,
\left(\mu\,(z+i\xi) (I(\tilde\mu)+J(\tilde\mu))-\tilde 
\mu\,(z+i\xi)(I(\mu)+J(\mu))\right),
\end{equation}
where 
\begin{equation}\label{eq:I}
I(f)=\int_{\gamma,\,\,|y-y'|\ge h} \Theta\,(z,\,z')\,f\,( 
z'+i\xi)\,e^{2\pi 
iz'}\,
         \psi_0(z') dz',
\end{equation}
\begin{equation}
\label{eq:J}
J(f)=\int_{\gamma,\,\,|y-y'|\le h}
         \Theta\,(z,\,z')\,(f\,( z'+i\xi)-f(z+i\xi))\,
         e^{2\pi iz'}\,\psi_0(z') dz'.
\end{equation}
Let us estimate $I(\mu)$, $I(\tilde\mu)$, $J(\mu)$ and $J(\tilde\mu)$.
The first two are defined $\forall x$. One has
\begin{lemma} Let $p\in P$. We pick $X>0$. As $y\to+\infty$, uniformly in 
$x\in[-X,X]$ 
  \begin{equation}\label{psi-0:up:I}
I(\tilde \mu)(z)=o(1),\qquad
I(\mu)(z)=e^{\frac{2\pi i (z+i\xi)}h} (a+o(1)),
  \end{equation}
where
\begin{equation}\label{eq:a}
a=2i\int_{\gamma} e^{-\frac{2\pi i(z'+i\xi)}h} \mu\,( z'+i\xi)\,e^{2\pi 
iz'}\, \psi_0(z') dz'.
\end{equation}
One has 
\begin{equation}\label{est:a}
  a = O\left((1+|\xi|)^3\lambda^{\nu(p)}H\right), \quad 
\dsize\nu(p)=\min\left\{1,\frac{1-2\re p}h\right\}.
\end{equation}
\end{lemma}
\noindent Note that the integral in~\eqref{eq:a} converges in view
of~\eqref{mu-est:rough} and~\eqref{est:psi-a:2}.
\begin{proof} Let us estimate $I(\tilde\mu)$. We assume that $y+\xi>0$.
Using~\eqref{est:theta}, ~\eqref{mu-est:rough} 
and~\eqref{est:psi-a:2}, we get 
\begin{multline*}
|I(\tilde \mu)(z)|\le H\left(
\int_{y}^\infty (1+|y'+\xi|)^2e^{-2\pi (y'+\xi)}e^{-2\pi y'}dy'\right.\\
\hspace{3cm}+\int_{-\xi}^y e^{-\frac{2\pi(y-y')}h}(1+|y'+\xi|)^2e^{-2\pi 
(y'+\xi)}e^{-2\pi y'}dy'\\
\left.+\int_{-Y}^{-\xi} e^{-\frac{2\pi(y-y')}h}(1+|y'+\xi|)^2e^{-2\pi 
y'}dy'\right).
\end{multline*}
The expression in the brackets is bounded by
\begin{equation*}
\int_{y}^\infty e^{-2\pi y'}dy'+
\int_{-\xi}^y e^{-\frac{2\pi(y-y')}h}e^{-2\pi y'}dy'+
(1+|Y-\xi|^2)\int_{-Y}^{-\xi} e^{-\frac{2\pi(y-y')}h}e^{-2\pi y'}dy'.
\end{equation*}
As $h<1$, this implies the first estimate from~\eqref{psi-0:up:I}.
\\
Let us turn to $I(\mu)$. 
Arguing as when estimating $I(\tilde \mu)$, we check that 
\begin{multline}\label{I(tilde-mu):up}
\left|\int_{\gamma,\,\,y'>y+h}
     \Theta\,(z,\,z')\,\mu\,( z'+i\xi)\,e^{2\pi iz'}\,\psi_0(z') 
dz'\right|\le\\
\le H\int_{y+h}^\infty
     (1+|y'+\xi|)^2e^{-\frac{2\pi (y'+\xi)}h-2\pi (y'+\xi)}e^{-2\pi 
y'}dy'=o(e^{-\frac{2\pi y}h})
\end{multline}
uniformly in $x$ as $y\to+\infty$. Then, using the estimate
\begin{equation*}
\left|\Theta(z,z')-2ie^{-\frac{2\pi i(z'-z)}h}\right|\le
C e^{-\frac{4\pi (y-y')}h},\quad y-y'\ge h,
\end{equation*}
and again arguing  as before, we prove that
as $y\to\infty$, uniformly in $x$
\begin{equation}\label{I(tilde-mu):down}
         \operatornamewithlimits\int_{\gamma,\,\,y>y'+h}
         \Theta\,(z,\,z')\,\mu\,( z'+i\xi)\,e^{2\pi iz'}\,
         \psi_0(z') dz'=e^{\frac{2\pi i (z+i\xi)}h} a+
         o\left(e^{-\frac{2\pi y}h}\right)
\end{equation}
with $a$ from~\eqref{eq:a}.
Estimates~\eqref{I(tilde-mu):up}--\eqref{I(tilde-mu):down} imply
the second estimate in~\eqref{psi-0:up:I}.

Let us prove~\eqref{est:a}.  Using~\eqref{eq:a}
and estimates~\eqref{mu-est:rough} and~\eqref{est:psi-a:2}, 
we get
\begin{equation*}
|a|\le H(1+|\xi|)^2
\operatornamewithlimits\int_{-Y}^{-\xi} e^{\frac{2\pi(1-2\re p) 
(y'+\xi)}h-2\pi y'} dy' +
H\operatornamewithlimits\int_{-\xi}^{+\infty}  e^{-2\pi 
(2y'+\xi)}(1+|y'+\xi|)^2 dy'.
\end{equation*}
One has $ \max_{0\le y'\le-\xi} e^{\frac{2\pi(1-2\re p) (y'+\xi)}h- 2\pi 
y'}=
e^{2\pi \nu(p)\xi}=\lambda^{\nu(p)}$ with  $\nu(p)$ defined by the 
second formula in~\eqref{est:a}. Furthermore, the second integral in the 
last 
estimate for $a$  equals  $\lambda\int_{0}^{+\infty}  e^{-4\pi t}(1+t)^2 
dt$.
These observations lead to \eqref{est:a}. This completes the proof of the 
lemma.
\end{proof}
Let us turn to the terms $J(\mu)$ and $J(\tilde\mu)$. One has
\begin{lemma}  Fix $\delta\in(0,1)$. Let $p\in P$. As $y\to+\infty$, one 
has
  \begin{equation}\label{psi-0:up:J}
 J(\tilde\mu)(z)=o(1),\quad J(\mu)(z)=o(e^{-\frac{2\pi y}h}).
  \end{equation}
These estimates are uniform in $x\in[-\delta h,\delta h]$. 
\end{lemma}
\begin{proof}
Let $y+\xi\ge h$. 
Using estimates~\eqref{mu-est} for $d\mu/dz$ and 
estimate~\eqref{est:psi-a:2}, 
we get
\begin{equation*}
|J(\mu)|\le H(1+|y+\xi|)^2e^{2\pi \delta (y+\xi)-\frac{2\pi (y+\xi)}h-2\pi 
(y+\xi)-2\pi y} (1+|y+\xi|)^2
\le He^{-\frac{2\pi (y+\xi)}h-2\pi y},
\end{equation*}
where we used the inequalities  $|x|\le \delta h$ and  $0<\delta<1$.
This implies the first estimate in~\eqref{psi-0:up:J}.
Similarly one proves the second one.
\end{proof}
Now we are ready to prove~\eqref{psi-0:up}. We do it in three steps.
Below we assume that $p\in P$. All the $o(1)$  are uniform in $x$.
\\
{\bf 1.} \  Fix $\delta\in(0,1)$. Let  
$|x|\le\delta h$. 
Estimates~\eqref{psi-0:up:I} and~\eqref{psi-0:up:J} 
imply that as $y\to\infty$
\begin{equation}\label{psi-0:up:delta-vic}
\psi_0=\mu(z+i\xi)-(K\psi_0)(z)=\mu(z+i\xi)(1+o(1))+ e^{\frac{2\pi i 
(z+i\xi)}h}
\tilde \mu(z+i\xi)(\tilde a+o(1))
\end{equation}
where $\tilde a= \lambda a/(2hi \{\mu,\,\tilde \mu\})$. 
Representation~\eqref{psi-0:up:delta-vic},
and formulas~\eqref{tildeMu} and~\eqref{mu:up} imply~\eqref{psi-0:up} with 
$\varkappa_0=ie^{\frac{i\pi}{4h}}\tilde a$. This completes the proof 
of~\eqref{psi-0:up}
for $x\in[-\delta h,\delta h]$.
We note that~\eqref{psi-0:up} implies that, for  sufficiently 
large $y$,
\begin{equation}\label{est:psi:inter}
|\psi_0(z)|\le C(1+|\varkappa_0|)
e^{\frac{2\pi|x|(y+\xi)}h -\frac{\pi( y+\xi)}h-\pi (y+\xi)}.
\end{equation}
{\bf 2.} \  Now, we pick $\delta\in(1/2,1)$ and justify~\eqref{psi-0:up}
assuming that $h\delta \le x\le h(1+\delta)$. 
Let  $\epsilon\in(0,1-\delta)$. We denote by 
$\gamma(z)$  the curve that goes first along a straight line
from $i(y-h)$ to the point $z-h+\epsilon h$ and next along a straight line 
from 
this point to $i(y+h)$. One obtains
representation~\eqref{psi-0:up:Kpsi} with  $J(f)$ defined by the new  
formula:
\begin{equation}\label{new:J}
J(f)(z)=\int_{\gamma(z)}\Theta\,(z,\,z')\, f\,(z'+i\xi)\, e^{2\pi 
iz'}\,\psi_0(z') dz'.
\end{equation}
We note that on  $\gamma(z)$, one has $|\Theta\,(z,\,z')|\le C$. As 
$$-\delta h\le (\delta-1) h< x-h+\epsilon h\le (\delta+\epsilon)h,$$
and as $\delta+\epsilon<1$,
we can and do assume that  on $\gamma(z)$ for   sufficiently large $y$ 
solution $\psi_0$ satisfies~\eqref{est:psi:inter}.
Using~\eqref{est:psi:inter} and~\eqref{mu-est}, we check  that  as 
$y\to+\infty$
\begin{equation*}
 |J(\mu)(z)|\le H(1+|\varkappa_0|)
e^{4\pi (\delta+\epsilon-1) (y+\xi)-\frac{2\pi(y+\xi)}h+2\pi 
\xi}(1+|y+\xi|)^2=
o(e^{-\frac{2\pi(y+\xi)}h}).
\end{equation*}
Reasoning analogously, we also prove that $ J(\tilde\mu)(z)=o(1)$ as 
$y\to+\infty$.
These two estimates and~\eqref{psi-0:up:I} lead again to~\eqref{psi-0:up}.
This completes the proof of~\eqref{psi-0:up} for $0\le x\le (1+\delta)h$. 
The case of  
$ (1+\delta)h\le x\le 0$ is analyzed similarly.
\\
{\bf 3.} \ To justify~\eqref{psi-0:up} for larger $|x|$,
one uses equation~\eqref{pert-model}.
We discuss only the case of $x\ge 0$ and omit further details.
Pick $\delta\in(0,1)$ and $X>0$.
In the case of Theorem~\ref{psi-0:up}, 
we can assume that $|\varkappa_0|\le 1/2$. Then, $1+\varkappa_0\ne 0$ and  
for $x\in [\delta h, X]$~\eqref{psi-0:up} actually means that
\begin{equation*}
  \psi_0(z)=A\xi_-(z+i\xi)(1+o(1)),\qquad   
A=e^{\frac{i\pi}4}(1+\varkappa_0)v_-\,.
\end{equation*}
By~\eqref{pert-model}, one can write
\begin{equation*}
\psi_0(z)=-\lambda e^{-2\pi i(z-h)}\psi_0(z-h)\,\left(1+o(1)+
e^{2\pi i(z+i\xi-h)}\,\psi_0(z-2h)/\psi_0(z-h)\right).
\end{equation*}
Let  $(1+\delta)h\le x\le (2+\delta)h$.
Then $\psi_0(z-h)=A\xi_-(z+i\xi-h)(1+o(1))$. This and~\eqref{est:psi:inter}
(that is valid on any given compact subinterval of $(-h,h)$)
imply that $e^{2\pi i z}\psi_0(z-2h)/\psi_0(z-h)=o(1)$. 
Therefore,
\begin{equation*}
\psi_0(z)=-\lambda\,e^{-2\pi 
i(z-h)}\psi_0(z-h)\,(1+o(1))=A\xi_-(z+i\xi)(1+o(1)).
\end{equation*}
So, we proved~\eqref{psi-0:up} for $x\in[0,(2+\delta)h]$. Continuing 
in this way, one proves~\eqref{psi-0:up} for all $x$ such that 
$0\le x\le X$.
This completes the proof of~\eqref{psi-0:up}.\qed
\subsection{Asymptotics in the lower half-plane}
Here, we prove the third statement of Theorem~\ref{psi0}, i.e.,   
estimate~\eqref{psi-0:down}.
We pick $X>0$ and a sufficiently large $Y>0$ and assume that 
$|x|\le X$ and $|y|\le Y/2$. We also assume that $\lambda$ is so small 
(or $\lambda \le e^{C/h}$ with $C $ so large) that $Y<-\xi$.

The proof follows the same plan and uses the same estimates for 
$\mu$, $\tilde \mu$ and $\psi_0$ as for studying $\psi_0$ as $y\to\infty$.
So, we omit elementary details.
\\
{\bf 1.} \ First, we represent $K\psi_0$ in the 
form~\eqref{psi-0:up:Kpsi} with $I$  and $J$ given 
by~\eqref{eq:I}--~\eqref{eq:J}.
Then,  using~\eqref{mu-est:rough} and~\eqref{est:psi-a:2} and the rough 
estimate 
$|\Theta(z,z')|\le C$ for $|y-y'|\ge h$, we get 
\begin{equation}\label{psi-0:down:ItildeI}
|I(\tilde \mu)(z)|\le H (1+|\xi|)^2,\quad
|I(\mu)(z)|\le  H (1+|\xi|)^2\,e^{4\pi\re p |\xi|/h}.
\end{equation}
{\bf 2.} \ Let us turn to the terms $J(\mu)$ and $J(\tilde \mu)$.
We first  fix a positive $\delta<1$, and 
consider the case where $|x|\le \delta h$.  By means of 
Lemma~\ref{mu:estimates}
and~\eqref{est:psi-a:2} , we get
\begin{equation}\label{psi-0:down:J's}
|J(\tilde\mu)(z)|\le H\,(1+|\xi|)^2,\quad
|J(\mu)(z)|\le H\,(1+|\xi|)^2\, e^{4\pi\re p |\xi|/h}.
\end{equation}
{\bf 3.} We  recall that $\psi_0(z)-\mu(z+i\xi)=K\psi_0(z)$. Estimating 
the 
right hand side 
by means of~\eqref{psi-0:down:ItildeI}--\eqref{psi-0:down:J's}, the 
estimate from Lemma~\ref{mu:estimates} for $\mu$ 
and~\eqref{tildeMu},  the definition of $\tilde \mu$, we get
\begin{equation}\label{psi-0:down:delta-vic}
|\psi_0(z)-\mu(z+i\xi)|\le  \lambda H (1+|\xi|)^3\,e^{2\pi\re p |\xi|/h},
\end{equation}
i.e.,  representation~\eqref{psi-0:down} for $|x|\le \delta h$
\\
{\bf 4.} \ Now, we pick  $\delta\in(1/2,1)$ and  
justify~\eqref{psi-0:down:delta-vic}
for $|x|\le (1+\delta)h$.
\\
As $|y|<Y/2$, we can assume that the point $z-h$ is above the lower end of 
$\gamma$.  This allows to choose the curve $\gamma(z)$ as in
section~\ref{sss:asymp:up}, and redefine $J$ by~\eqref{new:J}. 

We can and do assume  that estimate~\eqref{psi-0:down:delta-vic} is proved 
on $\gamma(z)$.
This estimate and~\eqref{mu-est} imply that on $\gamma(z)$ one has
$|\psi_0(z)|\le H(1+|\xi|)e^{2\pi \re p |\xi|/h}$.
Using this and~\eqref{mu-est} we obtain 
for the new $J(\mu)$ and $J(\tilde \mu)$  the old 
estimates~\eqref{psi-0:down:J's}, 
and, therefore, we come to~\eqref{psi-0:down:delta-vic} for $\delta h\le 
x\le (1+\delta)h$.
The case of negative $x$ is treated similarly.
\\
{\bf 5.} \ Let us prove~\eqref{psi-0:down:delta-vic} for   all
$|x|\le X$. Therefore, we use a difference analog of the Gr\"onwall's 
inequality. We discuss only the case where $x>0$. 
The case of $x<0$ is treated similarly. 

Let $\delta(z)=\psi_0(z)-\mu(z+i\xi)$. Equations~\eqref{eq:harper} 
and~\eqref{mu:eq} for $\psi_0$ and $\mu$ imply that
\begin{equation*}
  \delta(z+h)+\delta(z-h)+2(\lambda\cos(2\pi z)-\cos(2\pi p))\delta(z)=
-\lambda e^{2\pi iz}\mu(z+i\xi).
\end{equation*}
Therefore,
\begin{equation*}
  \Delta(z+h)=M(z)\Delta (z) -\lambda e^{2\pi iz}\mu(z+i\xi)\,e_1,
\end{equation*}
where
\begin{equation*}
\Delta(z)=\begin{pmatrix} \delta(z)\\ \delta(z-h)\end{pmatrix},\quad
M(z)=\begin{pmatrix} 2(\cos(2\pi p)-\lambda\cos(2\pi z))& -1\\1& 
0\end{pmatrix},\quad
e_1=\begin{pmatrix} 1\\0\end{pmatrix}.
\end{equation*}
Let $n(z)=\|\Delta(z)\|_{\C^2}$ and $A=\max_{|y|\le Y}\|M(z)\|$, 
where $\|\cdot\|$ is the Hilbert-Schmidt norm. In view of~\eqref{mu-est} 
we obtain 
\begin{equation}\label{ineq:n}
n(z)\le B+A\, n(z-h),\quad B=\lambda(1+|\xi|) H e^{2\pi \re p|\xi|/h}.    
\end{equation}
Therefore, for $N\in\N$, we have
\begin{equation}\label{ineq:n-mult}
n(z)\le B\,(1+A+...A^{N-1})+ A^N\, n(z-Nh)=B\frac{A^N-1}{A-1}+A^N\,n(z-Nh).
\end{equation}
Assume that $\frac{3h}2\le x \le X$ and choose $N$ so that  
$\frac h2\le x-Nh\le \frac{3h}2$.  Then,
\\
$\bullet$ \  by~\eqref{psi-0:down:delta-vic}, \ 
$n(z-Nh)\le \lambda H (1+|\xi|)^3\,e^{2\pi \re p |\xi|/h}$;
\\
$\bullet$ \ $1\le N\le C/h$, and  $A^N\le H$.
\\
Using these observations and the estimate for $B$ from~\eqref{ineq:n},
we deduce from~\eqref{ineq:n-mult} estimate~\eqref{psi-0:down:delta-vic} 
for $0\le x\le X$. This completes the proof of Theorem~\ref{psi0}.
\section{A monodromy matrix for the Harper equation}
\label{s:mm:str}
Here  we give a  definition of  
the minimal entire solutions to the Harper equation and  describe 
the monodromy matrix corresponding to a basis of two minimal entire 
solutions. This is the matrix described in Theorem~\ref{th:1}. Then,
we prove Theorem~\ref{th:2}.
\subsection{Minimal entire solutions and monodromy matrices}
In~\cite{Bu-Fe:01} the authors studied entire 
solutions to equation~\eqref{eq:matrix} with an $SL(2,\C)$-valued 
$2\pi$-periodic entire function $M$. Using the equivalence
described in Section~\ref{sss:MM:diffSch},  we turn their results into 
results for the Harper equation.
\\
Let $Y\in\mathbb R$. Below,   $\mathbb 
C_+(Y)=\{z\in\mathbb C\,:\, y>Y\}$ and $\mathbb 
C_-(Y)=\{z\in\mathbb C\,:\, y<Y\}$.
\subsubsection{Solutions with the simplest behavior as $y\to\pm\infty$}
To characterize the behavior of an entire solution as 
$y\to\pm \infty$, we express it in terms of solutions 
having the simplest asymptotic behavior as 
$y\to\pm \infty$. Let us describe these solutions.
The next theorem follows from Theorem 1.1a from~\cite{Bu-Fe:01}.
\begin{theorem}\label{th:bloch} 
If $Y_1>0$ is sufficiently large,  there exist two 
solutions $u_\pm$ to \eqref{eq:harper} that are analytic in the 
half-plane $\C_+(Y_1)$ and admit there the representations
\begin{equation} \label{f-solutions}
u_\pm(z)=e^{\dsize
\pm\frac{ i\pi}{h}\,(z-\frac12+i\xi)^2+i\pi z+o\,(1)},\quad y\to+\infty.  
\end{equation} 
One has 
\begin{equation}
  \label{eq:Wronsk-f}
  \{u_+(z),u_-(z)\}=\lambda.
\end{equation}
Moreover, $u_\pm$ are Bloch solutions in the sense of~\cite{Bu-Fe:01}, 
i.e., 
the ratios $u_\pm(z+1)/u_\pm(z)$ are $h$-periodic in $z$.
\end{theorem}
\begin{Rem} The expressions 
$u_\pm^0(z)=e^{\dsize\pm\frac{i\pi}{h}\,(z-\frac12+i\xi)^2+i\pi z}$, 
the leading terms in~\eqref{f-solutions}, satisfy the equations
$u^0_\pm(z\mp h)+\lambda e^{-2\pi iz}\,u^0_\pm(z)=0$ (compare it with 
the Harper equation!).
\end{Rem}
\noindent
We construct two solutions with the simplest asymptotic behavior as 
$y\to-\infty$ by the formulas
\begin{equation}
  \label{eq:g-sol}
  d_\pm(z)=-u_\pm(1-z),\qquad z\in \mathbb C_-(-Y_1).
\end{equation}
We use 
\begin{lemma} \label{le:d-star-u} One has
\begin{equation}\label{u-ast-d}
d_\pm^*(z)=\alpha_\pm(z)u_\mp(z),\qquad  z\in\mathbb C_+(Y_1),
\end{equation}
where $\alpha_\pm$ are analytic and $h$-periodic functions, and, as $ 
y\to\infty$, \  $\alpha(z)\to 1$ uniformly in $x$.
\end{lemma}
\begin{proof} 
We check~\eqref{u-ast-d} for $d_+$. Mutatis mutandis, for $d_-$, the 
analysis is the same. 
\\
As $\{u_+(z),u_-(z)\}$, the Wronskian of $u_+$ and $u_-$, does not
vanish in $\mathbb C(Y_1)$, one has (see section~\ref{sss:MM:diffSch})
\begin{equation*}
d_+^*(z)=\alpha(z) u_-(z)+\beta(z) u_+(z), \quad 
\alpha(z)=\frac{\{u_+(z), d_+^*(z)\}}{\{u_+(z), 
u_-(z)\}}, \ \ \beta(z)=\frac{\{d_+^*(z), u_-(z)\}}{\{u_+(z), 
u_-(z)\}}.
\end{equation*}
The coefficients $\alpha$ and $\beta$ are $h$-periodic and 
analytic in  $\C_+(Y_1)$.
\\
We recall that $u_-$ is a Bloch solution. It means that the ratio 
$r(z)=u_-(z+1)/u_-(z)$
is $h$-periodic. Using~\eqref{f-solutions}, we get the asymptotic 
representation
\begin{equation*}
  r(z)=-\lambda^{1/h}e^{-2\pi i z/h+o(1)},\quad y\to\infty,
\end{equation*}
where the error estimate is  uniform in $x$. This implies that, for 
sufficiently large $y$, the solution $u_-$ tends to zero as 
$x\to-\infty$. Similarly one proves 
that $d_+^*$
does the same.  Therefore,  being periodic, the Wronskian $ \{d_+^*(z), 
u_-(z)\}$ equals zero. This implies that $\beta=0$.  
\\
Using~\eqref{f-solutions}, we check that as $y\to\infty$ one has 
$\alpha(z)=1+o(1)$ uniformly in $x$.
\end{proof}
\subsubsection{The minimal solutions}\label{s:min-sol}
Let $\psi$ be an entire solution to~\eqref{eq:harper}, and let $Y_1$ be 
as in Theorem~\ref{th:bloch}.
Then, $\psi$ admits the representations:
\begin{gather}
\label{psi:up}
\psi(z)=A(z)\,u_+(z)+B(z)\,u_-(z),\quad z\in \mathbb C_+(Y_1),\\
\label{psi:down}
\psi(z)=C(z)d_+(z)+D(z)\,d_-(z),\quad z\in \mathbb C_-(-Y_1),
\end{gather}
where $A$, $B$, $C$ and $D$ are analytic and $h$-periodic in  $z$.
The solution $\psi$ is called {\it minimal} if  $A$, $B$, $C$ and $D$
are bounded and one of them tends to zero as  $y$ tends to either 
$-\infty $ or $+\infty$.

Let $\psi$ be a minimal solution such that
$\lim_{y\to-\infty}D(z)=0$ and  $C(-i\infty)=\lim_{y\to-\infty}C(z)\ne0$. 
We set
$\psi_D(z)=\psi(z)/C(-i\infty)$.
In section~\ref{as:mm}, for sufficiently small $\lambda$, we construct  
$\psi_D$ in terms of  the solution $\psi_0$ from section~\ref{sec:int-eq}.
\\
Let  $A$, $B$, $C$ and $D$ be  defined for  $\psi=\psi_D$
by~\eqref{psi:up} and~\eqref{psi:down}.  The limits 
$$A_D=\lim_{y\to\infty}A(z), \ B_D=\lim_{y\to\infty}B(z), \ 
C_D=\lim_{y\to-\infty}C(z), \ D_D=\lim_{y\to-\infty}
e^{2\pi iz/h}D(z)$$ 
are  called the {\it asymptotic coefficients} of  $\psi_D$. By 
definition of $\psi_D$ one has $C_D=1$.  
\subsubsection{The monodromy matrix}
In terms of $\psi_D$, we define one more solution to the Harper 
equation~\eqref{eq:harper} by the formula
\begin{equation}
  \label{eq:psiB}
  \psi_B(z)=\psi_D(1-z).
\end{equation}
Clearly, $\psi_B$ is one more minimal entire solution. 

Theorem 7.2 from~\cite{Bu-Fe:01} can be formulated  as follows:
\begin{theorem}
The minimal entire solutions $\psi_D$  and $\psi_B$  exist. They, their 
asymptotic coefficients and their Wronskian are nontrivial meromorphic 
functions 
of
$E$.  The Wronskian is independent of $z$.  The monodromy matrix
corresponding to  $\psi_D$  and $\psi_B$ is of the form~\eqref{firstM}, and
\begin{equation}\label{first:ts}
s=-\lambda_1\frac{D_D}{B_D},\quad\quad t= -\lambda_1\,{A_D},\quad\quad
\lambda_1=\lambda^{\frac1h}=e^{\frac{2\pi\xi}h}.
\end{equation} 
\end{theorem}
In the case of $\lambda=1$, the analysis of the poles of $s$ and $t$ was 
done in~\cite{Bu-Fe:96b}. Mutatis mutandis, it can be done similarly in 
the general case. One can see that $s$ and $t$ are analytic on the 
interval $I=[-2-2\lambda, 2+2\lambda]$.
In section~\ref{as:mm}, as $\lambda\to 0$, we  compute the asymptotics  
of $s$ and $t$  for $E\in I$ and so check independently their analyticity.
\subsection{Symmetries and the monodromy matrix}
For a function $f$ of $z$ and $E$, we let
$f^*(z,E)=\overline{f(\bar z,\,\bar E)}$.
It is clear that $f$ and $f^*$ satisfy~\eqref{eq:harper}
simultaneously. Here, using this symmetry, we prove Theorem~\ref{th:2}.
\subsubsection{A relation for the monodromy matrix}
Let us consider  $E$ such that  $\psi_D$ and $\psi_B$ form a basis in the 
space of
solutions to Harper equation. The monodromy matrix corresponding to this 
basis is defined by~\eqref{eq:MM-def} with $\psi_1=\psi_D$ and 
$\psi_2=\psi_B$. Below we denote it by $M$ (instead of $M_1$).
\\
As $\psi_D^*$ and $\psi_B^*$ also are solutions to Harper equation, one 
can write
\begin{equation}\label{S1:def}
\Psi^*(z)= S(z/h)\Psi(z),\quad  \Psi^*(z)=\begin{pmatrix}
\psi_D^*(z)\\ \psi_B^*(z)\end{pmatrix},
\end{equation}
where $S$ is a $2\times2$ matrix with $1$-periodic entries.
It is entire in $z$ and meromorphic in $E$ like the basis 
solutions.
One has
\begin{lemma} \label{le:M-S}
The matrices $M$ and $S$ satisfy the relation
  \begin{equation}\label{M:sym}
S(z+h_1)\,M(z)= M^*(z)\,S(z),\qquad h_1=\{1/h\},
\end{equation}
where $M^*$ is obtained from $M$ by applying the operation ${}^*$ to each 
of its entries.
\end{lemma}
\begin{proof}
The definition of the monodromy matrix  and~\eqref{S1:def} imply that
$$S\left((z+1)/h \right)\,M(z/h)\,\Psi(z) = M^*(z/h)\,S(z/h)\,\Psi(z).$$
Let $(\Psi(z+h), \Psi(z))$ be the matrix made of the column vectors 
$\Psi(z+h)$ and $\Psi(z)$.
As the functions $S$ and $M$ are $1$-periodic, we get
$$S\left((z+1)/h \right)\,M(z/h)\,(\Psi(z+h),\Psi(z)) = 
M^*(z/h)\,S(z/h)\,(\Psi(z+h),\Psi(z)).$$
As 
$$\det (\Psi(z+h), \Psi(z))=\{\psi_D(z),\psi_B(z)\},$$ 
the determinant of  $(\Psi(z+h),\Psi(z))$  is nontrivial, and  we come to 
the relation
$S\left(z/h+1/h \right)\,M(z/h)= M^*(z/h)\,S(z/h).$ As $S$ is 
$1$-periodic, 
it
implies~\eqref{M:sym}.
\end{proof}
\subsection{Properties of the matrix {\it S}}
We start with the following elementary observation.
\begin{lemma} One has 
  \begin{equation}
    \label{eq:S-prop}
    \sigma_1\,S(h_1-z)\,\sigma_1=S(z),\quad 
    \sigma_1=\begin{pmatrix} 0 & 1 \\ 1 &0 \end{pmatrix}.
  \end{equation}
\end{lemma}
\begin{proof}
In view of~\eqref{eq:psiB}, one has $\Psi(1-z)=\sigma_1\Psi(z)$. 
This and~\eqref{S1:def} imply that $ \Psi^*(z)=\sigma_1 S(h_1-z/h)\sigma_1 
\Psi(z)$. 
Using~\eqref{S1:def} once more, we obtain the relation $ 
S(z/h)\Psi(z)=\sigma_1 S(h_1-z/h)\sigma_1 \Psi(z)$.
Using the argument from the end of the proof of 
Lemma~\ref{le:M-S}, we
deduce~\eqref{eq:S-prop} from this relation. The proof is complete.
\end{proof}
\noindent Now, we check
\begin{Pro}
The entries of $S$ are independent of $z$. 
One has
  \begin{gather}\label{eq:SIJ}
    S_{11}=S_{22}=\frac1{B_D},\qquad S_{12}=S_{21}
    =\frac{A_D}{B_D},\\
    \label{ADandBD}
    A_D^*=A_D,\qquad  B_DB_D^*-A_DA_D^*=1.
  \end{gather}
\end{Pro}
\begin{proof}
  We prove formulas~\eqref{eq:SIJ} for $S_{11}$ and $S_{12}$.
  These formulas and relation~\eqref{eq:S-prop} imply the formulas 
  for the other entries of $M$.
  \\
  According to~\eqref{eq:three-solutions}--~\eqref{eq:periodic-coef},
  relation~\eqref{S1:def} implies that
\begin{equation}
  \label{eq:S-wronsk}  
S_{11}(z/h)=\frac{\{\psi_D^*(z),\psi_B(z)\}}{\{\psi_D(z),\psi_B(z)\}},
\qquad 
  S_{12}(z/h)=\frac{\{\psi_D(z),\psi_D^*(z)\}}{\{\psi_D(z),\psi_B(z)\}}.
\end{equation}
Below, we  compute the Wronskians in~\eqref{eq:S-wronsk} in terms
of the asymptotic coefficients of  the solution $\Psi_D$. 

Let us  begin with $\{\psi_D(z),\psi_B(z)\}$. We recall that, for 
sufficiently
large $Y_1$, solution $\psi=\psi_D$ admits  representations~\eqref{psi:up}
and~\eqref{psi:down} with $h$-periodic coefficients $A$, $B$, $C$ and $D$, 
and $D=e^{-2\pi i z/h}D_1(z)$, where $D_1$ is bounded  in $\mathbb 
C_-(-Y_1)$.
By means of~\eqref{eq:psiB} and~\eqref{eq:g-sol}    we get for $z\in 
\mathbb C_+(Y_1)$
\begin{equation*}
\begin{split}
   \{\psi_D(z)&,\psi_B(z)\}=\\
&=\left\{A(z)u_+(z)+B(z)u_-(z), -C(1-z)u_+(z)-e^{-\frac{2\pi i 
(1-z)}h}D_1(1-z)u_-(z)\right\}\\
&=\left(-e^{-\frac{2\pi i (1-z)}h}A(z)D_1(1-z)+B(z) 
C(1-z)\right)\left\{u_+(z),u_-(z)\right\}.
\end{split}
\end{equation*}
Using this representation, (\ref{eq:Wronsk-f}) and the definitions of the
asymptotic coefficients of  $\psi_D$, see~\eqref{s:min-sol}, we check that 
\begin{equation*}
   \{\psi_D(z),\psi_B(z)\}\to\lambda B_DC_D=\lambda B_D\quad 
\text{as}\quad 
y\to\infty.
\end{equation*}
Similarly, we prove that 
\begin{multline*}
   \{\psi_D(z),\psi_B(z)\}=\\
=\left\{C(z)d_+(z)+e^{-2\pi i z/h}D_1(z)d_-(z), 
-A(1-z)d_+(z)-B(1-z)d_-(z)\right\}\\
\longrightarrow \quad \lambda B_D \qquad \text{as} \qquad y\to-\infty.
\end{multline*}
As $ \{\psi_D(z),\psi_B(z)\}$ is an $h$-periodic entire function, these 
observations imply that 
\begin{equation}\label{W:DB}
  \{\psi_D(z),\psi_B(z)\}=\lambda B_D.
\end{equation}
By means of Lemma~\ref{le:d-star-u}, one 
similarly computes the Wronskians $\{\psi_D^*(z),\psi_B(z)\}$ and 
$\{\psi_D(z),\psi_D^*(z)\}$,
and obtains the formulas
\begin{gather*}
  \{\psi_D^*(z),\psi_B(z)\}=\lambda+o(1)\quad\text{as}\quad y\to+\infty,\\ 
  \{\psi_D^*(z),\psi_B(z)\}=\lambda( B_DB_D^*-A_DA_D^*)+o(1)\quad
  \text{as}\quad y\to-\infty,\\
  \intertext{and}
  \{\psi_D(z),\psi_D^*(z)\}=\lambda A_D+o(1)\quad\text{as}\quad 
y\to+\infty,\\ 
  \{\psi_D(z),\psi_D^*(z)\}=-\lambda A_D^*+o(1)\quad
  \text{as}\quad y\to-\infty.
\end{gather*}
The last four formulas imply~\eqref{ADandBD} and formulas 
$$\{\psi_D^*(z),\psi_B(z)\}=\lambda,\quad  
\{\psi_D(z),\psi_D^*(z)\}=\lambda A_D.$$
This,~\eqref{W:DB} and~\eqref{eq:S-wronsk} imply~\eqref{eq:SIJ}.
\end{proof}
\subsubsection{Proof of  Theorem~\ref{th:2}}
Formulas~\eqref{M:sym} and~\eqref{eq:SIJ} imply the relation
\begin{equation*}
  \tilde S\,M(z)= M^*(z)\,\tilde S,\quad \forall z\in\C,\qquad
  \tilde S=\begin{pmatrix} 1 & A_D\\ A_D & 1 \end{pmatrix}.
\end{equation*}
This relation implies that
$$M_{12}(z)+A_D M_{22}(z)=M_{11}^*(z) A_D+ M_{12}^*(z), \qquad \forall 
z\in\C.$$ 
Substituting into this formula the expressions for the monodromy matrix 
entries from~\eqref{firstM}, we get
\begin{equation*}
  s+te^{-2\pi i z} +A_D\frac{st}{\lambda_1}=
  A_D\big(a^*-2\lambda_1\cos(2\pi z)\big)+s^*+t^*e^{2\pi i z},\qquad 
\forall z\in\C.
\end{equation*}
This equality of two trigonometric polynomials leads to the relations
\begin{equation*}
  t=-\lambda_1A_D,\quad s+A_D\frac{st}\lambda_1=a^*A_D+s^*.
\end{equation*}
The first of these two relations and the first formula in~\eqref{ADandBD}
imply that $t=-t^*$,   and substituting  in the second one the formula
$A_D=-t/\lambda_1$ and the formula for $a$ from~\eqref{firstM}, one easily 
checks that
$$ss^*=\lambda_1^2\frac{1-t^2}{\lambda_1^2-t^2}.$$
These two observations imply~\eqref{eq:s/a}.
\\
Let $E\in \R$. One has  $t=i\tau$ and
$s=-i\lambda_1\sqrt{\frac{1-t^2}{\lambda_1^2-t^2}} e^{i\alpha}$
with real $\tau$ and $ \alpha$. Using these representations,  we get the
formula~\eqref{trace}:
\begin{equation*}
  L=\frac{\lambda_1}{st}\,\big(1-s^2-t^2\big)+\frac{st}{\lambda_1}=
  \frac{2}{\tau}\sqrt{(1+\tau^2)(\lambda_1+\tau^2)}\cos \alpha.
\end{equation*}
The proof of Theorem~\ref{th:2} is complete.\qed
\section{Asymptotics of the monodromy matrix}\label{as:mm}
\subsection{Formulation of the Riemann-Hilbert problem}
%
Fix $Y>0$. Below  $\mathbb C_+=\C_+(-Y)$ and
$\mathbb C_-=\C_-(Y)$.
To construct the solution $\psi_D$, we paste it of 
solutions analytic in $\mathbb C_+$ and solutions analytic in
$\mathbb C_-$ by means of a Riemann-Hilbert 
problem. Here, we formulate this problem.
\subsubsection{Relations between entire solutions and 
solutions analytic in $\C_\pm$}
%
Let $\sigma$ be either the sign ``$+$'' or the sign ``$-$''.
Let $\mathbb{S}_\sigma$ be the set of solutions to Harper equation
that are analytic in $\C_\sigma$, and let  $\mathbb{K}_\sigma$ be the set 
of
the complex valued functions that are analytic and $h$-periodic in in 
$\C_\sigma$.
\\
Assume that  $\psi_\sigma$ and  $\phi_\sigma$ belong $\mathbb{S}_\sigma$.
Let  $w_\sigma(z)=\{\psi_\sigma(z),\phi_\sigma(z)\}$. One has 
$w_\sigma\in\mathbb{K}_\sigma$.
We assume that $w_\sigma$ does not vanish in $\C_\sigma$. Then, the pair 
$\psi_\sigma,\,\phi_\sigma$ is a basis in $\mathbb{S}_\sigma$.
Any entire solution $\psi$ to~\eqref{eq:matrix}  admits the representations
\begin{equation}\label{RP:1}
  \psi\,(z)=a_+(z)\,\psi_+(z)+b_+(z)\,\phi_+(z), \quad  z\in\C_+,\qquad
  a_+,\,b_+\in{\mathbb{K}}_+,
\end{equation}
\begin{equation}\label{RP:2}
     \psi\,(z)=a_-(z)\,\psi_-(z)+b_-(z)\,\phi_-(z),\quad z\in\C_-,\qquad
    a_-,\,b_-\in{\mathbb{K}}_-,
\end{equation}
with
\begin{equation*}
a_\pm(z)=\frac1{w_\pm(z)}\{\psi\,(z),\,\phi_\pm(z)\},\quad 
b_\pm(z)=\frac1{w_\pm(z)}\{\psi_\pm(z),\,\psi\,(z)\}.
\end{equation*}
It follows from ~\eqref{RP:1} and~\eqref{RP:2} that
\begin{equation}\label{RP:3}
a_+(z)\,\psi_+(z)+b_+(z)\,\phi_+(z)=a_-(z)\,\psi_-(z)+b_-(z)\,\phi_-(z),
\quad z\in\R.
\end{equation}
This implies that
\begin{equation}\label{RP:4}
V_+=\,G\,V_-,\quad z\in\R,\quad\quad 
V_+=\left(\begin{array}{c} a_+\\ b_+\end{array}\right),\quad
V_-=\left(\begin{array}{c} a_-\\ b_-\end{array}\right),
\end{equation}
where $G$ is the matrix  with the entries
\begin{equation}\label{RP:5}
\begin{array}{ll}
G_{11}(z)=\frac1{w_+(z)}\,\{\psi_-(z),\,\phi_+(z)\}, &
G_{12}(z)=\frac1{w_+(z)}\,\{\phi_-(z),\,\phi_+(z)\},\\
G_{21}(z)=\frac1{w_+(z)}\,\{\psi_+(z),\,\psi_-(z)\},&
G_{22}(z)=\frac1{w_+(z)}\,\{\psi_+(z),\,\phi_-(z)\}.
\end{array}
\end{equation}
\begin{Rem} One has
\begin{equation}\label{detG}
\det G(z)= w_-(z)\,/\,w_+(z),\quad z\in\R.
\end{equation}
Indeed,~\eqref{RP:3} implies also the relation
$$V_-=\frac1{w_-(z)}\,\begin{pmatrix}
\{\psi_+(z),\,\phi_-(z)\} &\{\phi_+(z),\,\phi_-(z)\}\\
\{\psi_-(z),\,\psi_+(z)\}& \{\psi_-(z),\,\phi_+(z)\}
\end{pmatrix}\, V_+.$$ 
One also can express $V_-$ via $V_+$ by inverting the matrix $G$ 
in~\eqref{RP:4}. Comparing the results, one comes to~\eqref{detG}.
\end{Rem}
We have checked
\begin{lemma}
Any entire solution to~\eqref{eq:matrix} can be 
represented by~\eqref{RP:1} -- \eqref{RP:2} with $a_+,\, 
b_+\in{\mathbb{K}}_+$ and  
$a_-,\,b_-\in {\mathbb{K}}_-$, and these coefficients satisfy the
relation~\eqref{RP:4} with the matrix $G$ given by~\eqref{RP:5}.
\end{lemma}
One can easily prove also 
\begin{lemma}
If $a_+$ and $b_+$ belong to ${\mathbb{K}}_+$, \  
$a_-$ and $b_-$ belong to ${\mathbb{K}}_-$, and if these four functions 
satisfy 
relation~\eqref{RP:4} with the matrix $G$ given by~\eqref{RP:5}, 
then formulae~\eqref{RP:1} -- \eqref{RP:2} 
describe an entire solution to~\eqref{eq:matrix}.
\end{lemma}
\subsubsection{Basis solutions for constructing $\psi_D$}
\label{sss:bases-for-psiD}
Let $\psi_0(z,p)$ be the solution to~\eqref{eq:harper}
described in Proposition~\ref{psi0}, and let
\begin{equation}\label{0:up}
\psi_+(z,p)=\psi_0(z,p),\quad
\phi_+(z,p)=e^{-\frac{\pi i(z+i\xi)}h} \psi_0\,(z-1/2,1/2-p),
\end{equation}
and
\begin{equation}\label{0:down}
\psi_-(z,p)=\psi_+^*(z,\,p),\quad
\phi_-(z,p)=\phi_+^*(z,\,p).
\end{equation}
Clearly, $\psi_+,\,\phi_+\in{\mathbb{S}}_+$, and 
$\psi_-,\,\phi_-\in{\mathbb{S}}_-$.
 
To work with $\psi_\pm$ and $\phi_\pm$, we need to describe their behavior
for large $y$ and for $y\sim 0$. 
Theorem~\ref{psi0} and formulas~\eqref{f-solutions}
and~\eqref{eq:g-sol} imply
\begin{Cor} In the case of Theorem~\ref{psi0}, for sufficiently large 
$Y_1$, and for all $z\in\C_+(Y_1)$, one has
\begin{equation}\label{psi-phi-plus}
\begin{split}
\psi_+(z,p)=&A_\psi(z) e^{\frac{2\pi i(z+i\xi)}h}
u_+(z)+B_\psi (z)u_-(z),\\
\phi_+(z,p)&=A_\phi(z)\, u_+(z)+\,B_\phi(z)\,
u_-(z).
\end{split}
\end{equation}
For $z\in\C_-(-Y_1)$, one has
\begin{gather}\label{psi-phi-minus}
\begin{split}\textstyle
  \psi_-(z,p)=&\textstyle\frac{B_\psi^*(z)}{\alpha_+^*(z)}\; d_+(z)+
  \frac{A_\psi^*(z)}{\alpha_-^*(z)} \;e^{-\frac{2\pi i(z-i\xi)}h} d_-(z),\\
  \phi_-(z,p)&=\textstyle\frac{B_\phi^*(z)}{\alpha_+^*(z)}\; d_+(z)+
   \frac{A_\phi^*(z)}{\alpha_-^*(z)} \;d_-(z),
\end{split}
\end{gather}
where $\alpha_\pm$, $A_\psi$, $B_\psi$, $A_\phi$ and $B_\phi$ are 
$h$-periodic 
and analytic in $z$; \ $\alpha_\pm$ are described in 
Lemma~\ref{le:d-star-u},
and one has
\begin{gather}\label{A0,B0} 
  A_\psi(z)=A_{\psi,0}(1+o(1)),\quad B_\psi(z)=B_{\psi,0}(1+o(1)),\quad 
y\to+\infty,\\
\nonumber
A_{\psi,0}=e^{\frac{i\pi}4-\frac{2\pi ip^2}h-\frac{i\pi}{3h}-\frac{i\pi 
h}{12}}\,
e^{-\pi \xi},\quad
B_{\psi,0}=e^{-\frac{i\pi}4-\frac{i\pi h}4-\pi\xi}(1+\varkappa_0),
\\
\label{A0,B0-hats}
A_\phi(z)=A_{\phi,0}(1+o(1)),\quad B_\phi(z)=B_{\phi,0}(1+o(1)),\quad 
y\to+\infty,\\
\nonumber
A_{\phi,0}(p)=-iA_{\psi,0}(1/2-p)\,e^{-\frac{i\pi}{4h}} ,\quad
B_{\phi,0}(p)=-iB_{\psi,0}(1/2-p)\,e^{-\frac{3i\pi}{4h}}.
\end{gather}
These  asymptotics are uniform in $\re z$.
\end{Cor}
Furthermore, the third point of Theorem~\ref{psi0} and
Proposition~\ref{pro:mu:down} imply
\begin{Cor} Let $Y$, $\lambda$ and $h$ be as in Theorem~\ref{psi0}. Fix 
$\alpha\in(0,1)$ and $X>0$. For $p\in P$, \  $|y|\le 
\frac{Y}2$
and $|x|\le X$ the following holds.
\\ 
$\bullet$ \ Pick $\beta\in(0,\alpha)$. If \ 
$\re p\ge \alpha h/2$, then
\begin{equation}\label{psi-plus:1}
\psi_+(z)= e^{\frac{2\pi ip (z+i\xi)}h}
\left(a(p)+O_H(\lambda^\beta)\right),
\end{equation}
and if \ $\re p\le \frac12-\alpha h/2$, then
\begin{equation}\label{phi-plus:1}
\phi_+(z)= e^{-\frac{2\pi ip (z+i\xi)}h}
\left(e^{-\frac{i\pi (1/2-p)}h}a(1/2-p)+O_H(\lambda^\beta)\right).
\end{equation}
\\
$\bullet$ \ Pick $\beta\in (0,1)$. If \ $\re p\le \alpha h/2$, then
\begin{equation}\label{psi-plus:2}
\begin{split}
\psi_+(z)= a(p)\,e^{\frac{2\pi ip (z+i\xi)}h}+a(-p)\,e^{-\frac{2\pi 
ip(z+i\xi)}h}&+
O_H\left(\lambda^\beta e^{-\frac{2\pi p\xi}h}\,\right),
\end{split}
\end{equation}
and if \ $p\ge \pi-\alpha h/2$, then
\begin{equation}\label{phi-plus:2}
\begin{split}
\phi_+(z)= e^{-\frac{2\pi ip (z+i\xi)}h}\,&e^{-\frac{i\pi 
(1/2-p)}h}a(1/2-p)\\
+&e^{-\frac{2\pi i(1/2-p)(z+i\xi)}h}\,e^{\frac{i\pi (1/2-p)}h}a(p-1/2)+
O_H\left(\lambda^\beta e^{\frac{2\pi p\xi}h}\,\right).
\end{split}
\end{equation}
\end{Cor}
\noindent We complete this section by computing the Wronskians 
$w_\pm=w_\pm(z)$. 
As
\begin{equation}\label{w=w}
w_-(z)=w_+^*(z),
\end{equation}
we need to compute only $w_+$. 
One has
\begin{lemma}\label{le:w+} Let $Y$, $\lambda$ and $h$ be as in 
Theorem~\ref{psi0}. Fix $0<\beta<1$.
For  $p\in P$ and $z\in\C_+(-Y/2)$, one has
\begin{equation}\label{w+}
w_+= w_0+O(H\lambda^\beta),\quad
w_0=ie^{-\frac{2\pi ip^2}h+\frac{2ip\pi}h-\frac{13i\pi}{12h}-\frac{i\pi 
h}3}.
\end{equation}
\end{lemma}
\begin{Rem}
Lemma~\ref{le:w+} implies that $\psi_+$ and $\phi_+$ are linearly 
independent
if $\lambda<e^{-C/h}$, and $C$  is sufficiently large.
\end{Rem}
\begin{proof} First, we compute $w_+$ as $y\to+\infty$. 
Using~\eqref{psi-phi-plus} we get 
\begin{align*}
  \{\psi_+,\phi_+\}&=
\left\{A_\psi e^{\frac{2\pi i(z+i\xi)}h}u_++B_\psi u_-,
A_\phi u_++B_\phi u_-\right\}\\
&=\left(A_\psi B_\phi e^{\frac{2\pi i(z+i\xi)}h}-B_\psi A_\phi\right)\,
\{u_+,u_-\}\\
&=-\lambda A_{\phi,0} B_{\psi,0}+o(1).
\end{align*}
So, the Wronskian is bounded as $y\to+\infty$. 

Let us pick $\beta\in(0,1)$ and check that
\begin{equation}\label{w-plus-and-a}
w_+=2i\sin (2\pi 
p)\,a(p)\,a(1/2-p)\,e^{-\frac{i\pi(1/2-p)}h}+O_H(\lambda^\beta),\quad
|y|\le Y/2.
\end{equation}
Note that in view of~\eqref{wr-relation} this already implies 
representation~\eqref{w+} for $|y|\le Y/2$.
 
To prove~\eqref{w-plus-and-a}, we pick $\alpha\in(\beta,1)$ 
and consequently consider four cases.
In the case where $\frac{\alpha h}2\le \re p\le 1/2-\frac{\alpha h}2$
formula~\eqref{w-plus-and-a} follows from 
representations~\eqref{psi-plus:1},
\eqref{phi-plus:1} and estimate~\eqref{est:a0}. In the case where 
$\re p\le \min\{\frac{\alpha h}2,\,\frac12-\frac{\alpha h}2\}$, we use  
representations~\eqref{phi-plus:1}--\eqref{psi-plus:2} and 
estimates~\eqref{est:a0}--\eqref{est:a_0}, and get 
\begin{equation*}
w_+=2i\sin(2\pi p)\,a(p)\,a(1/2-p)\,e^{-\frac{i\pi(1/2-p)}h}
+O_H(\xi\lambda^\beta).
\end{equation*}
As we can assume that $\beta$ in the last formula is larger than
in~\eqref{w-plus-and-a}, we again obtain~\eqref{w-plus-and-a}.
The case where $\re p\ge \max\{\frac{\alpha h}2,\,\frac12-\frac{\alpha 
h}2\}$ 
is treated similarly (by means of~\eqref{psi-plus:1} 
and~\eqref{phi-plus:2}). 
Finally, let $\frac12-\frac{\alpha h}2\le\re p\le \frac{\alpha h}2$.
Note that in this case one has $h>\frac1{2\alpha}>1/2$. Now we get
\begin{equation}\label{5.29}
\begin{split}
w_+=2i&\sin(2\pi p)a(p)a(1/2-p)e^{-\frac{i\pi(1/2-p)}h}\\
&-2i\sin(2\pi  p)a(-p)a(p-1/2) e^{\frac{i\pi(1/2-p)}h}\,e^{-\frac{2\pi 
i(z+i\xi)}h}
+O(\xi\lambda^\beta).
\end{split}
\end{equation}
Let us estimate the second term in the right hand side of~\eqref{5.29}.
We denote it by $T$.   According to~\eqref{a(p,h)}, for $p\in P$, 
 we have
\begin{equation*}
  |T|\le C \lambda^{1/h}\big|\sin(2\pi p)\sigma_{\pi h}(-4\pi p-\pi h-\pi)
  \sigma_{\pi h}(4\pi (p-1/2)-\pi h-\pi)\big|.
\end{equation*}
By means of~\eqref{eq:sigma} and~\eqref{sigma:qp}, we get
\begin{equation*}
  |T|\le C \lambda^{1/h}\left|\frac{\sigma_{\pi h}(4\pi (h/2-p)-\pi h-\pi)
  \sigma_{\pi h}(4\pi p-\pi h-\pi)}{1-e^{-\frac{4\pi i p}h}}\right|.
\end{equation*}
As  $\frac12-\frac{\alpha h}2\le\re p\le \frac{\alpha h}2$ and
$\frac12<h<1$, we have 
\begin{gather*}
4\pi (h/2-\re p)\ge   2\pi h(1-\alpha)> \pi(1-\alpha), \quad
 4\pi \re p\ge 2\pi(1-\alpha h)\ge 2\pi(1-\alpha),\\
 \pi (1-\alpha)\le \re 2\pi p/h\le \alpha\pi.
\end{gather*}
The first two inequalities  and Corollary~\ref{theta:uni-rep:1} imply that
$|\sigma_{\pi h}(4\pi (h/2-p)-\pi h-\pi)\sigma_{\pi h}(4\pi p-\pi 
h-\pi)|\le C$.
The third inequality implies that $|1-e^{-4\pi i p/h}|\ge C$. These 
observations prove that $|T|\le C \lambda^{1/h}$.
As $1/2<h<1$, the term $T$ can be included in the error term. This 
completes 
the proof of~\eqref{w+} for $|y|\le Y/2$.
 
The expression $w_+-w_0$ is $h$-periodic and analytic in $z$.
Since it is bounded as $y\to+\infty$, representation~\eqref{w+} justified 
for $y=-Y/2$ and the maximum principle imply that 
$w_+-w_0=O(H\lambda^\beta)$ 
for all $z\in \C_+(-Y/2)$ uniformly in $p\in P$.
This completes the proof.
\end{proof}
\subsubsection{Riemann-Hilbert problem for constructing $\psi_D$}
Let $\psi_\pm$ and $\phi_\pm$ be the bases chosen in 
section~\ref{sss:bases-for-psiD}. The minimal solution 
$\psi=\psi_D$ admits the representations~\eqref{RP:1}-~\eqref{RP:2}. 
The coefficients $a_\pm,b_\pm\in\mathbb K_\pm$ satisfy the 
equation~\eqref{RP:4} with the matrix $G$ defined 
by~\eqref{RP:5}. To formulate the  Riemann-Hilbert problem for these 
coefficients,
we need to study their behavior at $\pm i\infty$.
\\
The coefficients $a_\pm$ and $b_\pm$ being $h$-periodic,
we shall regard them as functions of the variable
$\zeta=e^{2\pi i z/h}$. Let
\begin{equation*}
{\mathbb T}=\{\zeta\in\C\,:\,|\zeta|=1\},\quad 
{\mathbb B}_o=\{\zeta\in\C\,:\,|\zeta|\le 1\},\quad 
{\mathbb B}_\infty=\{\zeta\in\C\,: |\zeta|\ge 1\}\cup\{\infty\}.
\end{equation*} 
The function $V_+$ is  analytic in 
${\mathbb B}_o\setminus \{0\}$, and $V_-$ is analytic in 
${\mathbb B}_\infty\setminus\{\infty\}$. 
\\
Substituting~\eqref{psi-phi-plus} into~\eqref{RP:1}, we see 
that if  $a_+$ and $b_+$  are bounded  as $\zeta\to 0$
($y\to+\infty$), then $\psi_D$ admits representation~\eqref{psi:up}
with $A$ and $B$ staying bounded as $y\to+\infty$, and one has
\begin{equation}\label{1:RP:4}
A_D=A_{\phi,0}\,b_+(0),\quad B_D=B_{\psi,0}\,a_+(0)+B_{\phi,0}\,b_+(0).
\end{equation}
Substituting~\eqref{psi-phi-minus} into~\eqref{RP:2} and taking into 
account
Lemma~\ref{le:d-star-u}, we see 
that if, as $\zeta\to\infty$, the coefficient  $a_-$   is bounded
and $b_-(\zeta)\to 0$, then $\psi_D$ admits 
representation~\eqref{psi:down} 
with $C$  staying bounded and $D$ vanishing as $\zeta\to \infty$. One 
has
\begin{equation}\label{1:RP:6}
\begin{array}{c}
1=C_D=B_{\psi,0}^*\,a_-(\infty),\quad
D_D=e^{-\frac{2\pi\xi}h}\,A_{\psi,0}^*\, 
a_-(\infty)+A_{\phi,0}^*\,b_{-,1},\\{}\\
b_{-,1}=\lim_{\zeta\to\infty} (\zeta b_-(\zeta)).
\end{array}
\end{equation}
Let us collect the obtained information on the coefficients
$a_\pm$ and $b_\pm$. One has
\begin{gather}
\label{1:RP:b}
V_+(\zeta)=G(\zeta)\,V_-(\zeta),\quad\quad \zeta\in{\mathbb T},\\
\label{1:RP:a}
V_+\text{ is analytic in } {\mathbb B}_o,\quad 
V_-\text{ is analytic in } {\mathbb B}_\infty,
\quad \dsize 
V_-(\infty)=\frac1{B_{\psi,0}^*}\begin{pmatrix} 1\\ 0\end{pmatrix},
\end{gather}
Equation~\eqref{1:RP:b}  and conditions~\eqref{1:RP:a}
form a Riemann-Hilbert problem.
We shall see that, for sufficiently small $\lambda$,
this problem has a unique solution. Having solved this 
problem, we shall reconstruct the coefficients of the minimal 
solution $\psi_D$ by the formulae~\eqref{1:RP:4} and~\eqref{1:RP:6}.
\subsection{Matrix G}
In this section, we study  the matrix $G$.
\subsubsection{Functional relations}
The properties of the matrix $G$ we discuss here  immediately 
follow from~\eqref{RP:5},~\eqref{0:up} and~\eqref{0:down}. 
When describing these properties, we use the variable $z$,
assume that  $\lambda<e^{C/h}$ where $C$ is sufficiently large and 
that $p\in P$.
\\
As  $w_\pm$ are bounded away from zero in the domain
$(z,p)\in \{|y|\le Y/2\}\times P$, the matrix $G$ is analytic there.
Let
\begin{equation*}
g_{ij}(z,p)= w_+(z,p)\,G_{ij}(z,p),\quad i,j\in\{1,2\}.
\end{equation*}
As $\psi_-=\psi_+^*$ and $\phi_-=\phi_+^*$,~\eqref{RP:5} implies that
\begin{gather}\label{G:r-an}
g_{22}(z,p)= g_{11}^*(z,p), \ \ 
g_{12}(z,p)=-g_{12}^*(z,p), \ \ 
g_{21}(z,p)=-g_{21}^*(z,p).
\end{gather}
Furthermore, relation~\eqref{w=w} and the formula~\eqref{detG} imply that
\begin{equation*}
\det G^*=\det G^{-1},\qquad g_{11}\,g_{22}-g_{12}\,g_{21}=w_+^*w_+.
\end{equation*}
Finally, as 
$\phi_+(z,p)=e^{-\frac{i\pi(z+i\xi)}h}\,\psi_+(z-1/2,1/2-p)$, we get
\begin{equation}\label{G:invers}
g_{12}(z,p)=e^{\frac{2\pi\xi}h}\,g_{21}(z-1/2,1/2-p).
\end{equation}
\subsubsection{The asymptotics of  $G$ for $0<p<1/2$}
Here we prove
\begin{Pro}\label{first:G:ne0} Let $0<\beta<1/2$.
There is a constant $C$ such that if 
$\lambda<e^{-C/h}$, then
for $p\in P$ such that $h/4\le\re p\le 1/2-h/4$,
and for $|y|\le \frac{Y}2$, one has
\begin{equation*}\textstyle
G=\frac1{w_0}\begin{pmatrix}
\delta & e^{\frac{4\pi\xi p}h}(2iF(1/2-p)\sin(2\pi p)+\delta)\\
e^{-\frac{4\pi\xi p}h}(2iF(p)\sin(2\pi  p)+\delta)& \delta
\end{pmatrix},
\end{equation*}
where $w_0$ is defined in~\eqref{w+}, $\delta$ denotes $O(\lambda^\beta 
H)$, and $F$ is  the meromorphic function  such that
\begin{equation}\label{F_1,F_2}
F(p)=\left|\sigma_{\pi h}( 4\pi p-\pi-\pi h)\right|^2, \quad p \in \R.
\end{equation}
Moreover, one has $|F(p)|\le H$ for $p\in P$ such that $h/4\le \re p$.
\end{Pro}
\noindent 
So in the case of this proposition, for sufficiently small 
$\lambda$  the matrix $G$ appears to be close to a constant one.
\begin{proof}
Below, we assume that all the hypotheses of the proposition are satisfied.
First, we estimate the Wronskian $g_{11}=\{\psi_-,\phi_+\}$. 
Using~\eqref{psi-plus:1}--~\eqref{phi-plus:1} and the definition
$\psi_-$ from~\eqref{0:down}, we get
\begin{equation*}
g_{11}=
\left\{e^{-\frac{2\pi ip(z-i\xi)}h}(a^*(p) +\delta),\ 
e^{-\frac{2\pi ip(z+i\xi)}h}(e^{-\frac{i\pi(1/2-p)}h}a(1/2-p) 
+\delta)\right\}.
\end{equation*}
Obviously, the leading term  equals zero, and using estimate~\eqref{est:a0}
we prove that $g_{11}=O(\lambda^\beta\,H)$. 
By means of the first relation from~\eqref{G:r-an} 
we also see that $g_{22}=O(\lambda^\beta\,H)$.
As $G_{jj}=(w_+)^{-1}g_{jj}$, $j=1,2$, and in view of Lemma~\ref{le:w+}, 
we 
get 
the announced estimate for the diagonal elements of the matrix $G$.
\\
Now consider $g_{12}=\{\phi_-(z),\,\phi_+(z)\}$.
Using~\eqref{phi-plus:1},~\eqref{0:down} and~\eqref{est:a0}, we get
\begin{align*}
g_{12}&=
\{e^{\frac{2\pi ip(z-i\xi)}h}(a^*(1/2-p) +\delta),\;
e^{-\frac{2\pi ip(z+i\xi)}h}(a(1/2-p) +\delta) \}=\\
&=e^{\frac{4\pi p\xi}h}\Big(2i\sin (2\pi p)\, |a(1/2-p)|^2+
O(\lambda^\beta H)\Big).
\end{align*}
Let us note that, to get this formula, instead of $h/4\le p\le 1/2-h/4$,
we have only to assume that  $\re p\le 1/2-h/4$. 
The definition of $a$, formula~\eqref{a(p,h)}, implies that
\begin{equation}
  \label{eq:F-a}
  F(p)=|a(p)|^2.
\end{equation}
Therefore,
\begin{equation}\label{g12}
g_{12}=e^{\frac{4\pi p\xi}h}
\Big(2i\,F(1/2-p)\,\sin(2\pi  p)+\delta\Big),
\end{equation}
and also, in view of~\eqref{est:a0} \ $|F(p)|\le H$ for $p\in P$ such 
that $\re p\ge h/4$. This estimate and representations~\eqref{g12} and
\eqref{w+} imply
the formula for $G_{12}$ announced in the proposition.
We note that it is valid for all $p\in P$
such that  $\re p\le \pi-h/4$.
\\
Formula~\eqref{g12} and relation~\eqref{G:invers} imply the formula
for $G_{21}$ announced in the proposition.
It is valid for all $p\in P$ such that 
$h/4\le \re p$.
 \\
We have checked all the statements of the proposition.
\end{proof}
To use Proposition~\ref{first:G:ne0}, we need 
\begin{lemma}\label{F1F2:prop} One has 
\begin{equation}
  \label{eq:F}
 F(p+1/2)=4\sin^2\frac{2\pi p}h\;F(p),
\end{equation}
\begin{equation}\label{F1F2:rel}
4\sin^2(2\pi p)\;F(1/2-p)\,F(p)=1.
\end{equation}
\end{lemma}
\begin{proof}
Formula~\eqref{eq:F} follows from~\eqref{sigma:qp}, 
and~\eqref{F1F2:rel} follows from~\eqref{sigma:sym} and~\eqref{eq:sigma}.
\end{proof}
\subsection{Proof of  Theorem~\ref{th:0:ne0}}
Here, we compute the coefficients $s$ and $t$ of the monodromy matrix
in the case where $p\in P$ is bounded away from $0$  and $1/2$.  
Therefore, first, we solve the Riemann-Hilbert  
problem~\eqref{1:RP:b}--\eqref{1:RP:a} 
to find the asymptotics of $a_+(0)$, $b_+(0)$, $a_-(\infty)$ and
$\lim_{\zeta\to\infty}\zeta  b_-(\zeta)$. Then, by means 
of formulae~\eqref{1:RP:4} and~\eqref{1:RP:6}, we compute
the coefficients $A_D$, $B_D$ and $D_D$
of the minimal entire solution $\psi_D$. Finally, using
formulae~\eqref{first:ts}, we compute $s$ and $t$.  
\subsubsection{Solving the Riemann-Hilbert problem}
The leading term of the asymptotics of  $G$ being independent 
of $z$, the analysis of the Riemann-Hilbert problem is elementary.  
Assume that $\lambda$, $z$ and $p$ satisfy assumptions of 
Proposition~\ref{first:G:ne0}.
Let
\begin{equation*}
G_0=\frac{2i\sin(2\pi p)}{w_0}\,\begin{pmatrix}
0& F(1/2-p)\\
F(p) & 0\end{pmatrix},\quad\quad
T=\begin{pmatrix} \tau &0\\ 0 & \tau^{-1}\end{pmatrix},\quad
\tau=e^{\frac{2\pi \xi p}h}.
\end{equation*}
Relation~\eqref{F1F2:rel} implies that $\det G_0=1/w_0^2$.
In view of Proposition~\ref{first:G:ne0}, we have
\begin{equation}\label{G:fact}
G=T\,G_0\,(I+\Delta)\,T^{-1},\quad
\Delta=O(\lambda^\beta H).
\end{equation}
The  term $\Delta$ is analytic in 
$(z,p)\in \{|y|\le \frac{Y}2\}\times 
\{p\in P\,:\,h/4\le \re p\le 1/2-h/4\}$.

Now, we pass to the variable $\zeta=e^{2\pi i z/h}$. 
Let $\|.\|$ be a  matrix norm. Pick $\alpha\in(0,1)$. 
For matrix functions on ${\mathbb T}=\{|\zeta|=1\}$ 
denote by $\|.\|_\alpha$ the  standard H\"older
norm defined in terms of $\|.\|$.  One has 
\begin{lemma}\label{le:RP:standard} Let $\Delta$ be a matrix-valued 
function on ${\mathbb T}$.  
If $\|\Delta\|_\alpha$ is sufficiently small,
then there exist unique matrix functions
$W_\pm$ such that
\begin{gather}
\nonumber
W_+ \text{ is analytic in }  {\mathbb B}_o,\\
\nonumber
W_-\text{ is analytic in }  {\mathbb B}_\infty ,\quad  W_-(\infty)=I,\\
\label{I+Delta:fact}
W_+(\zeta)=(I+\Delta(\zeta))\,W_-(\zeta),\quad \zeta\in{\mathbb T}.
\end{gather}
These functions satisfy the estimates:
\begin{equation}\label{Wpm:est}
\|W_+(\zeta)-I\|\le C\|\Delta\|_\alpha,\quad |\zeta|\le 1;\quad\quad
\|W_-(\zeta)-I\|\le \frac{C\|\Delta\|_\alpha}{|\zeta|},\quad |\zeta|\ge 1.
\end{equation}
\end{lemma}
\begin{proof} The Lemma follows from standard results of the 
theory of singular integral operators, see, e.g.,~\cite{Mu:53}. So, we 
describe 
the proof omitting standard details.  
\\
First, in  $H_\alpha({\mathbb T})$, the space of 
matrix-valued H\"older functions on $\mathbb T$, one
constructs a solution to the equation 
\begin{equation}\label{eq:W}
W_-=I+S(\,\Delta W_-),
\end{equation}
where, for $f\in H_\alpha(\mathbb T)$, 
\begin{equation*} 
S(f)(\zeta)=-\frac12 \,f(\zeta)+\frac1{2\pi i}\;v.p.\int_{\mathbb T}
\frac{f(\zeta')\,d\zeta'}{\zeta'-\zeta},\qquad \zeta\in{\mathbb T},
\end{equation*}
and the orientation of  $\mathbb T$ is positive.
As  $S$ is a bounded operator in $H_\alpha(\mathbb T)$, and as for 
$f,g\in H_\alpha(\mathbb T)$ one has $\|fg\|_\alpha\le 
\|f\|_\alpha\|g\|_\alpha$,
equation~\eqref{eq:W} has a unique solution provided $\|\Delta\|_\alpha$ 
is sufficiently small.
\\
One defines $W_+(\zeta)$ for  $\zeta\in {\mathbb B}_o$ and 
$W_-(\zeta)$ for  $\zeta\in {\mathbb B}_\infty$ by the formulas
$$W_\pm(\zeta)=I+ \frac1{2\pi i} \int_{\mathbb T}
\frac{\Delta(\zeta')W_-(\zeta')}{\zeta'-\zeta}\,
d\zeta',$$ 
and checks that these two function have all the properties 
described in Lemma~\ref{le:RP:standard}.
We omit further details.
\end{proof}
In our case, in the ring $ e^{-\pi Y/h}\le|\zeta|\le e^{\pi Y/h}$, \
$\Delta$ is analytic  and satisfies the estimate
$\|\Delta(\zeta)\|\le H\lambda^\beta$. Therefore, for any fixed 
$\alpha\in(0,1)$, one has
$\|\Delta\|_\alpha\le C(\alpha)\,H\lambda^\beta$. So, there is a  $C>0$ 
such that if $\lambda\le e^{-C/h}$, then $\Delta$ satisfies
the assumptions of Lemma~\ref{le:RP:standard}. For this $\Delta$, 
we construct $W_\pm$ by Lemma~\ref{le:RP:standard}. The vector-valued 
functions defined by the formulas
\begin{equation}\label{Vpm:answer}
V_+(\zeta)= TG_0W_+(\zeta)T^{-1}e,\quad
V_-(\zeta)= T W_-(\zeta)T^{-1} e,\quad
e=\frac1{B_{\psi,0}^*}\,\begin{pmatrix} 1\\0\end{pmatrix},
\end{equation}
are a solution of the Riemann-Hilbert 
problem~\eqref{1:RP:a}--\eqref{1:RP:b}.
Indeed, $V_+$ is analytic in $\mathbb B_0$, \ $V_-$ is analytic in 
$\mathbb B_\infty$,  and $V_-(\infty)= T W_-(\infty) T^{-1} e=e$.
Moreover, by~\eqref{G:fact} and~\eqref{I+Delta:fact},
for $|\zeta|=1$,
\begin{align*}
 V_+(\zeta)&=T G_0 W_+(\zeta)T^{-1} e=
T G_0 (I+\Delta(\zeta)) W_-(\zeta)T^{-1} e=\\
 &=T G_0 (I+\Delta(\zeta))T^{-1} V_-(\zeta)=G(\zeta) V_-(\zeta).
\end{align*}
We compute the coefficients $s$ and $t$ of the monodromy matrix
using~\eqref{1:RP:4}--\eqref{1:RP:6}, where $a_\pm$ and $b_\pm$ are the 
first and the second components of the vectors $V_\pm(0)$. 
Formulas~\eqref{Vpm:answer},  formula for $B_{\psi,0}$ 
from~\eqref{A0,B0} and the  estimate for $W_--I$ from~\eqref{Wpm:est} 
imply that
\begin{equation}
  \label{eq:a-b-}
  a_-(\infty)=\frac1{B_{\psi,0}^*},\qquad  
  b_{-,1}=\zeta b_-(\zeta)\Big|_{\zeta=\infty}=O(\lambda^\beta e^{-4\pi 
\xi 
p/h+\pi\xi}H).
\end{equation}
Using also formula~\eqref{w+} for $w_0$ and the estimate 
$\varkappa_0=O(H\lambda)$ (following from the third point of
Proposition~\ref{psi0}), we get 
\begin{equation}\label{a+,b+}
a_+(0)=O\left(e^{\pi\xi}\lambda^\beta H\right),\quad
b_+(0)=\frac{e^{-\frac{4\pi\xi p}h}}{w_0B_{\psi,0}^*}\left(2i\sin(2\pi 
p)F(p)+
O(\lambda^\beta H)\right).
\end{equation}
\subsubsection{Asymptotics of the coefficients $s$ and $t$}
%
Using~\eqref{first:ts}, \eqref{1:RP:4}, \eqref{1:RP:6},
we get
\begin{equation}\label{ts:ne0}
t=-\lambda_1 A_{\phi,0}\,b_+(0),\qquad
s=-
\frac{A_{\psi,0}^*\,a_-(\infty)+\lambda_1A_{\phi,0}^*\,b_{-,1}}
{B_{\psi,0}\,a_+(0)+B_{\phi,0}\, b_+(0)}.
\end{equation}
Using \eqref{a+,b+}--\eqref{eq:a-b-},
estimates for $a_\pm$, $b_+$ and $b_{-,1}$, the estimate  
$\varkappa_0=O(H\lambda)$, and~\eqref{A0,B0} and~\eqref{A0,B0-hats},  
formulas for $A_{\psi,0}$,
$A_{\phi,0}$, $B_{\psi,0}$ and $B_{\phi,0}$,
we obtain
\begin{equation*}
t=2ie^{\frac{4\pi(1/2-p)\xi}h}\,\frac{F(p)\sin(2\pi p)+O(\lambda^\beta H)}
{1+O(\lambda H)},
\qquad
s=\frac1{2i}\,\,\,
\frac{e^{\frac{4\pi p\xi}h}\, e^{\frac{2\pi ip}h}\,
 +O(\lambda_1\lambda^\beta H))}
{F(p)\sin (2\pi p) +O(\lambda^\beta H)}.
\end{equation*}
Let us  simplify these formulae. For $h/4\le \re p\le 1/2-h/4$, one has
$|F(p)|\le H$ (see Proposition~\ref{first:G:ne0}). By this estimate 
and~\eqref{F1F2:rel},
one also has
$$\left|2F(p)\sin (2\pi p)\right|^{-1}=|2\sin(2\pi p)\,F(1/2-p)|\le H.$$ 
Using these observations, we get
\begin{equation}\label{ts:temp}
t=2ie^{\frac{4\pi (1/2-p)\xi}h}\,F(p)\sin(2\pi p)(1+O(\lambda^\beta H)),
\ \
s=\frac{e^{\frac{4\pi p\xi}h}e^{\frac{2\pi ip}h}(1+ O(\lambda^\beta H))}
{2i\,F(p)\,\sin(2\pi  p)}.
\end{equation}
Finally, by means of~\eqref{eq:sigma},
we check that, for $p\in\R$,
$$
2i\sin(2\pi p) F(p)=2i\sin (2\pi p)
\left|\sigma_{\pi h}( 4\pi p-\pi-\pi h)\right|^2=
-\frac{\left|\sigma( 4\pi p-\pi+\pi h)\right|^2}
{2i\sin (2\pi p)}.
$$
This relation and~\eqref{ts:temp} imply the statement of  
Theorem~\ref{th:0:ne0}.
\qed
\subsection{Asymptotics of  {\it s} and {\it t} for {\it p} close to 
zero}\label{ss:s-and-t-for-p-close-to-0}
Here we prove
\begin {theorem} \label{th:0:0} Pick $\beta\in(0,1)$.
There is a positive constant $C$ such that if 
$\lambda<e^{-C/h}$, then for  $ p \in P $ satisfying the condition $ 0 \le 
{\rm Re} \, p \le h / 4 $, one has
\begin {equation}\label{eq:st:0}
t = ie ^ {\frac {4\pi (1/2-p) \xi} {h}} \, F_0(p) \,
\left [\frac {1-e ^ {\frac {8\pi p \xi} { h}} \, \rho (p)} {2 \sin (2\pi 
p)} + \delta \right], \qquad
s = \frac {-i\,e ^ {\frac {4\pi p \xi} {h} + \frac {2\pi ip} h}}
{F_0 (p) \left [
\frac {1-e ^ {\frac {8\pi p (\xi + i/2)} {h}} \, \rho (p)} {2\sin(2\pi  
p)} 
+ \delta \right]}\;,
\end {equation}
where for $p\in \R$, one has $ \rho (p) =\frac{F_0(-p)}{F_0(p)}=
\left|\frac{\sigma(-4\pi p-\pi+\pi h)}{\sigma(4\pi p-\pi+\pi h)}\right|^2$,
and $ \delta $ denotes  $ O (\lambda ^ \beta H) $.
\end {theorem}
The proof of Theorem~\ref{th:0:0} is similar to one of 
Theorem~\ref{th:0:ne0}. So, when proving Theorem~\ref{th:0:0}, we omit
elementary details.
\subsubsection{Asymptotics of the matrix $G$}
We have
\begin{Pro}\label{first:G:0}  Pick $0<\beta<1$.
There is a positive constant $C$ such that if 
$\lambda<e^{-C/h}$, then, for $p\in P$,  
$0\le\re p\le h/4$, and for $|y|\le Y/2$, 
one has
\begin{equation*}
G=\frac1{w_0}\,\begin{pmatrix}
e^{\frac{4\pi p\xi}h}\,F_d(p)+\delta &
e^{\frac{4\pi p\xi}h}\left(2iF(1/2-p)\,\sin(2\pi p)+
\delta \right)  \\
e^{-\frac{4\pi p\xi}h}\,\left(F_a(p)+\delta \right)&
e^{\frac{4\pi p\xi}h}\,F_d^*(p)+\delta
\end{pmatrix},
\end{equation*}
where $\delta$ denotes $O(\lambda^\beta H)$,
\begin{gather}\label{F_d}
F_d(p)=-4i e^{\frac{2\pi i p}h-\frac{3i\pi}{4h}}
\sin\frac{2\pi p}h\sin(2\pi p)\;F(-p),\\
\label{F_a} 
F_a(p)=2i\sin(2\pi p) (F(p)-e^{\frac{8\pi p\xi}h}F(-p)),
\end{gather}
and  $F$, $F_d$ and $F_a$ satisfy the estimates
\begin{equation}\label{F:1,3,4:est}
p^2F(p)=O(H),\quad F_d(p)=O(H),\quad F_a(p)=O(\xi H).
\end{equation}
\end{Pro}
\begin{proof} 
First, we note that  $G_{12}$ was already computed when proving 
Proposition~\ref{first:G:ne0}: when computing  it we assumed 
that $0\le \re p\le 1/2-h/4$, and now, as $0<h<1$, one has
$0\le \re p\le h/4\le 1/2-h/4$.
The formula for $G_{22}=g_{22}/w_+$ follows from one for 
$G_{11}=g_{11}/w_+$,~\eqref{G:r-an} and~\eqref{w+}.
\\ 
So we only have to compute 
$G_{11}=g_{11}/w_+$ and $G_{21}=g_{21}/w_+$.
Using formulas~\eqref{psi-plus:2},~\eqref{phi-plus:1} 
and~\eqref{0:down}, and estimates~\eqref{est:a0} and~\eqref{est:a_0}, we 
get
\begin{gather}
\label{g11:temp0}
g_{11}=e^{\frac{4\pi p\xi}h}\;
2i\sin(2\pi p)\,e^{-\frac{i\pi(1/2-p)}h}a^*(-p)a(1/2-p)
+O((1+|\xi|) \lambda^\beta H),
\\
\nonumber
g_{21}=e^{-\frac{4\pi p\xi}h}\left(2i\sin(2\pi 
p)\,(|a(p)|^2-e^{\frac{8\pi p\xi}h}|a(-p)|^2)
+O(\lambda^\beta(1+|\xi|) H)\right).
\end{gather}
By means of~\eqref{a(p,h)} and~\eqref{sigma:qp}, we 
transform~\eqref{g11:temp0} to the form
\begin{equation}\label{g11:temp1}
g_{11}=e^{\frac{4\pi p\xi}h}\,F_d(p)+O(H\lambda^{\beta}(1+|\xi|)).
\end{equation}
Furthermore, the definition of $F_a$ and relation~\eqref{eq:F-a} allows 
to get the formula
\begin{equation}\label{g21:temp1}
g_{21}=e^{-\frac{4\pi 
p\xi}h}\left(F_a(p)+O(\lambda^\beta(1+|\xi|) H)\right).
\end{equation}
Estimates~\eqref{F:1,3,4:est} follow from~\eqref{est:a0} 
and~\eqref{est:a_0}.
\\
As we can assume that $\beta$ in formulas~\eqref{g11:temp1} 
and~\eqref{g21:temp1} is larger than in Proposition~\ref{first:G:0}, 
these formulas,~\eqref{w+} and estimates~\eqref{F:1,3,4:est}
imply the representations for $G_{11}$ and $G_{12}$ from 
Proposition~\ref{first:G:0}. This completes its proof.
\end{proof}
\subsubsection{Solving the Riemann-Hilbert problem}
Let
\begin{equation}\label{G0:0}
G_0=\frac1{w_0}\,\begin{pmatrix}
e^{\frac{4\pi p\xi}h}\,F_d(p)& 2iF(1/2-p)\,\sin (2\pi p)  \\
F_a(p) & e^{\frac{4\pi p\xi}h}\,F_d^*(p).
\end{pmatrix}
\end{equation}
First, we prove that again $\det(G_0)=1/w_0^2$. This follows 
from~\eqref{G0:0}, the definitions of $F_a$ and $F_d$, see~\eqref{F_a} 
and~\eqref{F_d}, and Lemma~\ref{F1F2:prop}.
\\
Then, we prove that $G$  admits again
representation~\eqref{G:fact} with the new $G_0$. For this, using 
the estimates for $F_d$ and $F_a$  from~\eqref{F:1,3,4:est} and the 
estimate for $F$ from Proposition~\ref{first:G:0},  we check that 
$G_0^{-1}T^{-1}GT=I+O(\lambda^{\beta} H)$. 
\\
Having obtained~\eqref{G:fact}, we proceed as in the case where $h/4\le 
\re p$ and 
obtain
\begin{gather}\label{a+,b+:0}
a_+(0)=\frac{e^{\frac{4\pi p\xi}h}F_d(p)+
O(\lambda^{\beta}H)}{w_0B_{\psi,0}^*},\quad
b_+(0)=\frac{e^{-\frac{4\pi p\xi}h}(F_a(p)+
O(\lambda^{\beta}H))}{w_0B_{\psi,0}^*},\\
\label{a-,b-:0}
 a_-(\infty)=\frac1{B_{\psi,0}^*},\qquad 
b_{-,1}=O(e^{-\frac{4\pi p\xi}h+\pi\xi}\lambda^{\beta}H).
\end{gather}
\subsubsection{Asymptotics of $s$ and $t$}
Computing the coefficients $s$ and $t$ by means of  
formulas~\eqref{ts:ne0},~\eqref{a+,b+:0}, and~\eqref{a-,b-:0}, and  
estimates~\eqref{F:1,3,4:est}, we get
\begin{equation*}
t=e^{\frac{4\pi(1/2-p)\xi}h}\,\left(F_a(p)+O(\lambda^{\beta}H)\right),\quad
s=\frac{e^{\frac{2\pi i p}h+\frac{4\pi p\xi}h}}
{F_a(p)+ ie^{\frac{3\pi i}{4h}+\frac{8\pi p\xi}h}F_d(p)+O(\lambda^\beta 
H)}.
\end{equation*}
Let us prove~\eqref{eq:st:0} for $t$. In view 
of~\eqref{F_1,F_2},~\eqref{eq:sigma},~\eqref{def:F0} and~\eqref{F_a}, 
one has
\begin{gather*}
4\sin^2(2\pi p) F(p)=F_0(p),\ 
\frac{F(-p)}{F(p)}=\frac{F_0(-p)}{F_0(p)}=\rho(p),\ 
F_a(p)=iF_0(p)\frac{1-\rho(p)e^{\frac{8\pi p\xi}h}}{2\sin(2\pi p)}.
\end{gather*}
As $1/2-\re p\ge 1/2-h/4\ge h/4$, formula~\eqref{F1F2:rel} and the 
estimate for $F$ from Proposition~\ref{first:G:ne0}
imply that  $1/F_0(p)=O(H)$. This
and the last formulas for $t$ and $F_a$ imply the formula for $t$ 
from~\eqref{eq:st:0}.

The last formula for $s$ and the relation
\begin{equation*}
F_a(p)+ie^{\frac{3\pi i}{4h}+\frac{8\pi 
p\xi}h}F_d(p)=iF_0(p)\frac{1-e^{\frac{8\pi 
p}h(\xi+\frac{i}2)}\rho(p)}{2\sin(2\pi p)}.
\end{equation*}
imply the formula  for $s$ from~\eqref{eq:st:0}.
The proof of Theorem~\ref{th:0:0} is complete. 
\qed
\section{A trigonometric analog of the Euler 
Gamma-function}\label{sec:sigma}
\noindent 
Here, following mostly~\cite{Bu-Fe:01, Fe:2016},  we discuss  
equation~\eqref{eq:sigma} with $a\in(0,\pi)$.
\subsection{Definition and elementary properties}
\subsubsection{}
In $S_0=\{|x|<\pi+a\}$, equation~\eqref{eq:sigma} has a unique 
solution  $\sigma_a$ that is analytic, nonvanishing, and 
having the representations 
\begin{gather} \label{sigma:down}
\sigma_a(z)=1+o(e^{-\alpha|y|}), \quad y\to-\infty,
\\
\nonumber
\sigma_a(z)= e^{\dsize
-\frac{iz^2}{4a}+\frac{i\pi^2}{12a}+\frac{ia}{12}+o\,(e^{-\alpha|y|})}, 
\quad 
y\to+\infty,
\end{gather}
for any fixed $\alpha\in(0,1)$.
These representations are uniform in $x$. If $a$ is bounded away from 
zero, they are also uniform in $a$. The function $\sigma_a$ is 
continuous in $a$.
\subsubsection{}\label{sigma:poles,zeros}
Using equation~\eqref{eq:sigma}, one can analytically continue  $\sigma_a$
to a meromorphic function. Its poles are located at the points
\begin{equation*}
-\pi-a -2\pi l-2ak,\quad 
l,k=0,\,1,\,2,\,\dots,
\end{equation*} 
and its zeros are at the points
\begin{equation*}
  \pi+a +2\pi l+2ak,\quad l,k=0,\,1,\,2,\,\dots.
\end{equation*}
Its zero at $\pi+a$ and its pole at  $-\pi-a$ are simple.
\subsubsection{} The function $\sigma_a$ satisfies the following relations:
\begin{gather}\label{sigma:qp}
\sigma_a(z+\pi)=(1+e^{\dsize -\frac{i\pi}a\,z})\,\sigma_a(z-\pi),\\
\label{sigma:sym}
 \sigma_a(-z)\,
=e^{\dsize -\frac{i}{4a}\,z^2+\frac{i\pi^2}{12a}+
\frac{ia}{12}}\,\frac1{\sigma_a\,(z)},\\
\label{sigma:real-an}
\overline{\sigma_a(\overline{z})}=e^{\dsize 
\frac{i}{4a}\,z^2-\frac{i\pi^2}{12a}-
\frac{ia}{12}}\,\sigma_a\,(z).
\end{gather}
\subsubsection{} One also has
\begin{equation}\label{sigma:res}
 {\rm Res}_{z=-\pi-a} \sigma_a=-i\sigma_a(-\pi+a)=
\sqrt{\frac a\pi}\,e^{\dsize -\frac{i\pi^2}{12a}-\frac{i\pi}4
-\frac{ia}{12}}.
\end{equation}
\subsection{Quasiclassical asymptotics}
Here, we discuss $\sigma_a$ for small $a$.
\\
Thanks to~\eqref{sigma:sym}, it suffices to 
study this function for $y\le 0$.
\\
Below we use  the branch of the function 
$z\mapsto\ln\,(1+e^{\dsize-iz})$ analytic in $\C_-$, the lower halfplane, 
and satisfying the condition
\begin{equation*}
\ln\,(1+e^{\dsize-iz})\to 0,\quad y\to-\infty.
\end{equation*}
We set
\begin{equation*}
L(z)=\int_{-i\infty}^{z}\ln\,(1+e^{\dsize-iz'})\,dz',
\end{equation*}
where we integrate in $\C_-$, say, along the line $\re z'=\re z$.
\begin{theorem}\label{theta:uni-rep:1}\cite{Fe:2016}  Pick $0<\delta<\pi$.
In $\C_-\cup\mathbb R$ outside the  $\delta$-neighborhood 
of the half-lines $z\ge \pi$ and $z\le -\pi$, for sufficiently small $a$, \
$\sigma_a$ admits the representation
\begin{equation}\label{Pi}
\sigma_a(z)=\exp\left(\frac1{2a} L(z)+ 
O\left(\,a\,(1+|x|)e^{-|y|}\,\right)\right).
\end{equation}
\end{theorem}
Let us discuss the behavior of  $\sigma_a$ in a neighborhood of the point 
$-\pi$.  
\begin{theorem}\label{theta:uni-rep:2}\cite{Fe:2016}
Let $0<\delta<2\pi$. For $t$ in the $\delta$-neighborhood of zero, one has
\begin{equation}\label{theta:delta}
\sigma_a(-\pi+t)=\frac{e^{\frac{\ln(2a)}{2a} t}}{\sqrt{2\pi}}
\Gamma\left(\frac{t+a}{2a}\right)\,
e^{\frac1{2a}\,\int_0^t\tilde l(\zeta) d\zeta-\frac{i\pi^2}{6a}+
O\left(a\right)},
\quad
\tilde l(\zeta)=\ln\frac{1-e^{-i\zeta}}{\zeta},
\end{equation}
where $\tilde l$ is analytic, and $\tilde 
l(0)=i\pi/2$.
The error term in~\eqref{theta:delta} is analytic in $t$.
\end{theorem}
\subsection{Uniform estimates}
Fix $\delta\in(0,\pi)$ and $\kappa\in(0,1)$. 
\begin{Cor}\label{cor:sigma:2}
Outside the $\delta$-neighborhood of the ray $z<-\pi$, one has
\begin{gather}\label{rough-est:down}
\sigma_a(z)=e^{O\left(a^{-1}e^{-\kappa |y|}(1+|x|)\right)},\quad y\le 0,\\
\label{rough-est:up}
\sigma_a\,(z)=e^{ -\frac{iz^2}{4a}+\frac{i\pi^2}{12a}+
\frac{ia}{12}+O\left(a^{-1}e^{-\kappa |y|}(1+|x|)\right)   },
\quad y\ge 0.
\end{gather}
\end{Cor}
\begin{proof} Estimate~\eqref{rough-est:up} follows 
from~\eqref{rough-est:down}
and~\eqref{sigma:real-an}. Let us prove~\eqref{rough-est:down}. Assume 
that $y\le 0$.
First, we note that \eqref{rough-est:down} is valid for $|x|\le a$. 
Indeed, let $a_0>0$ be so small that~\eqref{Pi} holds for all $0<a<a_0$. 
For 
these $a$, formula \eqref{rough-est:down} follows directly 
from~\eqref{Pi}.
For $a_0\le a\le 2\pi$, it follows from~\eqref{sigma:down} that is valid 
and uniform in $a$ and in $x$ if $|x|\le a$.
\\
Now, we assume that  $z$ is outside the $\delta$-neighborhood of 
the ray $z<-\pi$.
We choose  $n\in\Z$ so that $|x-2an|\le a$. 
By~\eqref{eq:sigma}
\begin{equation*}
\sigma_a(z)=\prod_{j=1}^{|n|}\left(1+e^{\dsize-i(z\mp(2j-1)a)}\right)^{\pm1
}
\sigma_a(z-2na) \quad \text{if} \quad\pm n\ge1.
\end{equation*}
As $|n|\le \frac{|x|}{2a}+\frac12$, we get
\begin{equation}\label{est:aux:1}
\prod_{j=1}^{|n|}\left(1+e^{-i(z\mp(2j-1)a)}\right)^{\pm 1}
=e^{O(|n|\,e^{-|y|})}
=e^{O\left(xe^{-|y|}/a\right)}.
\end{equation}
Formula~\eqref{rough-est:down}  valid for $|x|\le a$ and 
\eqref{est:aux:1} imply~\eqref{rough-est:down} for all $z$ we consider.
\end{proof}
Fix positive $c_1$, $c_2$ and  $\delta<2\pi$. 
\begin{Cor}\label{cor:sigma:3}
Let $|z+\pi|\le \delta$ and   $x\ge -\pi-a+c_1a-c_2|y|$.
Then, $|\sigma_a(z)|\le e^{C/a}$.
\end{Cor}
\begin{proof} \   {\it 1)} \ Let $u=\frac{z+\pi+a}{2a}$. 
Then, under the hypothesis of the corollary, 
\begin{equation*}
\re u\ge \frac{c_1}2-c_2|\im u|.
\end{equation*}
Let $D$ be the domain defined by this inequality in
the complex plane of $u$.
\\
{\it 2)} \ For $u\in D$,  we set 
\begin{equation*}
Y(u)=\sqrt{\frac{u}{2\pi}}\,e^{-u\,\left(\ln u- 1\right)} \Gamma(u)
\end{equation*}
where the branches of $\ln{.} $ and $\sqrt{.}$ are analytic in $\mathbb 
C\setminus\{z\le 0\}$ and 
such that $\ln 1=0$ and $\sqrt{1}=1$. 
The function $Y$  is bounded in $D$.

\noindent{\it 3.} \ By~\eqref{theta:delta} and the previous steps, 
under the hypothesis of the Corollary,
\begin{equation}\label{est:sigma}
\begin{split}\dsize
|\sigma_a(z)|&\le C\,\left|\exp\left(u(\ln u-1)-\frac12\ln 
u+\frac{\ln(2a)}{2a} (z+\pi)+
\frac1{2a}\,\int_0^{z+\pi}\tilde l(\zeta) d\zeta\right)\right|\\
\dsize
&\le e^{C/a} 
\left|\exp\left(u\ln u+\frac{\ln(2a)}{2a}(z+\pi)\right)\right|\le 
e^{C/a}\left|e^{u\ln(2au)}\right|,
\end{split}
\end{equation}
as   $\tilde l$ is analytic in the 
$\delta$-neighborhood of zero and as $|u|\ge C$ in $D$ .

\noindent{\it 4.} As $|2au|=|z+\pi+a|\le C$, then one also has  
$|2au\ln(2au)|\le C$, and~\eqref{est:sigma} implies that  
$|\sigma_a(z)|\le e^{C/a}$. 
\end{proof}
\def\cprime{$'$}

\end{document}